\documentclass[a4paper,12pt]{article}

\usepackage{amsmath,amsfonts,amssymb,amsthm,amscd}

\newtheorem{thm}{Theorem}
\newtheorem{lem}[thm]{Lemma}
\newtheorem{prop}[thm]{Proposition}
\newtheorem{defn}[thm]{Definition}
\newtheorem{cor}[thm]{Corollary}
\newtheorem{rem}[thm]{Remark}
\newtheorem{notation}[thm]{Notation}
\newtheorem{exa}[thm]{Example}

\title{$(TE)$-structures  over  the irreducible $2$-dimensional globally nilpotent $F$-manifold germ  }

\date{}

\author{Liana David and Claus Hertling}

\begin{document}

\maketitle

{\bf Abstract:} We find formal and holomorphic normal forms
for a class of meromorphic connections (the so-called $(TE)$-structures) over the 
irreducible $2$-dimensional globally 
nilpotent $F$-manifold 
germ $\mathcal N_{2}$.
We find normal forms for Euler fields on $\mathcal N_{2}$ and we characterize
the  Euler fields on $\mathcal N_{2}$ 
which are  induced by a $(TE)$-structure.\\



\section{Introduction}

An important topic in modern mathematics is the theory of Frobenius manifolds. 
Originally introduced in \cite{dubrovin} by B.  Dubrovin as a geometrization of the Witten-Dijkgraaf-Verlinde-Verlinde
(WDVV)-equations, they received much
interest from the mathematical community owing to their relation with 
various research fields, like 
singularity theory, integrable systems, Gromov-Witten invariants  and the theory of meromorphic connections. A Frobenius manifold
is a complex manifold $M$ together with a commutative, associative, with unit multiplication $\circ$  on the
holomorphic tangent bundle $TM$, a flat holomorphic metric $g$
and a holomorphic vector field $E$ (the Euler field), satisfying certain compatibility conditions. 
In particular, in flat coordinates $(t_{i})$ for the metric the tensor field $c(X, Y, Z) = g(X\circ Y, Z)$
can be written in terms of   the third derivatives of a certain function, the so-called  potential:
$$
c(\partial_{i}\circ \partial_{j}, \partial_{k}) =\partial_{i} \partial_{j}\partial_{k} F.
$$ 
The associativity equations  for the multiplication $\circ$  reduce to the WDVV-equations for $F$. The Euler field
rescales the multiplication and metric by constants and  imposes a  quasi-homogeneity condition on the potential.
In dimension two, the associativiy equations are empty and in dimension three they are related to Painlev\'{e} VI equations
\cite{dubrovin}. An important class of Frobenius manifolds is represented by the orbit spaces of Coxeter groups
\cite{coxeter}.

Later on, C. Hertling and Y. Manin defined the weaker notion of an $F$-manifold \cite{HM},
which does not involve any potential  or metric, but only a multiplication with similar
properties as the multiplication of a Frobenius manifold \cite{HM}.
An $F$-manifold is a complex manifold $M$ together with a commutative, associative,  multiplication $\circ$ on $TM$,
with unit $e$, 
which satisfies a certain integrability condition (see Definition \ref{def-F}).  
An Euler field on $(M, \circ , e)$ is a holomorphic vector field $E\in {\mathcal T}_{M}$ which
satisfies $L_{E}(\circ ) =\circ .$
Any Frobenius manifold without metric is an $F$-manifold with Euler field. 
But  there are $F$-manifolds which cannot be enriched to a Frobenius manifold
(see e.g. \cite{HMT}, where the multiplication of such an $F$-manifold is described in terms
 of  its  spectral cover).  A way to produce a new $F$-manifolds from older ones
is  Dubrovin's duality developed in \cite{dubrovin-ad}, or  its generalizations
developed in  \cite{manin-ad} and   \cite{DS-adv}. 
Endowing  an $F$-manifold with a real structure or Hermitian metric  leads to the notions of harmonic Higgs bundles, 
CV-structures or Hodge structures, which are central objects in $tt^{*}$-geometry  \cite{CV}.  $F$-manifolds endowed with 
purely holomorphic objects  (holomorphic metrics or compatible holomorphic connections) 
lead to notions like  Frobenius manifolds, flat $F$-manifolds, bi-flat $F$-manifolds,  Riemannian $F$-manifolds etc, which
were largely considered in  the literature, see e.g.  \cite{A,  KMS15,LP, manin-ad}  (and Section 3 of  \cite{DH-dubrovin} for a survey).  $F$-manifolds as submanifolds of Frobenius manifolds were considered in \cite{Str}.

$F$-manifolds endowed with various objects as  above  arise naturally in the theory of meromorphic connections.
More precisely, the parameter space of a certain meromorphic connection, a so-called
$(TE)$-structure,
inherits, when the so-called  unfolding condition is satisfied, 
the structure of an $F$-manifold with Euler field.  
If, moreover, the $(TE)$-structure comes with  various  additional flat objects (holomorphic metrics,  Hermitian metrics,
real structures)
its  parameter space becomes an $F$-manifold  with  various additional objects  
mentioned above (see \cite{He03}  for the way $tt^{*}$-geometry arises in this setting or Section 4 of  \cite{DH-dubrovin} for a 
survey on the holomorphic theory).
The fundamental example of this  construction is represented   by Frobenius manifolds
and their  structure connections (see \cite{dubrovin} or \cite{Hbook}).
Conversely, if one wants to enrich an F-manifold with Euler field to a 
Frobenius manifold, the most important
step (of several steps) is the construction of a $(TE)$-structure
(with additional good properties) over the F-manifold
with Euler field. This stepwise construction is discussed
in general in \cite{DH-dubrovin}. 
Examples of 
$(TE)$-structures with unfolding condition  are  universal unfoldings of certain germs of meromorphic connections
 (see e.g. \cite{HM-unfolding}).

In this paper we are only interested in the relation between $F$-manifolds 
(with or without Euler fields)
and $(TE)$-structures
(with no additional metrics or real  structures on either side of the correspondence). A natural question which arises  is to classify (formally and holomorphically) the $(TE)$-structures
which lie over (or  induce) a given germ $((M, 0), \circ , e)$ of  an $F$-manifold. 
As the classification of $F$-manifolds in dimension bigger than two is still unknown, it is natural to address this
question in two dimensions. Recall that the  irreducible germs of   $2$-dimensional 
$F$-manifold are classified:    any such germ is either isomorphic to a generically semisimple
germ  $I_{2}(m)$ (for $m\in \mathbb{Z}_{\geq 3}$)  
or  to  the globally nilpotent germ $\mathcal  N_{2}$
(see Theorem 4.7 of \cite{Hbook}). As germs of manifolds, both $I_{2}(m)$ and $\mathcal N_{2}$
are $(\mathbb{C}^{2}, 0)$ with standard coordinates $(t_{1}, t_{2})$.
The multiplication of $I_{2}(m)$  is given by
$$
\partial_{1}\circ \partial_{1} = \partial_{1},\ \partial_{1}\circ \partial_{2} = \partial_{2},\
\partial_{2}\circ \partial_{2} = t_{2}^{m-2}\partial_{1}
$$
while the multiplication of  $\mathcal N_{2}$ is given by
$$
\partial_{1}\circ \partial_{1} = \partial_{1},\ \partial_{1}\circ \partial_{2} = \partial_{2},\
\partial_{2}\circ \partial_{2} = 0,
$$
where  $\{ \partial_{1},\partial_{2}\}$
are the  vector fields associated to the standard coordinates $( t_{1}, t_{2} ).$ 
Theorem 8.5 of \cite{DH-dubrovin}  answers  the above question
for the generically semisimple 
germ $I_{2}(m).$ It states that any $(TE)$-structure over $I_{2}(m)$  is formally isomorphic to a unique $(TE)$-structure 
which belongs to a short list 
of $(TE)$-structures (called the normal forms)
and that the formal isomorphism  between a $(TE)$-structure and its normal form is holomorphic.
Therefore, the formal and holomorphic classifications for $(TE)$-structures over $I_{2}(m)$  coincide. 
As a consequence, 
any Euler field on $I_{2}(m)$ is induced by a $(TE)$-structure  
(from Theorem 8.5 of \cite{DH-dubrovin} combined with Theorem 4.7  of \cite{Hbook}).

Our aim in this paper is to 
develop similar results for the globally nilpotent germ $\mathcal N_{2}.$ Namely, we 
classify  (formally and holomorphically) the $(TE)$-structures over $\mathcal N_{2}$,  we classify the Euler fields on $\mathcal N_{2}$ and we
characterize the Euler fields on $\mathcal N_{2}$  which are induced by a $(TE)$-structure. 
Like for $I_{2}(m)$,  the classifications 
for $(TE)$-structures  are done by determining   formal and holomorphic normal forms.
The formal normal forms are obtained by computations similar to those from the $I_{2}(m)$ case.
In order to obtain the holomorphic normal forms we prove that the restriction
of  any $(TE)$-structure $\nabla$ over $\mathcal N_{2}$ at the origin $0\in \mathcal N_{2}$  is either regular singular
(in which case $\nabla$ is holomorphically isomorphic to its formal normal forms)  
or is holomorphically isomorphic to a Malgrange universal connection
(in rank two, with pole of Poincar\'{e} rank one, with residue a regular endomorphism with one eigenvalue).  By developing a careful treatment for such Malgrange universal connections
we obtain the complete list of holomorphic normal forms for $(TE)$-structures over $\mathcal N_{2}.$

As opposed to $(TE)$-structures over $I_{2}(m)$, the $(TE)$-structures over $\mathcal N_{2}$ have the
following features: the  (formal or holomorphic) normal form for a given $(TE)$-structure is  not always unique; there are  $(TE)$-structures  over $\mathcal N_{2}$ which are not holomorphically isomorphic to their formal normal form(s);  there are Euler fields on $\mathcal N_{2}$ which are not induced by a $(TE)$-structure.

Formal and holomorphic classifications are  important topics of research in the theory of meromorphic connections. 
The results of this paper add to the existing knowledge in this field,  using 
down-to-earth  arguments rather than  the abstract, more commonly used theory of Stokes structures.  This paper is  a natural continuation  of \cite{DH-AMPA}, where
a formal classification of $(T)$-structures (rather than $(TE)$-structures) over 
$\mathcal N_{2}$  was developed.\

{\bf Structure of the paper.}  
In Section \ref{preliminary} we recall  
basic definitions on  (formal and holomorphic) $(TE)$-structures and $F$-manifolds. 
In Section \ref{formal-i2}  we find the formal normal forms  for $(TE)$-structures over $\mathcal N_{2}$
and in Section \ref{holom-section}  we find the holomorphic normal forms. 
In  Section \ref{application} we find normal forms for Euler fields on $\mathcal N_{2}$ and we characterize 
the Euler fields on $\mathcal N_{2}$ which are  induced by a $(TE)$-structure.

In the appendix we study some classes of differential equations which are useful in our treatment.
To keep our paper self-contained,  we  recall well-known general results on the theory of meromorphic connections, 
which we use  along the paper  (Fuchs criterion,  irreducible bundles,  Birkhoff normal forms and Malgrange universal
connections).

\section{Preliminary material}\label{preliminary}

We preserve the notation used in \cite{DH-AMPA}, which we now recall. 

\begin{notation}{\rm For a complex manifold $M$, we denote by $\mathcal O_{M}$, $\mathcal T_{M}$, $\Omega^{k}_{M}$  the sheaves of holomorphic functions, 
holomorphic vector fields and holomorphic $k$-forms on $M$  respectively. 
For an holomorphic vector bundle $H$, we denote by ${\mathcal O}(H)$ the sheaf of its holomorphic 
sections. We denote by  $\Omega^{1}_{\mathbb{C}\times M}(\mathrm{log} \{ 0\} \times M)$ the sheaf of meromorphic $1$-forms on $\mathbb{C}\times M$, logarithmic along $\{  0\} \times M.$
Locally, in a neighborhood of $(0, p)$, where $p\in M$, any $\omega \in \Omega^{1}_{\mathbb{C}\times M}(\mathrm{log} \{ 0\} \times M)$ is of the form
$$
\omega = \frac{f(z,t)}{z} dz + \sum_{i} f_{i}(z,t) dt_{i}
$$
where $t=(t_{1},\cdots , t_{m})$ is a coordinate system of $M$ around $p$  and  $f$, $f_{i}$ are holomorphic.   
The ring of holomorphic functions defined on a neighbourhood of $0\in \mathbb{C}$ will be denoted  by 
$\mathbb{C}\{ z\}$, the ring of formal power series $\sum_{n\geq 0} a_{n} z^{n}$ will be denoted by
$\mathbb{C}[[z]]$, the subring of formal power series  $\sum_{n\geq 0} a_{n} z^{n}$ with $a_{n} =0$ for any
$n\leq k-1$ will be denoted by $\mathbb{C} [[z]]_{\geq k}$ 
and the vector space of polynomials of degree at most $k$ in the variables $(t_{1},\cdots , t_{m})$ 
will be denoted by $\mathbb{C}[t]_{\leq k}.$  
Finally, we  denote by 
$\mathbb{C}\{ t, z]]$ the ring of 
formal power series $\sum_{n\geq 0} a_{n} z^{n}$ where all 
$a_{n}= a_{n}(t)$ are holomorphic on the {\it same} neighbourhood of $0\in \mathbb{C}$
and by $\mathbb{C}[[z]][t]_{\leq k}$ the vector space
of formal power series  $\sum_{n\geq 0} a_{n} z^{n}$ with $a_{n}$  polynomials of degree at most $k$  in $t$. 
For a  function $f\in\mathbb{C}\{t,z]]$
and matrix $A\in M_{k\times k}(\mathbb{C}\{t,z]])$,
we often write $f=\sum_{n\geq 0}f^{(n)}z^n$ and 
$A=\sum_{n\geq 0}A^{(n)}z^n$ where  $f^{(n)}$
and $A^{(n)}$ are independent on $z$. 
The ring of meromorphic functions defined on a neighborhood of the origin $0\in \mathbb{C}$, with pole at the origin only,
will  be denoted by $\textit{\textbf k}$.}
\end{notation}

\subsection{Basic facts on $(TE)$-structures}\label{t-te}

Let $M$ be a complex manifold.

\begin{defn}\label{t1.1}
i) A  $(T)$-structure over $M$ is a pair $(H\rightarrow  \mathbb{C}\times M,\nabla )$
where $\nabla$ is a map
\begin{equation}\label{1.1}
\nabla:{\mathcal O}(H)\to \frac{1}{z}\mathcal O_{\mathbb{C}\times M}\cdot\Omega^1_M\otimes 
{\mathcal O}(H)
\end{equation}
such that, for any $z\in\mathbb{C}^*$,  the restriction of $\nabla$
to $H\vert_{\{z\}\times M}$ is a flat connection.\

ii)  A $(TE)$-structure over $M$ is a pair $(H\rightarrow \mathbb{C}\times M,\nabla )$
where 
$H\rightarrow \mathbb{C}\times M$ is a holomorphic vector bundle and
$\nabla $ is a flat connection on $H\vert_{\mathbb{C}^{*}\times M}$
with  poles of Poincar\'e rank 1 along $\{0\}\times M$:
\begin{equation}\label{1.2}
\nabla:{\mathcal O}(H)\to \frac{1}{z}\Omega^1_{\mathbb{C}\times M}
(\mathrm{log}(\{0\}\times M)\otimes {\mathcal O}(H).
\end{equation}
\end{defn}

As we only consider  $(T)$ or $(TE)$-structures over germs of $F$-manifolds, 
we assume that  $H= ({\mathcal O}_{(\mathbb{C}^{m}, 0)})^{r}$
 is the trivial rank $r$ vector bundle and $M = (\mathbb{C}^{m},0) $ with coordinates $(t_{1},\cdots , t_{m})$
(in fact, in our computations 
$m = r =2$, but we prefer to present the local formulae below in any rank or dimension).

With respect to the standard basis 
$\underline{s}=(s_1,\cdots ,s_r)$ of $H$,  
\begin{equation}\label{2.8}
\nabla\underline{s} = \underline{s}\cdot \Omega ,\  
\Omega = \sum_{i=1}^{m }  z^{-1} A_{i}(z,t) dt_{i}
+z^{-2}B(z,t)dz, 
\end{equation}
where $A_{i}$, $B$ are holomorphic, 
\begin{equation}\label{2.9}
A_i(z,t)=\sum_{k\geq 0}A_i^{(k)}z^k,\  B(z,t)=\sum_{k\geq 0} B^{(k)}z^k
\end{equation}
and $A_{i}^{(k)}$ and $B^{(k)}$ depend only on $(t_{i}).$ 
The flatness of 
$\nabla$ gives,  for any $i\neq j$,
\begin{align}
\label{2.10}&   z\partial_iA_j-z\partial_jA_i+[A_i,A_j]=0,\\
\label{2.11}&   z\partial_i B-z^2\partial_z A_i + zA_i + [A_i,B] =0.
\end{align}
(When $\nabla$ is a $(T)$-structure, the summand $z^{-2}B(t,z)d z$
in $\Omega$ and relations (\ref{2.11}) are dropped).
Relations  (\ref{2.10}),  (\ref{2.11}) 
split according to the  powers
of $z$  as follows: for any $k\geq 0$, 
\begin{align}
\label{2.12}&  \partial_iA_j^{(k-1)}-\partial_jA_i^{(k-1)}+\sum_{l=0}^k[A_i^{(l)},A_j^{(k-l)}] =0,\\
\label{2.13} &\partial_i B^{(k-1)}-(k-2)A_i^{(k-1)} + \sum_{l=0}^k[A_i^{(l)},B^{(k-l)}]= 0,
\end{align}
where  $A_i^{(-1)}=B^{(-1)}=0$.\

Let $(H, \nabla )$ and $(H, \tilde{\nabla})$
be two $(TE)$-structures over $(\mathbb{C}^{m}, 0)$, with underlying bundle 
$H= ({\mathcal O}_{(\mathbb{C}^{m}, 0)})^{r}$, defined  by matrices $A_{i}$, $B$ and $\tilde{A}_{i}$, $\tilde{B}$ respectively.  
An  isomorphism $T$ 
between $(H, \nabla )$ and $(\tilde{H}, \tilde{\nabla })$ 
which covers $h: (\mathbb{C}^{m}, 0)  \rightarrow (\mathbb{C}^{m}, 0)$,
$h = (h^{1},\cdots , h^{m})$, 
 is given by a matrix
$(T_{ij}) = \sum_{r\geq 0}T^{(k)} z^k
\in M_{r\times r}( \mathcal O _{\Delta \times U})$ 
(where $\Delta\subset \mathbb{C}$ is a small disc around the origin), with 
$T^{(k)}\in M_{r\times r}(\mathcal O_{U})$, $T^{(0)}$ invertible, such that 
 \begin{align}
\label{2.15} &z \partial_{i}\tilde{T} + \sum_{j=1}^{m} (\partial_{i}{h^{j}})( A_{j} \circ h ) \tilde{T}- \tilde{T} 
\tilde{A}_{i}=0,\ \forall i\\
\label{2.16}& z^{2}\partial_{z}\tilde{T}+ (B\circ h) \tilde{T} - \tilde{T}\tilde{B} =0,
\end{align}
where $\tilde{T}:= T \circ h$ (relation (\ref{2.16}) has to be omitted when $\tilde{\nabla}$ and $\nabla$ are $(T)$-structures). Relations
(\ref{2.15}),  (\ref{2.16})  split
according to the powers of $z$ as
\begin{align}
\label{compat-t}&  \partial_i \tilde{T}^{(r-1)}+ \sum_{l=0}^r (\sum_{j=1}^{m} (\partial_{i}{h^{j}}) (A_j^{(l)}\circ h)  \tilde{T}^{(r-l)}
-\tilde{T}^{(r-l)} \tilde{A}_i^{(l)})=0\\
\label{compat-te}&  (r-1)\tilde{T}^{(r-1)}+\sum_{l=0}^r ( (B^{(l)}\circ h )  \tilde{T}^{(r-l)}
-\tilde{T}^{(r-l)}  \tilde{B}^{(l)}) =0,
\end{align}
for any $r\geq 0$, where $\tilde{T}^{(-1)}=0.$ 
When  $h = \mathrm{Id}_{(\mathbb{C}^{m}, 0)}$, the isomorphism $T$ is called a gauge isomorphism.
It satisfies 
\begin{align}
\label{2.17-w}z\partial_{i}T + A_{i}T - T\tilde{A}_{i}=0\\
\label{2.18-w}z^{2} \partial_{z}T + BT - T\tilde{B}=0, 
\end{align}
or 
\begin{align}
\label{2.17} \partial_i T^{(r-1)}+\sum_{l=0}^r (A_i^{(l)}  T^{(r-l)}
-T^{(r-l)} \tilde{A}_i^{(l)})=0,\\
\label{2.18} (r-1){T}^{(r-1)}+\sum_{l=0}^r (B^{(l)}  T^{(r-l)}
-T^{(r-l)}  \tilde{B}^{(l)})=0.
\end{align}
for any $r\geq 0.$  

\begin{rem}{\rm  i)  
A formal $(T)$ or $(TE)$-structure $\nabla$ over  $(\mathbb{C}^{m}, 0)$ is given by a connection form  
 (\ref{2.8}), where $A_{i}$ and $B$
(the latter only when $\nabla$ is a $(TE)$-structure) are matrices  with entries in 
$\mathbb{C} \{ t, z]]$, satisfying relations (\ref{2.10}), (\ref{2.11}) or (\ref{2.12}), (\ref{2.13}) (relations
 (\ref{2.11}), (\ref{2.13}) only when $\nabla$ is a $(TE)$-structure). 

ii) A formal isomorphism between two formal $(T)$ or $(TE)$-structures
$\nabla$ and $\tilde{\nabla}$ 
which covers a biholomorphic map 
$h: (\mathbb{C}^{m}, 0) \rightarrow (\mathbb{C}^{m}, 0)$,  is given by a matrix  $T= (T_{ij})$
with entries $T_{ij}\in \mathbb{C} \{ t, z]]$,  such that  relations  (\ref{2.15}),  (\ref{2.16}) 
or (\ref{compat-t}), (\ref{compat-te}) 
are satisfied with $\tilde{T} = T\circ h$  
(relations (\ref{2.16}),  (\ref{compat-te}) only  when $\nabla$ and $\tilde{\nabla}$ are  $(TE)$-structures). Formal gauge isomorphisms between $(T)$ or $(TE)$-structures 
are formal isomorphisms which cover the identity map. They are given by matrices  $T= (T_{ij})$ with entries in 
$\mathbb{C} \{ t, z]]$ such that relations 
(\ref{2.17-w}), (\ref{2.18-w}) or 
(\ref{2.17}), (\ref{2.18}) are satisfied
(relations (\ref{2.18-w}) and (\ref{2.18}) only for $(TE)$-structures).}
\end{rem}

\subsection{$(TE)$-structures and  $F$-manifolds}\label{s2.2}

\subsubsection{General results}\label{gen-sect}

Let $(H, \nabla )$ be a $(T)$-structure over a complex manifold $M$. 
It induces a Higgs field  $C\in \Omega^{1}(M, \mathrm{End}(K))$
on the restriction $K:=H_{|\{0\}\times M}$, defined by
\begin{equation}
C_X[a]:=[z\nabla_Xa],\ \forall  X\in {\mathcal T}_M,a\in{\mathcal O} (H),
\end{equation}
where $[\ ]$ means the restriction to $\{ 0\} \times M$ and $X\in{\mathcal T}_M$
is lifted canonically 
to ${\mathbb C}\times M$.  If $(H, \nabla )$ is a $(TE)$-structure then there is in addition an endomorphism $\mathcal U\in \mathrm{End}(K) $, 
\begin{equation}\label{mathcal-u}
\mathcal U :=[z\nabla_{z\partial_{z}}]:\mathcal O (K)\rightarrow  \mathcal O (K).
\end{equation}

\begin{defn} (\cite{He03}) 
The $(T)$-structure  (or  $(TE)$-structure)  $(H, \nabla )$  satisfies the 
{\it unfolding condition} if   there is 
an open cover $\mathcal V$ of $M$ and for any $U\in \mathcal V$ 
a section  $\zeta_{U}\in {\mathcal O} (K\vert_{U})$ (called a local primitive section) 
with the property that the map
$TU\ni X\rightarrow C_{X} \zeta_{U} \in K$  is an isomorphism.
\end{defn}

When $(H\rightarrow \mathbb{C}\times M, \nabla )$ satisfies the unfolding condition 
the rank of $H$ coincides with the dimension of $M$.

 \begin{defn}\label{def-F}  (\cite{HM})
A complex manifold $M$ with a (fiber-preserving) commutative, associative  
multiplication $\circ$ on the holomorphic 
tangent bundle $TM$ and unit field $e\in {\mathcal T}_M$ is an
{\it $F$-manifold} if 
\begin{equation}\label{2.19}
L_{X\circ Y}(\circ) =X\circ L_Y(\circ)+Y\circ L_X(\circ),\  \forall X, Y\in {\mathcal T}_{M},
\end{equation}
where $L_{X}$ denotes the Lie derivative in the direction of $X\in {\mathcal T}_{M}.$ 
A vector field $E\in{\mathcal T}_M$ is called an {\it Euler field}
(of weight $1$)  if
\begin{equation}\label{2.20}
L_E(\circ) =\circ.
\end{equation}
\end{defn}

The following theorem  was  proved in 
Theorem 3.3 of \cite{HHP10}.

\begin{thm}\label{t2.4}
A $(T)$-structure  $(H\rightarrow \mathbb{C}\times M, \nabla )$  
with unfolding condition 
induces a multiplication $\circ$ on $TM$ which makes $M$ an $F$-manifold.
A $(TE)$-structure $(H\rightarrow \mathbb{C}\times M, \nabla )$   with unfolding condition induces in addition
a vector field $E$ on $M$, which, together with $\circ$, makes $M$ an $F$-manifold with Euler field.
The multiplication $\circ$, unit field $e$ and Euler field $E$ (the latter, in the case of a $(TE)$-structure), are defined by
\begin{equation}
C_{X\circ Y} = C_{X} C_{Y},\ C_{e} =\mathrm{Id},\ C_{E} = - {\mathcal U} 
\end{equation}
where $C$
and $\mathcal U$ are the Higgs field and endomorphism defined by $\nabla$ as above.
\end{thm}

A $(T)$- or $(TE)$-structure as in Theorem \ref{t2.4} is said to lie over the $F$-manifold $(M, \circ , e)$.
$F$-manifold isomorphisms lift naturally to isomorphisms  between the spaces of $(T)$ or $(TE)$-structures lying over the
respective $F$-manifolds. 
In particular,  the  spaces of (formal or holomorphic) $(T)$- or $(TE)$-structures  over isomorphic germs of $F$-manifolds  are isomorphic.

\subsubsection{$(T)$-structures over $\mathcal N_{2}$}\label{definition-n2}

Following \cite{DH-AMPA}, we recall the formal normal forms of 
$(T)$-structures over $\mathcal N_{2}.$ 
They are the starting point in our treatment of $(TE)$-structures over $\mathcal N_{2}.$

\begin{notation}{\rm 
We define matrices  $C_{1}$, $C_{2}$, $D$ and $E$,  by
\begin{equation}\label{6.3}
C_1:= \mathrm{Id }_2,\ 
C_{2}:= \left( \begin{tabular}{cc}
$0$ & $0$\\
$1$ & $0$
\end{tabular}\right) ,\  
D:= \left( \begin{tabular}{cc}
$1$ & $0$\\
$0$ & $-1$
\end{tabular}\right) , \ 
E := \left( \begin{tabular}{cc}
$0$ & $1$\\
$0$ & $0$\end{tabular}\right) . 
\end{equation}
We remark that 
\begin{align}
\label{r1} & (C_2)^{2}=0,\  D^{2}  =C_1,\  E^{2}=0,\\
\label{r2} & C_2  D =C_2= - D  C_2,\  D  E=E =-E  D,\\
\label{r3} &   C_2  E=\frac{1}{2} (C_{1}-D),\  E  C_2 = \frac{1}{2} (C_{1} + D),\\
\label{r4}& [C_{2}, D] = 2 C_{2},\ [C_{2}, E]= - D,\ [D, E]=2E. 
\end{align} }
\end{notation}

\begin{thm}(\cite{DH-AMPA}) \label{iso-formal}   Any  
$(T)$-structure  
over $\mathcal N_{2}$ is formally isomorphic  
to a   $(T)$-structure of the form
\begin{align}
\label{5.19a} &A_{1}= C_{1},\   A_{2}= C_{2} + z E\\
\label{5.19b} &A_{1}= C_{1},\  A_{2}= C_{2} + z t_{2} E\\
\label{5.19} & A_{1} = C_{1},\ A_{2} = C_{2},
\end{align}
or to a holomorphic or formal  $(T)$-structure of the form
\begin{equation}\label{5.20}
A_{1} = C_{1},\ A_{2} = C_{2} + z (t_{2}^{r} +\sum_{k\geq 1}
P_{k}z^{k})E, 
\end{equation}
where $r\in \mathbb{Z}_{\geq  2}$ and $P_{k}\in\mathbb{C}[t_2]_{\leq r-2}$ are polynomials 
of degree at most $r-2$.
\end{thm}

The $(T)$-structures from Theorem \ref{iso-formal} 
are called  formal normal forms.  They are pairwise formally gauge non-isomorphic
(i.e.  there is no formal gauge isomorphism between any two distinct formal normal forms).  
However, there  exist   (distinct) formal normal forms  
which are formally  isomorphic (by a  formal  isomorphism which is not a formal gauge  
isomorphism).  For
a precise statement, see Theorem  21  of   \cite{DH-AMPA}. 
Recall that the  automorphism group  $\mathrm{Aut} (\mathcal N_{2})$
of the $F$-manifold germ $\mathcal N_{2}$ 
is the group of  all biholomorphic maps 
\begin{equation}\label{tilde-f}
(t_{1}, t_{2}) \rightarrow ( t_{1}, \lambda (t_{2})),
\end{equation}
where $\lambda\in \mathbb{C}\{ t_{2}\}$, with $\lambda (0) =0$ and $\dot{\lambda}(0) \neq 0$
(i.e. $\lambda \in \mathrm{Aut} (\mathbb{C}, 0))$).

\section{Formal classification of $(TE)$-structures}\label{formal-i2}

Our aim in this section is to prove the following two theorems, which  classify formally  the
$(TE)$-structures over $\mathcal N_{2}.$

\begin{thm}\label{te}  Any formal $(TE)$-structure  over $\mathcal N_{2}$ is formally isomorphic 
to a   $(TE)$-structure  of the following forms:\ 

i)  for  $c, \alpha , c_{0}\in \mathbb{C}$, 
\begin{align}
\nonumber & A_{1} = C_{1},\ A_{2} = C_{2} + zE,\\
\label{f-1}& B = ( - t_{1} + c +\alpha z) C_{1} + ( -\frac{t_{2}}{2} + c_{0}) C_{2} -\frac{z}{4}D +
z( -\frac{t_{2}}{2} + c_{0})E;
\end{align}
ii)  for $c, \alpha \in \mathbb{C}$ and $r\in\mathbb{Z}_{ \geq 1}$, 
\begin{align}
\nonumber & A_{1} = C_{1},\ A_{2} = C_{2} + z t^{r}_{2}E,\\
\label{f1}& B = ( - t_{1} + c +\alpha z) C_{1} -\frac{t_{2}}{r+2} C_{2} -\frac{z(r+1)}{2(r+2)}D 
-\frac{zt_{2}^{r+1}}{r+2}E,
\end{align}
iii)  a $(TE)$-structure with underlying $(T)$-structure 
$A_{1} = C_{1}$, $A_{2} =C_{2}$ and matrix $B$ of one of the following forms:
\begin{align}
\nonumber B= & ( - t_{1} + c + \alpha z) C_{1} -\frac{z}{2}D;\\
\nonumber B= & ( - t_{1} + c +\alpha z) C_{1} + t_{2}^{2} C_{2}-z( t_{2} +\frac{1}{2} ) D- z^{2} E;\\
\nonumber B= &  ( - t_{1} + c +\alpha z) C_{1} +\lambda t_{2}C_{2} -\frac{z}{2} (\lambda + 1) D,\ \lambda \notin \mathbb{Z}\setminus \{ 0\};\\
\nonumber B  =  & (-t_{1} + c +\alpha z) C_{1} +(\lambda t_{2} + 1)C_{2} -\frac{z}{2} (\lambda +1) D,\ \lambda \notin
\mathbb{Z}\setminus \{ 0\};\\
\nonumber B  =  & (-t_{1} + c +\alpha z) C_{1}+(\lambda t_{2} +1 +\gamma t_{2}^{2} z^{\lambda})C_{2}\\
\nonumber & -\frac{z}{2} (\lambda +1 +2\gamma t_{2} z^{\lambda}) D -\gamma z^{\lambda +2} E,\
\lambda \in \mathbb{Z}_{\geq 1},\\
\nonumber B = &  ( - t_{1} + c +\alpha z) C_{1} + t_{2} (\lambda +  t_{2} z^{\lambda} )C_{2} -\frac{z}{2} 
(\lambda +1 + 2 t_{2} z^{\lambda }) D\\
\nonumber&  -  z^{\lambda +2}E,\ \lambda \in \mathbb{Z}_{\geq 1},\\
\nonumber B=& ( - t_{1} + c +\alpha z) C_{1} + t_{2} \lambda  C_{2} -\frac{z}{2} ( \lambda + 1) D,\ \lambda
\in \mathbb{Z}_{\geq 1};\\
\nonumber B = & (-t_{1} + c+\alpha z) C_{1}  + (\lambda t_{2} + z^{-\lambda }) C_{2} -\frac{ z}{2} (\lambda + 1)
D,\ \lambda \in \mathbb{Z}_{\leq -1}\\
\label{b-2-0} B= & (-t_{1} + c+\alpha z) C_{1} + \lambda t_{2} C_{2} -\frac{z}{2}  (\lambda +1)D,\ 
\lambda \in \mathbb{Z}_{\leq -1},
\end{align}  
where  $c, \alpha ,\gamma \in \mathbb{C}$. 
\end{thm}

The $(TE)$-structures from the above theorem are called formal normal forms.
The next theorem studies when two formal normal forms  are formally isomorphic.

\begin{thm}\label{te-class} Any two (distinct) $(TE)$-structures $\nabla$ and $\tilde{\nabla}$ 
in formal normal form  are formally non-isomorphic, 
except when\  

i)  both $\nabla$ and $\tilde{\nabla}$ 
are of the form (\ref{f-1}), with constants $c$, $\alpha$, $c_{0}$, respectively $\tilde{c}$, $\tilde{\alpha}$,
$\tilde{c}_{0}$ and $c_{0}\tilde{c}_{0}\neq 0$. Then they are formally isomorphic if and only if $\tilde{c}= c$, $\tilde{\alpha } = \alpha$ and 
$\tilde{c}_{0} = -  c_{0}$ and they are  formally gauge non-isomorphic;\

ii) both $\nabla$ and $\tilde{\nabla}$ have underlying $(T)$-structure $A_{1}=C_{1}$, $A_{2} = C_{2}$ and their
matrices $B$ and $\tilde{B}$ are of the fourth form   in (\ref{b-2-0}), with constants  $c$, $\alpha$,
$\lambda$, respectively $\tilde{c}$, $\tilde{\alpha}$, $\tilde{\lambda}$ 
and $\lambda\tilde{\lambda}\neq 0.$ Then they are formally isomorphic if and only if
$\tilde{c}=c$, $\tilde{\alpha}=\alpha$
and $\tilde{\lambda} = -\lambda $ and they are formally gauge non-isomorphic.  
\end{thm}

In order to prove Theorems \ref{te} and \ref{te-class}, we  begin by  determining  the $(TE)$-structures which extend the $(T)$-structures
from Theorem \ref{iso-formal}.  This is done in Section
\ref{step11} below.

\subsection{$(TE)$-structures  with normal formal  $(T)$-structures}\label{step11}

Let $\nabla$ be a formal $(TE)$-structure, with underlying formal $(T)$-structure $A_{1} = C_{1}$, $A_{2} = C_{2} + z fE$,
where   $f\in \mathbb{C} \{ t_{2}, z]].$  

\begin{lem}\label{simplif-0} The formal $(TE)$-structure $\nabla$ is formally isomorphic to a $(TE)$-structure
with $A_{1}=C_{1}$, $A_{2} = C_{2}+ z f E$ and
matrix   $B$  of the form
\begin{equation}\label{B}
B = ( -t_{1}+ c + \alpha z ) C_{1} + b_{2} C_{2} + zb_{3} D + zb_{4} E,
\end{equation}
where $\alpha ,c\in \mathbb{C}$, $b_{2}, b_{3}, b_{4} \in \mathbb{C} \{ t_{2} ,  z]]$, 
\begin{equation}\label{b3-4}
b_{3} = -\frac{1}{2} ( \partial_{2} b_{2} + 1),\   b_{4}= -\frac{z}{2} \partial^{2}_{2} b_{2} + f b_{2}
\end{equation}
and
\begin{equation}\label{b-2} 
-\frac{z}{2}  \partial_{2}^{3} b_{2} + (\partial_{2}f) b_{2} + 2f \partial_{2} b_{2} - z\partial_{z} f + f=0.
\end{equation}
\end{lem}

\begin{proof} Relation (\ref{2.11}) with $i=1$ gives $\partial_{1} B = - C_{1}$. Relation  (\ref{2.11}) with $ i=2$ gives
$[ C_{2}, B^{(0)} ]=0$ i.e. $B^{(0)}$ is a linear combination of $C_{1}$ and $C_{2}.$  
We obtain 
\begin{equation}\label{6.32}
B= (-t_1+ b_1)C_1+b_2C_2+zb_3D+zb_4E,
\end{equation}
where $b_1,b_2,b_3,b_4\in\mathbb{C}\{t_2, z]]$. 
With $B$ given by 
(\ref{6.32})  and $A_{2} = C_{2}+ z f E$, relation  (\ref{2.11})
for $i=2$ is equivalent to  
$ \partial_2b_1=0$, 
(\ref{b3-4}) and (\ref{b-2}).
It remains to show that $b_{1}$ can be chosen of the form $b_{1}(z, {t}_{2}) = c + \alpha z$,  for $c\in \mathbb{C}.$ 
Since $\partial_{2}b_{1} =0$ and  $b_1\in\mathbb{C}\{t_2, z]]$, we can write $b_1=  c + \alpha  z+\sum_{k\geq 2}b_1^{(k)} z^k$,  for  $\alpha , b_{1}^{(k)}\in \mathbb{C}$. Define
\begin{equation}\label{6.37}
T:=\mathrm{exp}(-\sum_{k\geq 2}\frac{b_1^{(k)}}{k-1} z^{k-1}) C_1.
\end{equation}
The isomorphism $T$ maps $\nabla$ to a new $(TE)$-structure which has the same underling $(T)$-structure 
$A_1 = C_{1}$,  $A_2= C_{2}+ z f E$  and
the only change in $B$ is that $b_1$ is replaced by $\tilde{b}_1 (z, {t}_{2})= c+ \alpha  z$.
\end{proof}

To simplify terminology we introduce the next definition.

\begin{defn}\label{assoc-fct} 
A formal $(TE)$-structure $\nabla$ as in Lemma 
\ref{simplif-0} is said to be in pre-normal form, determined by $(f, b_{2}, c, \alpha ).$ 
The functions $(f, b_{2})$ are called associated to $\nabla .$  
\end{defn}

 In order to find all formal $(TE)$-structures which extend the formal 
$(T)$-structures from Theorem \ref{iso-formal}, we need  
to determine their associated functions $b_{2}$, i.e. to solve  equation 
(\ref{b-2})  in the unknown function $b_{2}$, for various classes of  functions $f$, which correspond to the
various classes of formal  $(T)$-structures
from Theorem \ref{iso-formal}.  This is done in the next proposition.

\begin{prop}\label{step-1} i) The first  $(T)$-structure from Theorem \ref{iso-formal} extends to formal $(TE)$-structures, with
matrices $B$  given by
\begin{equation}\label{forma-B}
B= ( -t_{1} + c+ \alpha z) C_{1}+  ( -\frac{t_{2}}{2} + \sum_{k\geq 0} c_{k} z^{k})C_{2} -\frac{z}{4}D 
+ z(-\frac{t_{2}}{2} +\sum_{k\geq 0} c_{k} z^{k}) E,
\end{equation}
where $\alpha , c, c_{k}\in \mathbb{C}.$\

ii) The second $(T)$-structure from Theorem \ref{iso-formal} extends to formal $(TE)$-structures, with matrices $B$  given by
\begin{equation}\label{classif-b} 
B=   ( -t_{1} +c+ \alpha z) C_{1}-\frac{t_{2}}{3}C_{2} -\frac{z}{3} D -\frac{zt_{2}^{2}}{3}E,
\end{equation}
where $c, \alpha \in \mathbb{C}.$\

iii) The third $(T)$-structure from Theorem \ref{iso-formal} 
extends to formal  $(TE)$-structures with
matrices $B$ as in (\ref{B}), functions $b_{3}$ and $b_{4}$ given by (\ref{b3-4}) with $f= 0$ and 
 function  $b_{2} = \sum_{n\geq 0}b_{2}^{(n)} z^{n}$, such that  $b_{2}^{(n)}\in \mathbb{C} \{ t_{2} \}$ satisfies  
 $\partial_{2}^{3}b_{2}^{(n)}=0$, for any $n\geq 0.$\

iv)  The fourth (formal or holomorphic) $(T)$-structure from Theorem \ref{iso-formal} extends to a formal $(TE)$-structure if and only
if $P_{k} =0$, for any $k\geq 1.$ When $P_{k} =0$ for any $k\geq 1$, the extended formal $(TE)$-structures have 
matrices $B$
given by
\begin{equation}\label{classif-b-1} 
 B= (-t_{1} +c+ \alpha z) C_{1} -\frac{t_{2}}{r+2}C_{2} -\frac{z(r+1)}{2(r+2)}D -\frac{ zt_{2}^{r+1}}{r+2} E, 
\end{equation}
where  $\alpha , c \in \mathbb{C}$ (and $r\in \mathbb{Z}_{\geq 2}$).\

\end{prop}

\begin{proof}  
We only prove claim iv) (which is more involved), since the other claims can be proved similarly.
Let 
\begin{equation}\label{f-t-z}
f( z, t_{2}) = t_{2}^{r} +  \sum_{k\geq 1} P_{k} (t_{2}) z^{k},
\end{equation}
where $P_{k}$ are polynomials of degree at most $r-2$. 
Equation (\ref{b-2}) with $f$ given by (\ref{f-t-z}) becomes 
\begin{align}
\nonumber &-\frac{z}{2}  \partial_{2}^{3} b_{2} + ( r t_{2}^{r-1} + \sum_{k\geq 1} \dot{P}_{k}(t_{2}) z^{k} ) b_{2}\\
\label{b-2-3}  &+ 2( t_{2}^{r} + \sum_{k\geq 1} P_{k}(t_{2}) z^{k}) \partial_{2} b_{2} +  \sum_{k\geq 1}( 1- k) P_{k}(t_{2})  z^{k}+ t_{2}^{r} =0.
\end{align}
We write $b_{2} = \sum_{k\geq 0} b_{2}^{(k)}z^{k}$ with
$b_{2}^{(k)}$ independent on $z$.
Identifying the coefficients of $z^{0}$ in (\ref{b-2-3}) we obtain
$$
r b_{2}^{(0)}  + 2t_{2} \partial_{2} b_{2}^{(0)} + t_{2} =0,
$$
which  implies 
\begin{equation}\label{n1}
b_{2}^{(0) }= -\frac{t_{2}}{r+2} . 
\end{equation}
Identifying the coefficients of $z^{1}$ in (\ref{b-2-3}) and using   (\ref{n1}) we obtain
\begin{equation}\label{ident-degree}
r t_{2}^{r-1} b_{2}^{(1)} + 2 t_{2}^{r} \partial_{2} b_{2}^{(1)} -\frac{1}{r+2} \left( \dot{P}_{1}(t_{2})  t_{2} + 2 P_{1}(t_{2}) \right) =0.
\end{equation}
The first two terms in (\ref{ident-degree})   have degree at least $r-1$ and  the last two terms have degree at most $r-2.$ We obtain  that 
(\ref{ident-degree}) 
is equivalent to
\begin{align*}
& rb_{2}^{(1)} + 2 t_{2} \partial_{2}  b_{2}^{(1)} =0\\
& \dot{P}_{1} (t_{2})t_{2} + 2 P_{1} (t_{2}) =0,
\end{align*}
which imply  $b_{2}^{(1)} =0$ and $P_{1} =0.$  
Identifying the coefficients of $z^{n}$ for $n\geq 2$ in (\ref{b-2-3}) and using 
an induction argument  we obtain that $b_{2}^{(k)} =0$ for any $k\geq 2$ and $P_{k} =0$ for any $k\geq 1.$ 
From (\ref{n1}) and (\ref{b3-4}) we obtain 
$$ 
b_{2} = -\frac{t_{2}}{r+2},\  b_{3} = -\frac{r+1}{2(r+2)},\ b_{4}= -\frac{t_{2}^{r+1}}{r+2},
$$
which implies  claim iv). 
\end{proof}

\subsection{Proof of Theorem \ref{te}}\label{step111}

The existence of a formal
isomorphism between an arbitrary  $(TE)$-structure and one from Theorem \ref{te} will be proved 
by applying to the $(TE)$-structures from Proposition \ref{step-1}   formal automorphisms of their
underlying $(T)$-structures. 
We shall proceed in two steps:  I) we start with
the $(TE)$-structures  
from Proposition \ref{step-1} i) and we obtain  the $(TE)$-structures  
from Theorem \ref{te} i); 
 II) we start with  the $(TE)$-structures from Proposition \ref{step-1} iii) and we obtain
the $(TE)$-structures  
from Theorem \ref{te} iii). 
(The $(TE)$-structures from Proposition \ref{step-1} ii) and iv) are written in 
Theorem \ref{te} ii)  in a unified way).\

\subsubsection{The first step}

The proof of the next lemma is straightforward and will be omitted. 

\begin{lem}\label{aut-lema1}  A  formal automorphism  of the $(T)$-structure
$A= C_{1}$,  $A_{2} =    C_{2} + z E$
is either a formal  gauge automorphism,  given by 
\begin{equation}\label{auto-t}
T = (\sum_{n\geq 0} \tau_{1}^{(n)} z^{n})C_{1} + (\sum_{n\geq 0} \tau_{2}^{(n)} z^{n}) C_{2} + (\sum_{n\geq 1} 
\tau_{2}^{(n-1)}z^{n}) E 
\end{equation}
where $\tau_{1}^{(n)}, \tau_{2}^{(n)} \in \mathbb{C}$  and $\tau_{1}^{(0)}\neq 0$, or  
covers the map $h(t_{1}, t_{2}) = (t_{1}, - t_{2})$ and is given by 
\begin{equation}\label{auto-t-t}
{T} = (\sum_{n\geq 0} \tau_{2}^{(n)} z^{n}) C_{2} + (\sum_{n\geq 0} \tau_{3}^{(n)} z^{n}) D - (\sum_{n\geq 1} 
\tau_{2}^{(n-1)} z^{n}) E ,
\end{equation}
where $\tau_{2}^{(n)}, \tau_{3}^{(n)}\in \mathbb{C}$ and $\tau_{3}^{(0)}\neq 0.$ 
\end{lem}

\begin{lem}\label{stat1} Let $\nabla$ and $\tilde{\nabla}$ be two formal  
$(TE)$-structures as in Proposition \ref{step-1} i),  
with constants
$c, \alpha , c_{k}$ ($k\geq 0$) and, respectively, $\tilde{c}, \tilde{\alpha}, 
\tilde{c}_{k}$ ($k\geq 0).$ 
Then $\nabla$ and $\tilde{\nabla}$ are formally gauge isomorphic 
if and only if 
\begin{equation}\label{low-cond}
c = \tilde{c},\  \alpha = \tilde{\alpha},\ c_{0} = \tilde{c}_{0}.
\end{equation} 
In particular, $\nabla$ is formally gauge isomorphic 
to the  $(TE)$-structure  (\ref{f-1}).
\end{lem}

\begin{proof} Let $T$ be a formal gauge isomorphism 
between $\nabla$ and $\tilde{\nabla }.$ As $\nabla$ and $\tilde{\nabla}$
have the same underlying $(T)$-structure  $A_{1} = C_{1}$, $A_{2} = C_{2} + z E$, 
$T$ is a formal gauge  automorphism of this $(T)$-structure. From Lemma \ref{aut-lema1}, 
$T$ is of the form
(\ref{auto-t}), and must satisfy  relations
(\ref{2.18}), with $B$ and $\tilde{B}$ of the form (\ref{forma-B}),
with constants $c, \alpha , c_{k}$ and $\tilde{c}, \tilde{\alpha}, \tilde{c}_{k}.$

Replacing $T^{(l)}$, $B^{(l)}$ and $\tilde{B}^{(l)}$ in (\ref{2.18})  and identifying  the coefficients
of $\{ C_{1}, C_{2}, D, E\}$  we obtain, from a straightforward computation
which uses relations (\ref{r1})-(\ref{r4}),  
\begin{align}
\nonumber & c =\tilde{c},\ \alpha = \tilde{\alpha},\ c_{0} = \tilde{c}_{0};\\
\nonumber& \frac{\tau_{2}^{(0)}}{2} + \tau_{1}^{(0)} ( c_{1} - \tilde{c}_{1}) =0;\\
\nonumber & (n-1) \tau_{1}^{(n-1)} + \sum_{l=2}^{n} \tau_{2}^{(n-l)} (c_{l-1} -\tilde{c}_{l-1}) =0;\\
\label{specific}& (n-\frac{1}{2}) \tau_{2}^{(n-1)} + 
\sum_{l=1}^{n} \tau_{1}^{(n-l)} (c_{l} - \tilde{c}_{l}) =0,
\end{align}
for any  $n\geq 2.$ 
(In all. relations (\ref{2.18}) the coefficients of $D$ vanish; the  coefficients of $E$ in 
(\ref{2.18}), with $r=0,1$, vanish as well
and the coefficient of $E$ in (\ref{2.18}), with $r\geq 2$, coincides with the coefficient of $C_{2}$  
in (\ref{2.18}), with $r$ replaced by $r-1$. Thus, relations (\ref{2.18}) are equivalent to the vanishing of their coefficients
of $C_{1}$ and $C_{2}$, which leads to relations (\ref{specific})). 
In particular, if $\nabla$ and $\tilde{\nabla}$ are formally gauge  isomorphic, then 
(\ref{low-cond}) is satisfied. Conversely, assume that (\ref{low-cond}) is satisfied.
We aim to construct a formal gauge isomorphism between $\nabla$ and $\tilde{\nabla}$, i.e. to find
$\tau_{1}^{(n)}, \tau_{2}^{(n)}\in \mathbb{C}$, with $\tau_{1}^{(0)}\neq 0$, 
such that relations (\ref{specific}) hold. 
Let $\tau_{1}^{(0)}\in \mathbb{C}^{*}$ be 
arbitrary. The second relation (\ref{specific}) determines $\tau_{2}^{(0)}$ and then, $\tau_{1}^{(1)}$ is determined by the third relation
(\ref{specific}) with $n=2$:
$$
\tau_{1}^{(1)} = \tau_{2}^{(0)} (\tilde{c}_{1} - {c}_{1}). 
$$ 
Knowing $\tau_{1}^{(0)}$ and $\tau_{1}^{(1)}$, the fourth relation (\ref{specific}) with $n=2$ determines $\tau_{2}^{(1)}$:
$$
\tau_{2}^{(1)}  =  \frac{2}{3} \left( \tau_{1}^{(1) }(\tilde{c}_{1} -  {c}_{1}) + \tau_{1}^{(0)} (\tilde{c}_{2} -{c}_{2})\right ). 
$$
Repeating the argument  we obtain inductively $\tau_{1}^{(l)}$ and $\tau_{2}^{(l)}$, for all $l\geq 1.$ 
\end{proof}

\subsubsection{The second step}

We use a similar  argument for the $(TE)$-structures from  Proposition \ref{step-1} iii).  
As before, we begin by finding the automorphisms of their underlying
$(T)$-structure.

\begin{lem}\label{auto-last} 
i) Any  formal automorphism $T$  of the $(T)$-structure  $A_{1} = C_{1}$, $A_{2} = C_{2}$ covers an
automorphism $h\in \mathrm{Aut} (\mathcal N_{2})$ of the form
\begin{equation}\label{H}
h(t_{1}, t_{2}) = (t_{1}, \frac{k t_{2}}{e t_{2} + d}),  
\end{equation}
where $e\in \mathbb{C}$, 
$k, d\in \mathbb{C}^{*}$ and  
\begin{equation}\label{tilde-t-def}
\tilde{T}:= T\circ  h=  \sum_{n\geq 0} \tilde{T}^{(n)} z^{n},\  \tilde{T}^{(n)} = \tau_{1}^{(n)} C_{1} + \tau_{2}^{(n)} C_{2} + \tau_{3}^{(n)} D + \tau_{4}^{(n)} E
\end{equation}
is given by: for any $n\geq 0$, $\tau_{4}^{(n)}\in \mathbb{C}$, with
$\tau_{4}^{(0)}=0$, $\tau_{4}^{(1)}= e$, and  
\begin{align}
\nonumber \tau_{1}^{(n)}& = \frac{ t_{2}(e t_{2} + 
d-k)}{2( et_{2} + d)} \tau_{4}^{(n+1)} + (\tau_{1}^{(n)})_{0}\\
\nonumber\tau_{2}^{(n)} & = -\frac{ t_{2}^{2} k}{et_{2} + d}
\tau_{4}^{(n+2)}+ \frac{t_{2}(e t_{2} + d-k)}{et_{2} + d}
(\tau_{1}^{(n+1)})_{0}\\
\nonumber& 
+ (\tau_{2}^{(n)})_{0}-\frac{t_{2}( et_{2} + d+k)}{e t_{2} + d}(\tau_{3}^{(n+1)})_{0},\\
\label{auto-coeff-0}\tau_{3}^{(n)} & = \frac{ t_{2}(e t_{2} + d+k)}{2(e t_{2} + d)}
\tau_{4}^{(n+1)}  + (\tau_{4}^{(n)})_{0},
\end{align}
where $(\tau_{1}^{(n)})_{0}, 
(\tau_{2}^{(n)})_{0},  (\tau_{3}^{(n)})_{0} \in \mathbb{C}$ and 
$(\tau_{1}^{(0)})_{0} = \frac{1}{2} ( d+k)$, $(\tau_{3}^{(0)})_{0} = \frac{1}{2}(d-k)$.

ii)  Any  formal gauge automorphism of the $(T)$-structure  $A_{1} = C_{1}$, $A_{2} = C_{2}$ is of the form 
$$
T= (\sum_{k\geq 0} \tau_{1}^{(n)}z^{n}) C_{1} + \sum_{n\geq 0} (\tau_{2}^{(n)}z^{n}) C_{2} + 
(\sum_{n\geq 0}\tau_{3}^{(n)}z^{n}) D + (\sum_{n\geq 0}\tau_{4}^{(n)}z^{n}) E
$$
where $\tau_{1}^{(n)}\in \mathbb{C}$ with $\tau_{1}^{(0)}\neq 0$, $\tau_{2}^{(n)}, \tau_{3}^{(n)}, \tau_{4}^{(n)} \in \mathbb{C}\{ t_{2}\}$, $\tau_{2}^{(n)}$ satisfies $\partial_{2}^{3}\tau_{2}^{(n)} =0$
for any $n\geq 0$ and 
\begin{equation}\label{d-e-p}
\tau_{3}^{(n)} = -\frac{1}{2} \partial_{2} \tau_{2}^{(n-1)},\ \tau_{4}^{(n)} =-\frac{1}{2} \partial^{2}_{2} \tau_{2}^{(n-2)},
\end{equation}
with  the convention $\tau_{2}^{(n)}=0$ for $n<0.$ 
\end{lem}

\begin{proof} i)  Let  $T$ be a formal automorphism of
the $(T)$-structure $A_{1} = C_{1}$, $A_{2} = C_{2}$ and 
$h(t_{1}, t_{2}) = (t_{1}, \lambda (t_{2}))$
the automorphism  of $\mathcal N_{2}$ covered by $T$.   
Let  $\tau_{1}^{(n)}$, $\tau_{2}^{(n)}$, $\tau_{3}^{(n)}$,  $\tau_{4}^{(n)}$ 
be the functions defined by
(\ref{tilde-t-def}). 
From relation 
(\ref{compat-t}) with $i=1$ and $A_{1} = \tilde{A}_{1} = C_{1}$, we
obtain that they are independent on $t_{1}.$  
Relation (\ref{compat-t}) with $i=2$ is 
\begin{equation}\label{explicit-2.17}
\partial_{2} \tilde{T}^{(n-1)} + \dot{\lambda}C_{2} \tilde{T}^{(n)} -\tilde{T}^{(n)} C_{2}=0,\ n\geq 0.
\end{equation} 

Using relations  (\ref{r1})-(\ref{r4})  we obtain that (\ref{explicit-2.17})  is equivalent to
\begin{equation}\label{dot-lambda}
\tau_{4}^{(0)} =0,\ \dot{\lambda} =\frac{ \tau_{1}^{(0)} - \tau_{3}^{(0)}}{\tau_{1}^{(0)} + \tau_{3}^{(0)}},
\end{equation}  
$\tau_{4}^{(n)}\in \mathbb{C}$ 
(for $n\geq 1$) and, for any $n\geq 0$, 
\begin{align}
\nonumber& \partial_{2} \tau_{1}^{(n)} = (\frac{1-\dot{\lambda}}{2}) \tau_{4}^{(n+1)},\\
\nonumber& \partial_{2} \tau_{2}^{(n)} = (1- \dot{\lambda}) \tau_{1}^{(n+1)} - ( 1+ \dot{\lambda }) \tau_{3}^{(n+1)},\\ 
\label{derivate}& \partial_{2} \tau_{3}^{(n)} =  (\frac{1+ \dot{\lambda}}{2}) \tau_{4}^{(n+1)}.
\end{align}
(Since $\tilde{T}^{(0)}$ is invertible and $\tau_{4}^{(0)} =0$, we obtain that $\tau_{1}^{(0)} - \tau_{3}^{(0)}$, 
$\tau_{1}^{(0)} + \tau_{3}^{(0)}$
are units in $\mathbb{C} \{ t_{2} \}$). 
Using that $\tau_{4}^{(n)}$ are constant, we obtain from the first and third relation (\ref{derivate})  that $\tau_{1}^{(n)}$ and $\tau_{3}^{(n)}$ are given by
\begin{align}
\nonumber \tau_{1}^{(n)} & = \frac{\tau_{4}^{(n+1)}}{2} (t_{2} - \lambda ) + (\tau_{1}^{(n)})_{0}\\
\label{k-d} \tau_{3}^{(n)} &= \frac{\tau_{4}^{(n+1)}}{2} (t_{2} +\lambda ) + (\tau_{3}^{(n)})_{0},
\end{align}
where $(\tau_{1}^{(n)})_{0}, (\tau_{3}^{(n)})_{0}\in \mathbb{C}$.  From (\ref{k-d}) and  the second relation (\ref{derivate}), we obtain 
$$
\tau_{2}^{(n)} = - \tau_{4}^{(n+2)} t_{2} \lambda - ( (\tau_{1}^{(n+1)})_{0} + (\tau_{3}^{(n+1)})_{0}) \lambda + ( (\tau_{1}^{(n+1)})_{0} - (\tau_{3}^{(n+1)})_{0}) t_{2} + (\tau_{2}^{(n)})_{0},
$$
where $(\tau_{2}^{(n)})_{0}\in \mathbb{C}.$ Replacing the expressions of $\tau_{1}^{(0)}$ and $\tau_{3}^{(0)}$ provided by
(\ref{k-d}) in the second relation (\ref{dot-lambda}) we obtain that $\lambda$ satisfies the differential equation
$$
\dot{\lambda} =\frac{ - \tau_{4}^{(1)}\lambda +  k}{ \tau_{4}^{(1)} t_{2} + d}.
$$
Let $e:= \tau_{4}^{(1)}.$  
Solving this differential equation for $\lambda$ we obtain (\ref{H}). Finally, replacing $\lambda$ 
in the above expressions for $\tau_{1}^{(n)}$, $\tau_{2}^{(n)}$ and $\tau_{3}^{(n)}$ we conclude the proof of  claim i).\

Claim ii) follows from relations (\ref{derivate}) with $\dot{\lambda }=1.$ The condition $\partial_{2}^{3}\tau_{2}^{(n)} =0$
follows from $\partial_{2}^{2} \tau_{3}^{(n)} =0$  (from the third relation (\ref{derivate}) and  $\tau_{4}^{(n+1)}\in \mathbb{C}$).  
\end{proof}

\begin{cor}\label{initial-cazuri} i) Consider two  formal $(TE)$-structures $\nabla$ and $\tilde{\nabla}$
as in Proposition \ref{step-1} iii), with associated functions $b_{2}$ and $\tilde{b}_{2}$ respectively. 
Assume that there is a formal isomorphism $T$ between $\nabla$ and $\tilde{\nabla}$, given by Lemma \ref{auto-last} i).
In the notation of that lemma, 
\begin{equation}\label{conformal}
\tilde{b}_{2}^{(0)}(t_{2}) = b_{2}^{(0)}( \frac{k t_{2}}{ et_{2} + d}) \frac{( e t_{2} + d)^{2}}{kd}.
\end{equation}
In particular, $b_{2}^{(0)}$ is a formal gauge invariant of $\nabla$.\

ii) Let $\nabla$ be a formal $(TE)$-structure as in Proposition \ref{step-1} iii).  There is a formal
isomorphism which maps $\nabla$ to another $(TE)$-structure 
as in Proposition \ref{step-1} iii), with  associated function 
$\tilde{b}_{2}$, such that $\tilde{b}_{2}^{(0)}$  is of one of the following forms:
\begin{equation}\label{form-b2-0}
\tilde{b}_{2}^{(0)} =0,\ \tilde{b}_{2}^{(0)} = 1,\  \tilde{b}_{2}^{(0)} = \lambda t_{2} + 1,\ \tilde{ b}_{2}^{(0)} = \beta t_{2},\   
\tilde{b}_{2}^{(0)} = t_{2}^{2},
\end{equation}
where $\lambda \in \mathbb{C}$ and $\beta \in \mathbb{C}^{*}.$ 
\end{cor}

\begin{proof}  i) Relation (\ref{conformal}) follows by identifying  the coefficients of $z^{0}$ in relation 
(\ref{2.16}) with $B$, $\tilde{B}$ as in Proposition \ref{step-1} iii)  and $\tilde{T}$ given in Lemma \ref{auto-last} i).\

ii) Let $b_{2}$ be the associated function of $\nabla .$ Since $\partial_{2}^{3}b_{2}^{(0)} =0$ and $b_{2}^{(0)}$ is independent on $t_{1}$, we can write
$b_{2}^{(0)} = a t_{2}^{2} + b t_{2} + c$ for $a, b, c\in \mathbb{C}.$ 
Let $T$ be any formal automorphism of the $(T)$-structure $A_{1}=C_{1}$, $A_{2}=C_{2}$, as in Lemma 
\ref{auto-last} i),
and $\tilde{\nabla}:= T\cdot \nabla$, with associated function $\tilde{b}_{2}.$    From (\ref{conformal}),
\begin{equation}\label{suitable-choices}
\tilde{b}_{2}^{(0)}(t_{2})  = ( a \frac{k}{d} + b\frac{e}{d}+ c\frac{ e^{2}}{kd}) t_{2}^{2} 
+ (b + 2c \frac{e}{k}) t_{2} + c\frac{d}{k}.
\end{equation} 
Suitable choices of $k, d\in \mathbb{C}^{*}$ and $e\in \mathbb{C}$ in (\ref{suitable-choices}) 
show that $\tilde{b}_{2}^{(0)}$ can be reduced to one of the forms  (\ref{form-b2-0}). Any formal automorphism
of the $(T)$-structure  $A_{1}=C_{1}$, $A_{2}= C_{2}$, as in Lemma 
\ref{auto-last} i), with such constants $k$, $d$ and $e$, maps $\nabla$ to a formal
$(TE)$-structure with the required property.
\end{proof}

\begin{cor}\label{first-formal-b2}
i) Two formal  $(TE)$-structures $\nabla$ and $\tilde{\nabla}$ as in Proposition \ref{step-1} iii), 
with associated functions $b_{2}$ and $\tilde{b}_{2}$ respectively, such that $b_{2}^{(0)}$ and $\tilde{b}_{2}^{(0)}$
are distinct,  of the form  (\ref{form-b2-0}), are formally non-isomorphic, unless 
\begin{equation}\label{distincte-iso}
b_{2}^{(0)} = \lambda t_{2} +1,\ \tilde{b}_{2}^{(0)} = -\lambda t_{2} +1,\ \lambda \in \mathbb{C}^{*}.
\end{equation}
ii) A formal $(TE)$-structure $\nabla$ 
as in Proposition \ref{step-1} iii),  with associated  function $b_{2}$ such that  $b_{2}^{(0)} = \lambda t_{2}+1$, where $\lambda \in \mathbb{C}^{*}$, 
can be mapped 
by a formal   isomorphism to another $(TE)$-structure as in Proposition \ref{step-1} iii), 
with associated function $\tilde{b}_{2}$, 
such that  $\tilde{b}_{2}^{(0)} = - \lambda t_{2}+ 1.$   
\end{cor}

\begin{proof} i)  We  consider two $(TE)$-structures 
$\nabla$ and $\tilde{\nabla}$ as in Proposition \ref{step-1} iii), with associated functions $b_{2}$ and $\tilde{b}_{2}$, but such that $b_{2}^{(0)}$
and $\tilde{b}_{2}^{(0)}$ are not necessarily of the form (\ref{form-b2-0}). As 
$\partial_{2}^{3}b_{2}^{(0)}=\partial_{2}^{3}\tilde{b}_{2}^{(0)} =0$,  
\begin{equation}
b_{2}^{(0)} = a t_{2}^{2} + b t_{2} +c,\ \tilde{b}_{2}^{(0)} = \tilde{a}t_{2}^{2} + \tilde{b} t_{2} +\tilde{ c},
\end{equation}
for $a, b, c, \tilde{a}, \tilde{b}, \tilde{c} \in \mathbb{C}$. 
From (\ref{suitable-choices}),  if there is a formal isomorphism between $\nabla$ and
$\tilde{\nabla}$ then the system
\begin{equation}\label{system-abc}
a \frac{k}{d} + b\frac{e}{d}+ c\frac{e^{2}}{kd}=\tilde{a},\  
b + 2c \frac{e}{k}=\tilde{b},\  c\frac{d}{k}= \tilde{c}
\end{equation}
in the unknown constants $k, d\in \mathbb{C}^{*}$ and  $e\in \mathbb{C}$,   has a solution. 
When $b_{2}^{(0)}$ and $\tilde{b}_{2}^{(0)}$ are distinct,  of the form (\ref{form-b2-0}), a solution 
of (\ref{system-abc})  exists  only when
$b_{2}^{(0)}$ and $\tilde{b}_{2}^{(0)}$ are of the form (\ref{distincte-iso}). \

ii) Consider the automorphism  given by  Lemma \ref{auto-last} i), 
with $k = d =1$, 
$e=\lambda$,  $\tau_{4}^{(n)} =0$ for any $n\neq 1$,  
$(\tau_{1}^{(n)})_{0} =(\tau_{3}^{(n)})_{0}=0$ for any $n\geq 1$ and $(\tau_{2}^{(n)})_{0} =0$ for any $n\geq 0.$ 
It maps $\nabla$ to a $(TE)$-structure $\tilde{\nabla}$ with associated function $\tilde{b}_{2}$
and   $\tilde{b}_{2}^{(0)} = - \lambda t_{2} +1$.
\end{proof}

\begin{lem}\label{lemm1} Let $\nabla$ be a formal $(TE)$-structure  as in  Proposition \ref{step-1} iii).  Then 
$\nabla$ is formally isomorphic to  a $(TE)$-structure
with underlying $(T)$-structure $A_{1}=C_{1}$, $A_{2}=C_{2}$ and whose matrix $B$ either belongs to
the first five  lines in (\ref{b-2-0}) or is of one of the forms 
\begin{align}
\nonumber B  =&  ( - t_{1} + c +\alpha z) C_{1} + t_{2} (\lambda +\gamma t_{2} z^{\lambda} )C_{2}
-\frac{z}{2} ( \lambda + 1 +2\gamma t_{2} z^{\lambda})D\\
\nonumber &  -\gamma z^{\lambda +2} E,\ \lambda \in
 \mathbb{Z}_{\geq 1}\\
\label{pre-reduced} B= & ( -t_{1} + c+\alpha z) C_{1} +(\lambda t_{2} +\gamma z^{-\lambda}) C_{2}
-\frac{z}{2}(\lambda +1)D,\ \lambda \in {\mathbb Z}_{\leq -1},
\end{align}
where $c, \gamma \in \mathbb{C}.$
\end{lem}

\begin{proof} Let $\nabla$, $\tilde{\nabla}$ be two  formal   $(TE)$-structures
as in  Proposition \ref{step-1} iii),  
with constants $c$, $\alpha$ and associated function $b_{2} = \sum_{n\geq 0} b_{2}^{(n)} z^{n}$,  
respectively constants $\tilde{c}$, $\tilde{\alpha}$ and associated  function
$\tilde{b}_{2}=\sum_{n\geq 0} \tilde{b}_{2}^{ (n)} z^{n}$. Recall that
$b_{2}, \tilde{b}_{2}\in \mathbb{C}\{ t_{2}, z]]$ satisfy 
$\partial_{2}^{3}b_{2}=\partial_{2}^{3}\tilde{b}_{2}=0.$ 
We determine conditions on $b_{2}^{(n)}$ and $\tilde{b}_{2}^{(n)}$ such that
$\nabla$ and $\tilde{\nabla}$ are formally gauge isomorphic. 
This happens  if and only if there is a formal gauge automorphism 
$T$ 
of their underling  $(T)$-structure $A_{1} = C_{1}$, $A_{2} = C_{2}$,
such that
(\ref{2.18}), with matrices $B$ and $\tilde{B}$ of $\nabla$ and $\tilde{\nabla}$, 
is satisfied. The automorphism  $T$ is given by   Lemma \ref{auto-last} ii). 
Relation (\ref{2.18})  for $r=0$ is equivalent to $b_{2}^{(0)} =\tilde{b}_{2}^{(0)}$ 
(which we already know, from 
Corollary \ref{initial-cazuri})  
and
$c =\tilde{c}.$ For $r=1$ it is equivalent to $\alpha = \tilde{\alpha}$ 
(by identifying the coefficients of $C_{1}$) 
together with 
\begin{equation}\label{b2-0}
b_{2}^{(0)} \partial_{2} \tau_{2}^{(0)} - ( \partial_{2} b_{2}^{(0)}  + 1) \tau_{2}^{(0)} + \tau_{1}^{(0)} ( \tilde{b}_{2}^{(1)} - b_{2}^{(1)}) =0
\end{equation} 
(by identifying the coefficients of $C_{2}$). The coefficients of $D$ and $E$ give no relations
in (\ref{2.18}) with $r=1$.

We now consider relation (\ref{2.18}) with $r= n\geq 2.$ 
Identifying the coefficients of 
$C_{1}$ in this relation we obtain  
\begin{align}
\nonumber& (n-1) \tau_{1}^{(n-1)}-\frac{1}{4} \sum_{l=1}^{n} \partial_{2}^{2} \tau_{2}^{(n-l-2)} (b_{2}^{(l)} -\tilde{b}_{2}^{(l)} )\\ 
\nonumber&+\frac{1}{4} \sum_{l=2}^{n-1} \partial_{2} \tau_{2}^{(n-l-1)} \partial_{2} (b_{2}^{(l-1)} -\tilde{b}_{2}^{(l-1)}) \\
\label{0c1} &-\frac{1}{4}\sum_{l=2}^{n-1} \tau_{2}^{(n-l-1)} \partial_{2}^{2} (b_{2}^{(l-1)} -\tilde{b}_{2}^{(l-1)})=0
\end{align}
(with  the convention $\tau_{2}^{(l)}: =0$ for $l\in \mathbb{Z}_{\leq -1}$). 
Identifying the coefficients of $C_{2}$ in the same relation we obtain
\begin{align}
\nonumber& n  \tau_{2}^{(n-1)} - b_{2}^{(0)} \partial_{2} \tau_{2}^{(n-1)}+ \tau_{2}^{(n-1)} \partial_{2} b_{2}^{(0)}
+\sum_{l=1}^{n} \tau_{1}^{(n-l)} (b_{2}^{(l)} -\tilde{b}_{2}^{(l)})\\
\label{0c2}&-\frac{1}{2} \sum_{l=1}^{n} \partial_{2}\tau_{2}^{(n-l-1)} (b_{2}^{(l)} +\tilde{b}_{2}^{(l)}) 
+\frac{1}{2}\sum_{l=1}^{n-1} \tau_{2}^{(n-l-1)} \partial_{2} (b_{2}^{(l)} +\tilde{b}_{2}^{(l)}   ) =0 .
\end{align}
Identifying the coefficients of $D$ in the same relation we obtain:
\begin{align}
\nonumber&(n-1) \partial_{2} \tau_{2}^{(n-2)} -b_{2}^{(0)} \partial_{2}^{2} \tau_{2}^{(n-2)} -\frac{1}{2}
\sum_{l=1}^{n} \partial_{2}^{2} \tau_{2}^{(n-l-2)} (b_{2}^{(l)} +\tilde{b}_{2}^{(l)} ) \\
\nonumber&  +\sum_{l=2}^{n} \tau_{1}^{(n-l)} \partial_{2}(b_{2}^{(l-1)} -\tilde{b}_{2}^{ (l-1)})+
\frac{1}{2} \sum_{l=0}^{n-2} \tau_{2}^{(n-l-2)} \partial_{2}^{2} (b_{2}^{(l)} +\tilde{b}_{2}^{(l)})\\
\label{0d}&=0.
\end{align}
Identifying the coefficients of $E$ in the same relation  we obtain:
\begin{align}
\nonumber & (n-2)\partial_{2}^{2} \tau_{2}^{(n-3)} -\partial_{2} b_{2}^{(0)} \partial_{2}^{2} \tau_{2}^{(n-3)} -\frac{1}{2} \sum_{l=1}^{n-1}
\partial_{2}^{2} \tau_{2}^{(n-l-3)} \partial_{2} (b_{2}^{(l)} +\tilde{b}_{2}^{(l)}) \\
\nonumber&   +\frac{1}{2} \sum_{l=0}^{n-2} \partial_{2}\tau_{2}^{(n-l-3)} \partial_{2}^{2} (b_{2}^{(l)} +\tilde{b}_{2}^{(l)}) 
+\sum_{l=2}^{n} \tau_{1}^{(n-l)} \partial_{2}^{2} (b_{2}^{(l-2)} -\tilde{b}_{2}^{(l-2)}) \\
\label{0e}& =0.
\end{align}
A long but straightforward computation shows that for $n\geq 3$, relation (\ref{0d}) is the derivative of relation
(\ref{0c2}) with $n$ replaced by $n-1.$ Similarly, for $n\geq 4$, relation (\ref{0e}) is the second derivative of
(\ref{0c2}) with $n$ replaced by $n-2.$ Therefore, 
relations (\ref{0c1})-(\ref{0e}) are equivalent to relations (\ref{0c1}), (\ref{0c2}), together with relation
(\ref{0d}) with $n=2$ and relation (\ref{0e}) with $n=2,3$. But relation (\ref{0d}) with $n=2$ is the derivative of
(\ref{b2-0}), relation (\ref{0e}) with  $n=2$ follows from $b_{2}^{(0) }=\tilde{b}_{2}^{(0)}$ and 
relation (\ref{0e}) with $n=3$ is  the second derivative of (\ref{b2-0}).

To summarize: we proved that $\nabla$ and $\tilde{\nabla}$ are formally gauge isomorphic if and only if
$c=\tilde{c}$, $\alpha =\tilde{\alpha}$,  $b_{2}^{(0)} =\tilde{b}_{2}^{(0)}$ and relations (\ref{b2-0}), 
 (\ref{0c1}) and (\ref{0c2}) are
satisfied (the last two for any $n\geq 2$).

Using the above considerations,  we now prove our claim.
Let $\nabla$ be a $(TE)$-structure
as in Proposition \ref{step-1} iii). We aim to construct a $(TE)$-structure $\tilde{\nabla}$ formally isomorphic
to $\nabla$, as required by the lemma. 
From Corollary \ref{initial-cazuri} ii)
we can  assume, without loss of generality,  that $b_{2}^{(0)} =\tilde{b}_{2}^{(0)}$ is of one of  the forms  (\ref{form-b2-0}).
From  Corollary \ref{first-formal-b2} ii)), we
can further assume that the constant $\lambda$
from (\ref{form-b2-0}) belongs to  $(\mathbb{C}\setminus \mathbb{Z})\cup \mathbb{Z}_{\geq 0}$.
Let $\tau_{1}^{(0)}\in \mathbb{C}^{*}$. 
We choose   $\tilde{b}_{2}^{(1)}$ in a suitable way   
such that relation (\ref{b2-0}), considered as an equation in the unknown function $\tau_{2}^{(0)}$, has
a solution with $\partial_{2}^{3} \tau_{2}^{(0)} =0$. More precisely, 
when $b_{2}^{(0)} =0$  we  choose $\tilde{b}_{2}^{(1)} =0$ and $\tau_{2}^{(0)} = - \tau_{1}^{(0)} b_{2}^{(1)}$. 
When $b_{2}^{(0)}\neq 0$ 
we use Lemma \ref{third-der} 
and we  choose  $\tilde{b}_{2}^{(1)}$ 
as follows:
if  $b_{2}^{(0)} = \lambda t_{2}$ (with $\lambda \notin\{ -1, 1\}$)  or $b_{2}^{(0)} = \lambda t_{2} +1$ (with $\lambda \neq 1$) or $b_{2}^{(0)}
= t_{2}^{2}$ we choose $\tilde{b}_{2}^{(1)}= 0$;  if $b_{2}^{(0)} = t_{2}$ or $b_{2}^{(0)} = t_{2} +1$
we choose $\tilde{b}_{2}^{(1)} = (b_{2}^{(1)})_{2}  t_{2}^{2}$ where 
$ (b_{2}^{(1)})_{i} $  denotes the coefficient of $t_{2}^{i}$ in $b_{2}^{(1)}$;
if $b^{(0)}_{2} = - t_{2}$ we choose $\tilde{b}_{2}^{(1)} = (b_{2}^{(1)})_{0}$. 
Suppose now that $\tau_{1}^{(i)}$ and $\tau_{2}^{(i)}\in\mathbb{C}$ (with $i\leq n-1$) 
and $\tilde{b}_{2}^{(i)}$ (with $i\leq n$) are known
(and satisfy $\partial_{2}^{3} \tau_{2}^{(i)} = \partial_{2}^{3} \tilde{b}_{2}^{(i)} =0$).
Relation (\ref{0c1}) with $n$ replaced by $n+1$ determines $\tau_{1}^{(n)}$:
\begin{align}
\nonumber n\tau_{1}^{(n)} &= \frac{1}{4} \sum_{l=1}^{n+1} \partial_{2}^{2} \tau_{2}^{(n-l-1)} (b_{2}^{(l)} -\tilde{b}_{2}^{(l)} ) - \frac{1}{4} \sum_{l=2}^{n} \partial_{2} \tau_{2}^{(n-l)} \partial_{2} (b_{2}^{(l-1)} -\tilde{b}_{2}^{(l-1)}) \\
\label{0c1-prime} &+ \frac{1}{4}\sum_{l=2}^{n} \tau_{2}^{(n-l)} \partial_{2}^{2} (b_{2}^{(l-1)} -\tilde{b}_{2}^{(l-1)}).
\end{align}
We remark that $\tau_{1}^{(n)}\in \mathbb{C}:$  
straightforward computation, which uses $\partial_{2}^{3}b_{2}^{(l)} =\partial_{2}^{3}\tilde{b}_{2}^{(l)}= 
\partial_{2}^{3}\tau_{2}^{(i)}=0$   (for $l\leq n-1$ and $i\leq n-2$) 
shows that the right  hand side of (\ref{0c1-prime}) is constant.
Relation (\ref{0c2}) with $n$ replaced by $ n+1$ is
\begin{align}
\nonumber& (n+1) \tau_{2}^{(n)} +\partial_{2}b_{2}^{(0)} \tau_{2}^{(n)} - b_{2}^{(0)} \partial_{2} \tau_{2}^{(n)} + \tau_{1}^{(0)} (b_{2}^{(n+1) }- \tilde{b}_{2}^{(n+1)})\\
\nonumber& + \sum_{l=1}^{n} \tau_{1}^{(n+1-l)} (b_{2}^{(l)} -\tilde{b}_{2}^{(l)}) -\frac{1}{2}\sum_{l=1}^{n}
\partial_{2} \tau_{2}^{(n-l)}(b_{2}^{(l)} +\tilde{b}_{2}^{(l)}) \\
\label{b2-0-ind}& +\frac{1}{2} \sum_{l=1}^{n} \tau_{2}^{(n-l)} \partial_{2} ( b_{2}^{(l)} +\tilde{b}_{2}^{(l)}) =0.
\end{align}
The second and third lines from the left  hand side of (\ref{b2-0-ind}) are known. 
When $b_{2}^{(0)}\neq 0$ we choose as before (using Lemma \ref{third-der}) 
$\tilde{b}_{2}^{(n+1)}$   such that  
(\ref{b2-0-ind})  has a solution  $\tau_{2}^{(n)}$
and $\partial_{2}^{3} \tilde{b}_{2}^{(n+1)} = \partial_{2}^{3} \tau_{2}^{(n)} =0.$ 
When $b_{2}^{(0)} =0$  we  choose $\tilde{b}_{2}^{(n+1)} =0$
and $\tau_{2}^{(n)}$ to satisfy (\ref{b2-0-ind}) (with $\tilde{b}_{2}^{(n+1)}=0$)
and   $\partial_{2}^{3} \tau_{2}^{(n)} =0.$ 
Using an induction procedure 
we define an automorphism $T$ of the $(T)$-structure $A_{1}=C_{1}$, $A_{2} = C_{2}$, 
such that  the associated function  
$\tilde{b}_{2}$ of  $\tilde{\nabla}:= T\cdot \nabla$ is of one of the following forms:

$\tilde{b}_{2}=0$,  $\tilde{ b}_{2} = t_{2}^{2}$,  $\tilde{b}_{2} = \lambda t_{2}$  ($\lambda \notin \mathbb{Z}\setminus \{ 0\}$),
$\tilde{b}_{2} = \lambda t_{2} +1$ ($\lambda \notin\mathbb{Z}\setminus \{ 0\}$),  $\tilde{b}_{2} = \lambda t_{2} +1 +\gamma t_{2}^{2}z^{\lambda}$
($\lambda \in \mathbb{Z}_{\geq 1},\ \gamma \in \mathbb{C}$), 
$\tilde{b}_{2} = \lambda t_{2} +\gamma t_{2}^{2} z^{\lambda}$  ($\lambda \in \mathbb{Z}_{\geq 1}$, $\gamma \in \mathbb{C}$),
$\tilde{b}_{2} = \lambda t_{2} + \gamma z^{-\lambda}$  ($\lambda \in \mathbb{Z}_{\leq -1}$, $\gamma \in \mathbb{C}$).\

The first five forms above  of $\tilde{b}_{2}$ give the  $(TE)$-structures with $A_{1}=C_{1}$, $A_{2} = C_{2}$ and matrix
$B$ as in  the first 
five  lines from
(\ref{b-2-0}). The last two forms of $\tilde{b}_{2}$ give 
the $(TE)$-structures with $A_{1}=C_{1}$, $A_{2}=C_{2}$ and matrix $B$ of either of the two  forms  
(\ref{pre-reduced}).
\end{proof}

The next lemma concludes the proof of Theorem \ref{te}. It shows that  the $(TE)$-structures from Lemma \ref{lemm1}, 
with matrices $B$ given in (\ref{pre-reduced}), are formally isomorphic to the last four classes of
$(TE)$-structures from Theorem \ref{te} iii).

\begin{lem}\label{lem-pre-red} The $(TE)$-structures with $A_{1}=C_{1}$, $A_{2}=C_{2}$ and matrices $B$ given
by (\ref{pre-reduced}) are formally isomorphic to  the $(TE)$-structures of the same form
(with the same constants $c$, $\alpha$, $\lambda$), but with $\gamma\in \{ 0,1\}.$ 
\end{lem}

\begin{proof} Let $T^{[1]}$ be a formal automorphism of the $(T)$-structure $A_{1}=C_{1}$, $A_{2}=C_{2}$ given by
Lemma \ref{auto-last} i), with $\tau_{4}^{(n)}= (\tau_{2}^{(n)})_{0}=0$ ($n\geq 0$),  
$(\tau_{1}^{(n)})_{0}= (\tau_{3}^{(n)})_{0}=0$ ($n\geq 1$), $k, d\in \mathbb{C}^{*}$ and $e=0.$
It  maps the  $(TE)$-structures
from Lemma \ref{lemm1}, with matrices  $B$ given by 
(\ref{pre-reduced}),  to  $(TE)$-structures of the same form, with the same $c, \alpha , \lambda \in \mathbb{C}$ 
but with $\gamma$ replaced by $\tilde{\gamma}:= \frac{k}{d}\gamma .$ 
When $\gamma \neq 0$ we can choose $k, d\in \mathbb{C}^{*}$ such that
$\tilde{\gamma }=1.$
\end{proof}

\subsubsection{Proof of Theorem \ref{te-class}}

Consider two distinct $(TE)$-structures $\nabla$ and $\tilde{\nabla}$ from Theorem \ref{te},
and suppose that they are formally isomorphic.  
Then their underlying $(T)$-structures are also formally isomorphic. 
From Theorem 21 ii) of \cite{DH-AMPA}, there are two cases, namely:   
a)  $\nabla$ and $\tilde{\nabla}$ are 
as in Theorem \ref{te} i); b)  $\nabla$ and $\tilde{\nabla}$ 
are as in Theorem \ref{te} iii).  
The next lemma treats the first possibility.

\begin{lem}\label{further-reduction} Consider two distinct  $(TE)$-structures 
as in Theorem \ref{te} i), with constants $c$, $\alpha$, $c_{0}$, respectively $\tilde{c}$, $\tilde{\alpha}$, $\tilde{c}_{0}.$ 
Then $\nabla$ and $\tilde{\nabla}$ are formally isomorphic 
if and only if  $\tilde{c} =c$, $\tilde{\alpha } =\alpha$ and $\tilde{c}_{0} = - c_{0}$.
\end{lem}

\begin{proof} Let ${T}$ be a formal automorphism 
the $(T)$-structure $A_{1} = C_{1}$, $A_{2} = C_{2} + zE$. From Lemma \ref{stat1}, $T$ is not a formal gauge 
automorphism. 
From Lemma \ref{aut-lema1}, $T$  covers the map $h(t_{1}, t_{2}) = (t_{1}, - t_{2})$
and $\tilde{T} =  T\circ h = T$ is of the form (\ref{auto-t-t}).
As $\tilde{B}   = \tilde{B} \circ h$, relation (\ref{2.16})  becomes 
\begin{equation}\label{tilde-z-b}
z^{2}  \partial_{z} T + \tilde{B} T - TB=0.
\end{equation}
By identifying the coefficients of $z^{0}$ and $z$  in (\ref{tilde-z-b})  we obtain 
that $\tilde{c} = c$, $\tilde{\alpha} =\alpha$ and $\tilde{c}_{0} = - c_{0}$.
Moreover, if these relations are satisfied then  the  isomorphism $T (z, t_{1}, t_{2}) : = D$  which covers $h$ maps  $\nabla$ to $\tilde{\nabla }.$ 
\end{proof}

It remains to study the second case. We begin with the next simple lemma.

\begin{lem}\label{X} Let $\nabla$ be a $(TE)$-structure as in Theorem \ref{te} iii). Then the constants $c$ and $\alpha$ 
are formal invariants. 
\end{lem}

\begin{proof}  We notice that $c$ and $\alpha$  remain unchanged  under the 
automorphisms 
from Lemma \ref{auto-last} i), for which  $\tau_{4}^{(n)} =0$ for any $n\neq 1$, $(\tau_{1}^{(n)})_{0}= (\tau_{3}^{(n)})_{0} =0$ for any $n\geq 1$,
$(\tau_{2}^{(n)})_{0} =0$ for any $n\geq 0$ 
(and $k, d\in \mathbb{C}^{*}$, $e\in \mathbb{C}$ arbitrary).
These  automorphisms cover all automorphisms of $\mathcal N_{2}$ which
lift to automorphisms of the $(T)$-structure $A_{1} = C_{1}$, $A_{2} = C_{2}$. They can be used
to reduce the statement  we need to prove  to showing that $c$ and $\alpha$ are formal gauge invariant This was shown
in  the proof of Lemma \ref{lemm1}.
\end{proof}

\begin{lem} Consider two distinct $(TE)$-structures  $\nabla$ and $\tilde{\nabla}$ as in 
Theorem \ref{te} iii).
They are formally isomorphic if and only if their matrices $B$ and $\tilde{B}$ 
are of the fourth  form in  (\ref{b-2-0}), with constants $c$, $\alpha$, $\lambda$ and
$\tilde{c}$, $\tilde{\alpha}$, $\tilde{\lambda }$,  such that
$\tilde{c} = c$, $\tilde{\alpha} =\alpha$ and $\tilde{\lambda}  = -\lambda .$
\end{lem}

\begin{proof} Assume that $\nabla$ and $\tilde{\nabla}$ are formally isomorphic.
Then, from Lemma \ref{X}, 
$c=\tilde{c}$ and $\alpha =\tilde{\alpha}$.
Using  Corollary \ref{first-formal-b2} i), 
we deduce (exchanging $\nabla$ with $\tilde{\nabla}$ if necessary) that  one of the following cases holds:  I)
$\nabla$ and $\tilde{\nabla}$  belong to the fourth class in Theorem \ref{te} iii)
and $\tilde{\lambda } = - \lambda$; II)  
$\nabla$ and $\tilde{\nabla}$   belong to the fifth class in Theorem \ref{te} iii) 
and $\tilde{\lambda} = \lambda$, $\tilde{\gamma}\neq\gamma$;   III)  $\nabla$ belongs  to 
the  sixth class and $\tilde{\nabla}$ belongs to the seventh class
in Theorem \ref{te} iii) 
and $\tilde{\lambda} =\lambda$;  IV)  $\nabla$ 
belongs to the eighth class and $\tilde{\nabla}$  to the  nineth class  
in Theorem \ref{te} iii) 
and $\tilde{\lambda} =\lambda$. 
In case  I) $\nabla$ and $\tilde{\nabla}$ are formally isomorphic by  
the isomorphism  $T$ used in  Corollary \ref{first-formal-b2} ii). 
It turns out that in the remaining cases $\nabla$ and $\tilde{\nabla}$ are in fact formally non-isomorphic. 
Let us sketch the argument for  case III). Assume, by contradiction, that $\nabla$ and $\tilde{\nabla}$
are formally isomorphic. 
Then the $B$-matrices of $\nabla$, $\tilde{\nabla}$  are of  the first form   (\ref{pre-reduced}), with  
$c=\tilde{c}$, $\alpha =\tilde{\alpha}$, $\lambda =\tilde{\lambda}$, 
$\gamma =1$ and  $\tilde{\gamma}=0.$ 
Since $\tilde{b}_{2}^{(0)}=b_{2}^{(0)} = \lambda t_{2}$, from  (\ref{system-abc}) with 
$a=\tilde{a} =0$, $b=\tilde{b} =\lambda$ and $c =\tilde{c} =0$ 
we  deduce that $e=0$ and $T$ covers a map of the form $h(t_{1}, t_{2} ) =( t_{1}, \frac{k}{d}t_{2})$,
with $k,d\in \mathbb{C}^{*}.$  
The automorphism   $T^{[1]}$ used in the proof of Lemma \ref{lem-pre-red} maps $\nabla$ to a $(TE)$-structure
$\nabla^{[1]}$ with $A_{1}^{[1]} = C_{1}$, $A^{[1]}_{2} = C_{2}$ and matrix $B^{[1]}$ of the first form (\ref{pre-reduced})  
with $\gamma^{[1]}  = \frac{k}{d}.$ Since $\nabla$ and $\tilde{\nabla}$ are formally isomorphic,  
$\nabla^{[1]}$ and $\tilde{\nabla}$ are formally gauge isomorphic (and are both
of the first form 
(\ref{pre-reduced})).
Going through the computations
of Lemma \ref{lemm1} and using Lemma \ref{third-der}  
we obtain  that $\nabla^{[1]} = \tilde{\nabla}$ which is a contradiction (as $\gamma^{[1]} \neq 0$ while $\tilde{\gamma} =0$).
\end{proof}

\section{Holomorphic classification of $(TE)$-structures}\label{holom-section}

\subsection{Restriction of $(TE)$-structures at the origin; holomorphic classification in the non-elementary case}\label{restriction-section}

By an {\it elementary model} we mean 
a meromorphic connection $\nabla^{0}$ 
on the germ $({\mathcal O}_{(\mathbb{C}, 0)})^{2}$ 
with connection form  
$\Omega^{0}=\frac{1}{z^{2}} B^{0} dz= \frac{1}{z^{2}} \sum_{k\geq 0} B_{0}^{(k)} z^{k}dz$ (where $B_{0}^{(k)}\in M_{2\times 2}(\mathbb{C})$),  such that 
$\nabla^{1}:={\mathcal E}^{\frac{\mathrm{tr}\,  (B_{0}^{(0)})}{2z}}\otimes \nabla^{0}$ is regular singular (we denote by $\mathcal E^{\rho }$ the 
connection in rank one with connection form $d\rho$).  
Obviously, $\nabla^{0}$ is an elementary model if and only if
 ${\mathcal E}^{\frac{\mathrm{tr}\, (B^{0})}{2z}}\otimes \nabla^{0}$ is regular singular.
The property of  a meromorphic connection to be an elementary model is invariant
under holomorphic isomorphisms: 
if $\tilde{\nabla}^{0}$, with connection form 
$\tilde{\Omega}^{0}=\frac{1}{z^{2}} \tilde{B}^{0} dz= 
\frac{1}{z^{2}} \sum_{k\geq 0} \tilde{B}_{0}^{(k)} z^{k}dz$ ($\tilde{B}_{0}^{(k)} \in 
M_{2\times 2}(\mathbb{C})$) is
isomorphic  to $\nabla^{0}$ by means of an
holomorphic isomorphism  $T^{0} = \sum_{k\geq 0} T_{0}^{(k)}z^{k}$, then 
$\mathrm{tr}\,  (\tilde{B}_{0}^{(0)} )= \mathrm{tr}\,  (B_{0}^{(0)})$ and 
$T^{0}$ is an  isomorphism also between $\nabla^{1}$ and $\tilde{\nabla}^{1}$ 
(the latter defined as $\nabla^{1}$ starting with $\tilde{\nabla}^{0}$ instead of
$\nabla^{0}$).

\begin{defn} A $(TE)$-structure $\nabla$ over $\mathcal N_{2}$ is called  elementary
if its restriction  to the slice $(\mathbb{C},0)\times \{ 0\}\subset (\mathbb{C}, 0)\times \mathcal N_{2}$  
is an elementary model. A $(TE)$-structure which is not elementary is called non-elementary.
\end{defn}

Our aim in this section is to prove the next proposition
(for Malgrange universal connections and Birkhoff normal forms, see appendix).

\begin{prop}\label{conceptual} Let $\nabla$ be a $(TE)$-structure 
in pre-normal form, determined by $(f, b_{2}, c, \alpha )$.\

i) Then $\nabla$ is  elementary if and only if $f(0, 0) b_{2}(0, 0) =0.$\

ii)  If $\nabla$ is elementary then 
$\nabla$ is holomorphically isomorphic to its formal normal form(s).\

iii)  If $\nabla$ is non-elementary then
$\nabla$ is holomorphically isomorphic to the Malgrange universal deformation of a
meromorphic connection in Birkhoff normal form,
with residue a regular endomorphism.
\end{prop}

We divide the proof of the above proposition into several steps. 
In the setting of Proposition \ref{conceptual},  
let   $\nabla^{\mathrm{restr}}$ be the restriction of $\nabla$  to the slice 
$\Delta\times \{ 0\}$ 
(where $\Delta$ is  a small disc around the origin in $\mathbb{C}$) 
and  $\eta := b_{2}$, $\lambda := \partial_{2}b_{2}$,  $\beta := \partial^{2}_{2}b_{2}$
and $\gamma : = f$, 
all restricted to this slice. 
They are functions on $z$ only and are holomorphic.  
From (\ref{B}), (\ref{b3-4}), the  connection form of $\nabla^{\mathrm{restr}}$  is
\begin{equation}\label{conn-restr}
\Omega^{\mathrm{restr}} = \frac{1}{z^{2}} \left( (c  + z\alpha )C_{1} + \eta  C_{2} -\frac{z(\lambda + 1)}{2} D + z ( -\frac{z\beta }{2} + \gamma\eta  ) E \right) dz.
\end{equation}

\begin{lem}\label{lem-elem}The connection $\nabla^{\mathrm{restr}}$ is  
an elementary model if and only if $\eta (0)\gamma (0) =0$.
\end{lem}

\begin{proof} 
We need to show that $\nabla^{\mathrm{restr}}$, with $c= \alpha =0$, is regular singular.
Assume from now on that $c=\alpha =0.$
When $\eta (0) =0$,  $\Omega^{\mathrm{restr}}$ has a logarithmic pole and 
the regular singularity  of $\nabla^{\mathrm{restr}}$ is obvious. 
Assume assume that $\eta (0) \neq 0.$   
Let $\{ v_{1}, v_{2} \}$ be the standard basis of $({\mathcal O}_{(\mathbb{C} , 0)})^{2}$, so that   
\begin{align}
\nonumber& \nabla^{\mathrm{restr} }_{\partial z} (v_{1})=( \Omega^{\mathrm{restr}}_{\partial z} )_{11} v_{1} + ( \Omega^{\mathrm{restr}}_{\partial z} )_{21} v_{2} \\
\label{restr-v}&\nabla^{\mathrm{restr} }_{\partial z} (v_{2})=( \Omega^{\mathrm{restr}}_{\partial z} )_{12} v_{1} + ( \Omega^{\mathrm{restr}}_{\partial z} )_{22} v_{2} .
\end{align}
Using the definitions of the matrices $C_{2}$, $D$ and $E$,  
we rewrite (\ref{restr-v}) as
\begin{align}
\nonumber& \nabla^{\mathrm{restr}}_{\partial z}(v_{1}) = -\frac{\lambda + 1}{2z} v_{1} +\frac{\eta }{z^{2}} v_{2}\\
\label{con-ex}& \nabla^{\mathrm{restr}}_{\partial z}(v_{2}) = \left( -\frac{\beta }{2} +\frac{\eta \gamma }{z}\right) v_{1} +\frac{\lambda + 1}{2z} v_{2}.
\end{align}
Since $\eta (0)\neq 0$, $v_{1}$ is a cyclic vector (see appendix).   
Let $\tilde{v}_{2}:=\nabla^{\mathrm{restr}}_{\partial_{z}} (v_{1})$. Then 
\begin{align}
\nonumber\nabla^{\mathrm{restr}}_{\partial z} (\tilde{v}_{2}) &= \left( -\partial_{z} ( \frac{\lambda +1}{2z} ) +\frac{\eta }{z^{2}} (-\frac{\beta}{2} +\frac{\eta \gamma }{z}) 
+\frac{(\lambda +1)}{2\eta } ( \frac{\dot{\eta}}{z} -\frac{2\eta }{z^{2}}) +\frac{ (\lambda + 1)^{2} }{4z^{2}} \right) v_{1}\\
& + \frac{ 1}{\eta} ( \dot{\eta} - \frac{2\eta}{z}) \tilde{v}_{2}.
\end{align}
The valuation of the coefficient of $\tilde{v}_{2}$ 
in $\nabla^{\mathrm{restr}}_{\partial z} (\tilde{v}_{2})$
is equal to $-1$,
while the valuation of the coefficient of $v_{1}$ is greater or equal to $-2$ if and only if $\gamma (0) =0$ (we used that $\eta (0)\neq 0$).  From the  Fuchs criterion (see appendix), we obtain our claim.
\end{proof}

\begin{cor}\label{cor-class-elem}  i) The property of a $(TE)$-structure over $\mathcal N_{2}$ to be elementary is a formal
invariant. Any formal isomorphism between elementary $(TE)$-structures is holomorphic.

ii) Any elementary $(TE)$-structure over $\mathcal N_{2}$ is holomorphically isomorphic to
its formal normal form(s).
\end{cor}

\begin{proof}
Let $\nabla$ and $\tilde{\nabla}$ be two formally isomorphic $(TE)$-structures 
in pre-normal form, determined by $(f, b_{2}, c, \alpha )$ and
$(\tilde{f}, \tilde{b}_{2},\tilde{ c}, \tilde{\alpha} )$. 
From Theorem 19 i) of \cite{DH-AMPA} 
we know that $f(0, 0) = 0$ if and only if $\tilde{f}(0, 0) =0$. Since 
$\nabla$ and $\tilde{\nabla}$ are formally isomorphic, 
$B^{(0)}$ and $\tilde{B}^{(0)}$ are conjugated, which implies that $b_{2}(0, 0) = 0$ if and only if $\tilde{b}_{2}(0, 0)=0.$  We proved that
$\nabla$ is elementary if and only if $\tilde{\nabla}$ is elementary, i.e. being elementary is a formal invariant.
Assume now that $\nabla$ and $\tilde{\nabla}$ are elementary.  Let  $T=\sum_{\geq 0} T^{(k)} z^{k}$
be a formal isomorphism between them and $T^{0} $  be the restriction of
$T$ to $\Delta \times \{ 0\} .$ As explained at the beginning of this section,
$T^{0}$ is an isomorphism between the connections
$(\nabla^{\mathrm{restr}})^{1}$ and $(\tilde{\nabla}^{\mathrm{restr}})^{1}$. Since these are regular 
singular,  we deduce that $T^{0}$ is holomorphic. 
Claim i)  is concluded by Theorem 5.6 of \cite{DH-dubrovin}. Claim ii) follows trivially from claim i). 
\end{proof}

Claims i) and ii) from Proposition \ref{conceptual} are proved. It remains to prove claim iii).  
We do this in several lemmas. Recall  the definition of $\eta$, $\lambda$, $\beta$ and  $\gamma$ stated before Lemma \ref{lem-elem}.

\begin{lem}\label{BNF}  Let $\nabla$ be a $(TE)$-structure in pre-normal form,
determined by $(f, b_{2}, c, \alpha ).$ 
If
$\eta (0)\gamma (0)  \neq 0$,  then $\nabla^{\mathrm{restr}}$ can be put in   Birkhoff normal form.
\end{lem}

\begin{proof} In the standard basis $\{ v_{1}, v_{2} \}$
of $({\mathcal O}_{(\mathbb{C}, 0)})^{2}$, 
$\nabla^{\mathrm{restr}}$ is given by 
\begin{align}
\nonumber& \nabla^{\mathrm{restr}}_{\partial z}(v_{1}) = (\frac{c +\alpha z}{z^{2}} -\frac{\lambda + 1}{2z}) 
v_{1} +\frac{\eta }{z^{2}} v_{2}\\
\label{lambda}& \nabla^{\mathrm{restr}}_{\partial z}(v_{2}) = \left( -\frac{\beta }{2} +\frac{\eta \gamma }{z}\right) v_{1} +
(  \frac{ c+ \alpha z}{z^{2}}+ \frac{\lambda + 1}{2z} ) v_{2}.
\end{align}
(Remark that these relations with $c = \alpha =0$   reduce to relations (\ref{con-ex})). 
We apply the irreducibility criterion as stated in Lemma \ref{lem-ired} (see appendix).
Suppose, by absurd,  that there is a section $w$ such that 
\begin{equation}\label{irred-cond}
\nabla^{\mathrm{restr}}_{\partial_{z}}(w) = hw,
\end{equation}
for a function $h\in \textit{\textbf{k}}.$   Since $\eta (0)\neq 0$, the first relation (\ref{lambda})
shows that $w$ cannot be a multiple of $v_{1}.$ 
Rescaling $w$ if necessary, we can assume that 
$w= gv_{1} + v_{2}$, where $g\in \textit{\textbf{k}}$.
A straightforward computation 
which uses (\ref{lambda}) shows that (\ref{irred-cond}) is equivalent to
$$
h =\frac{g\eta}{z^{2}} +\frac{\lambda +1}{2z} +\frac{ c + \alpha z}{z^{2}}
$$
and 
\begin{equation}\label{b1}
z^{2} \dot{g} - ( (\lambda +1) z + \eta g) g-\frac{\beta z^{2}}{2} + \eta\gamma z=0.
\end{equation}
We will show that (\ref{b1}) leads to a contradiction.
Since $g\in  \textit{\textbf{k}}$, we can write it as $g(z) = z^{k} r(z)$ for $k\in \mathbb{Z}$ and 
$r\in \mathbb{C} \{ z\}$ a unit. Relation (\ref{b1}) is equivalent to
\begin{equation}\label{k-z-g}
k z^{k+1}r(z) + z^{k+2} \dot{r}(z) - z^{k+1} r(z) (\lambda (z) +1) - z^{2k} \eta (z) r(z)^{2} -\frac{\beta (z)z^{2}}{2}
+\eta (z) \gamma (z) z =0.
\end{equation}
If $k\leq 0$ then, multiplying the above relation by $z^{-2k}$, we obtain
\begin{align*}
\nonumber&k z^{-k+1} r(z) + z^{-k+2}\dot{r}(z) - z^{-k+1} r(z) (\lambda (z) +1) - \eta (z) r(z)^{2} -\frac{\beta (z)}{2} z^{-2k+2}\\
\nonumber& +\eta (z) \gamma (z) z^{-2k+1}=0.
\end{align*}
All terms, except $\eta(z) r(z)^{2}$, contain 
$z$ as a factor. Since $r, \lambda , \eta , \beta , \gamma \in \mathbb{C} \{ z\}$
and  $\eta$, $r$,  are  units 
we obtain a contradiction.
 If $k\geq 1$ the argument is similar: we multiply (\ref{k-z-g}) by $z^{-1}$ and we use that 
$\eta$, $ \gamma$ are units in $\mathbb{C} \{ z\} .$ 
\end{proof}

To conclude Proposition \ref{conceptual} we notice that
if $b_{2}(0, 0)\neq 0$ then  
 the 'residue'
$c C_{1}+ \eta (0) C_{2}$  of the restriction $\nabla^{\mathrm{restr}}$ 
 of $\nabla$ to $\Delta \times \{ 0\}$ is a regular endomorphism. Therefore, the Birkhoff normal form  provided 
by Lemma \ref{BNF} also has a regular residue and  admits a (unique, up to holomorphic isomorphisms)
Malgrange universal deformation. The latter is  (holomorphically) isomorphic
to $\nabla .$

\subsection{Holomorphic classification: non-elementary case}\label{non-elem-sect}

The holomorphic classification of elementary $(TE)$-structures follows from 
Corollary \ref{cor-class-elem} ii): the formal normal forms for
elementary $(TE)$-structures coincide with the holomorphic normal forms. 
It remains to determine the holomorphic normal forms for  non-elementary $(TE)$-structures.
This will be done in the next sections.

\subsubsection{Classification of non-elementary models in Birkhoff normal form}

\begin{lem}\label{char-non-elem}
i) Any non-elementary $(TE)$-structure $\nabla$ is isomorphic to the Malgrange universal deformation 
of a connection $\nabla^{B_{0}^{0}, B_{\infty}}$  in Birkhoff normal form, with connection form
\begin{equation}\label{conn-form-elem}
\Omega^{B_{0}^{o}, B_{\infty}} = \frac{1}{z^{2}} ( B_{0}^{o} + B_{\infty} z)dz,
\end{equation}
where  
\begin{equation}\label{matrices}
B_{0}^{o} =  \left(\begin{tabular}{cc}
$c$ & $0$\\
$c_{0}$ & $c$
\end{tabular}\right) ,\  B_{\infty}= \left(\begin{tabular}{cc}
$ B^{\infty}_{11}$ & $B^{\infty}_{12}$\\
$B^{\infty}_{21}$ & $B^{\infty}_{22}$
\end{tabular}\right)
\end{equation}
with $B^{\infty}_{ij}\in \mathbb{C}$, $c, c_{0}\in \mathbb{C}$ and $c_{0} B^{\infty}_{12}\neq 0$.\

ii) Two non-elementary $(TE)$-structures $\nabla$, $\tilde{\nabla}$ are isomorphic if and only if  
the corresponding  connections in Birkhoff normal form $\nabla^{B_{0}^{0}, B_{\infty}}$ and
$\nabla^{\tilde{B}_{0}^{0}, \tilde{B}_{\infty}}$ are isomorphic.  
\end{lem}

\begin{proof} 
From Lemma \ref{BNF}, we know  that the restriction $\nabla^{\mathrm{restr}}$ of $\nabla$ to the origin of $\mathcal N_{2}$ can be put in  Birkhoff normal form. Let $\nabla^{B_{0}^{o}, B_{\infty}}$ be a connection in Birkhoff normal form,
with connection form given by 
(\ref{conn-form-elem}) (for some matrices $B_{0}^{o}, B_{\infty}\in M_{2\times 2}(\mathbb{C})$),  
isomorphic to $\nabla^{\mathrm{restr}}.$ 
The connection $\nabla^{B_{0}^{o}, B_{\infty}}$  has two properties:
it   is not an elementary model and its  'residue' 
$B_{0}^{o}$ 
is a regular endomorphism, with only one eigenvalue (these two properties are satisfied by  $\nabla^{\mathrm{restr}}$
and are invariant under holomorphic   isomorphisms). 
From the second property,  the 'residue' $B_{0}^{o}$ is as in (\ref{matrices}), with $c_{0}\neq 0.$ 
A direct check (using e.g.  the Fuchs criterion), shows that $\nabla^{B_{0}^{o}, B_{\infty}}$ is not an elementary model
if and 
only if $B^{\infty}_{12}\neq 0.$ This proves claim i).  Claim ii) follows from the unicity of Malgrange universal deformations. 
\end{proof}

In order to use Lemma \ref{char-non-elem}  for the classification of non-elementary $(TE)$-structures, we need to
 establish  when two meromorphic connections in Birkhoff normal form,  as in Lemma \ref{char-non-elem},
are isomorphic. We start  with the next lemma.

\begin{lem}\label{conj-ct} i)  Any meromorphic connection $\nabla^{B_{0}^{o}, B_{\infty}}$
in Birkhoff normal form, with connection form given by (\ref{conn-form-elem})  where 
\begin{equation}\label{restriction-cl}
B_{0}^{o} = c C_{1} + c_{0} C_{2},\ B_{\infty} = \alpha C_{1} + c_{1} C_{2} + y D + fE,
\end{equation}
and $c_{0} f\neq 0$, can be mapped, by means of a constant isomorphism, to a connection 
$\nabla^{\tilde{B}_{0}^{0}, \tilde{B}_{\infty}}$ 
in Birkhoff normal form with 
\begin{equation}\label{almost-fin}
\tilde{B}_{0}^{o} = c C_{1} + c_{0} C_{2},\ \tilde{B}_{\infty} = \alpha C_{1} + c_{1} C_{2} -\frac{1}{4} D + c_{0}E,
\end{equation}
where the constants $c$ and $\alpha$ are the same as in (\ref{restriction-cl}), 
$c_{0}$ and $c_{1}$ are possibly different  from those in (\ref{restriction-cl})  and $c_{0}\neq 0$.\

ii)    The constant isomorphism $T:= \mathrm{diag}(1, -1)$ maps the connection in Birkhoff normal 
form $\nabla^{\tilde{B}_{0}^{0}, \tilde{B}_{\infty}}$, with matrices 
$\tilde{B}_{0}^{o}$, $\tilde{B}_{\infty}$ given in (\ref{almost-fin}),
to a connection of the same form
(\ref{almost-fin}), with constants $(\tilde{c}, \tilde{c}_{0}, \tilde{\alpha}, 
\tilde{c}_{1})$ satisfying  $ \tilde{c} = c$, $\tilde{\alpha} =\alpha$, $\tilde{c}_{0} = - c_{0}$ and $\tilde{c}_{1} = - c_{1}.$ 
\end{lem}

\begin{proof} i)  First  we map   $\nabla^{B_{0}^{o}, B_{\infty}}$ to  a  connection of the same form 
(\ref{conn-form-elem}),  (\ref{restriction-cl}),  with the coefficient of $D$ equal to $-\frac{1}{4}.$  This is realized using  the  constant isomorphism 
$T_{1}:= C_{1} -\frac{1}{f} (y+\frac{1}{4}) C_{2}$. Therefore, without loss of generality we may (and will) assume
that $\nabla^{B_{0}^{o}, B_{\infty}}$ is given by
(\ref{conn-form-elem}),  (\ref{restriction-cl}),  with $y=-\frac{1}{4}.$ Under this assumption, if  $c_{0} \neq f$ in (\ref{restriction-cl}), let $\tilde{c}_{0}$ such that 
$(\tilde{c}_{0})^{2} = c_{0}f$ and  $T_{2}:= -2 \mathrm{diag} ( \frac{\tilde{c}_{0}}{c_{0} -\tilde{c}_{0}},
\frac{{c}_{0}}{c_{0} -\tilde{c}_{0}}).$ The   isomorphism $T_{2}$ maps  
$\nabla^{ B_{0}^{o}, B_{\infty}}$  to the connection $\nabla^{\tilde{B}^{o}_{0}, \tilde{B}_{\infty}}$ 
with
$$
\tilde{B}_{0}^{o} = c C_{1} + \tilde{c}_{0} C_{2},\ \tilde{B}_{\infty} = \alpha C_{1} + \frac{c_{1}\tilde{c}_{0}}{c_{0}} C_{2}
-\frac{1}{4} D +\tilde{c}_{0} E.
$$
This proves claim i). Claim ii) follows from a direct check.
\end{proof}

\begin{prop}\label{classif-Birk} Consider two distinct connections $\nabla$ and $\tilde{\nabla}$  in Birkhoff normal form
(\ref{conn-form-elem}), with matrices $B_{0}^{o}$, $B_{\infty}$, respectively  $\tilde{B}_{0}^{o}$, $\tilde{B}_{\infty}$ as in
(\ref{almost-fin}), with constants $c, \alpha , c_{0}, c_{1}$ and, respectively   $\tilde{c},\tilde{\alpha},  \tilde{c}_{0},\tilde{c}_{1}$. Assume that 
$c_{0} \tilde{c}_{0}\neq 0$.\

i) If $\nabla$ and $\tilde{\nabla}$ are formally isomorphic, then $ c=\tilde{c}$, $\alpha = \tilde{\alpha}$ and 
$c_{0}  = \epsilon \tilde{c}_{0}$ where  $\epsilon \in \{ \pm 1\}.$\

ii) Assume that the conditions from i) are satisfied. Then 
$\nabla$ is  isomorphic to $\tilde{\nabla}$ if and only if
there is $n\in \mathbb{N}_{\geq 2}$ such that
\begin{equation}\label{gen-c01}
4 (c_{0})^{2} ( c_{1} -\epsilon \tilde{c}_{1})^{2} - 8 (n-1)^{2} c_{0}( c_{1} +\epsilon  \tilde{c}_{1}) + (2n-1) (2n-3) (n-1)^{2} =0
\end{equation}
and, for any  $2\leq r\leq n-1$, $r\in \mathbb{N}$, 
\begin{equation}\label{gen-cc}
c_{0} ( c_{1}+\epsilon \tilde{c}_{1})\neq  \frac{(2n-1)(2n-3) (n-1)^{2} - (2r-1) (2r-3) (r-1)^{2}}{ 8 (n-r)(n-2+r)}.
\end{equation}
\end{prop}

\begin{proof} i) We consider a formal isomorphism $T := \sum_{n\geq 0} T^{(n)}z^{n}$ which maps $\nabla$
to $\tilde{\nabla }.$ We write  $T^{(n)} = \tau_{1}^{(n)} C_{1} + \tau_{2}^{(n)} C_{2} + \tau_{3}^{(n)} D 
+ \tau_{4}^{(n)} E$, where 
$\tau_{i}^{(n)} \in \mathbb{C}$. Relation  (\ref{2.18}) with $r=0$, applied to $\nabla$, 
$\tilde{\nabla}$ and $T$,  gives 
$c = \tilde{c}$, $\tau_{4}^{(0)} =0$ and
\begin{equation}\label{k-d-0}
\tau_{1}^{(0)} ( c_{0} -\tilde{c}_{0}) + \tau_{3}^{(0)} (c_{0} +\tilde{c}_{0}) =0.
\end{equation}
Using that $\tau_{4}^{(0)} =0$, relation (\ref{2.18}) with $r=1$ becomes
\begin{align}
\nonumber& \frac{\tau_{4}^{(1)}}{2} (c_{0} -\tilde{c}_{0}) + (\alpha-\tilde{\alpha}) \tau_{1}^{(0)} + (c_{0} -\tilde{c}_{0}) \frac{ \tau_{2}^{(0)}}{2} =0\\
\nonumber& \tau_{1}^{(1)} (c_{0} -\tilde{c}_{0}) + \tau_{3}^{(1)} (c_{0} +\tilde{c}_{0})
 + (\alpha -\tilde{\alpha } +\frac{1}{2}) \tau_{2}^{(0)} + (c_{1} 
-\tilde{c}_{1}) \tau_{1}^{(0)} + (c_{1} + \tilde{c}_{1} ) \tau_{3}^{(0)} =0\\
\nonumber&\frac{ \tau_{4}^{(1)}}{2} (c_{0} +\tilde{c}_{0}) - (\alpha -\tilde{\alpha }) \tau_{3}^{(0)} - \frac{ \tau_{2}^{(0)}}{2} (c_{0} +\tilde{c}_{0}) =0\\
\label{nivel-1}& (c_{0} -\tilde{c}_{0})\tau_{1}^{(0)} - (c_{0} +\tilde{c}_{0}) \tau_{3}^{(0)}=0.  
\end{align}  
Suppose that $\tilde{c}_{0} \neq c_{0}.$ 
Then, from (\ref{k-d-0}), $\tau_{1}^{(0)} = -(\frac{ c_{0} +\tilde{c}_{0}}{ c_{0} -\tilde{c}_{0}}) \tau_{3}^{(0)}$ and $\tau_{3}^{(0)}\neq 0$
(if $\tau_{3}^{(0)} =0$ then  $\tau_{1}^{(0)} =0$; since  $\tau_{4}^{(0)} =0$ we obtain that $T^{(0)}$ is not invertible, which is a contradiction). 
The last relation (\ref{nivel-1}) implies  
$c_{0} =-\tilde{c}_{0}$. 
We proved that $c=\tilde{c}$ and $c_{0} = \epsilon \tilde{c}_{0}$ where $\epsilon \in \{ \pm 1\} .$ 
The first and third relations (\ref{nivel-1}) (together with $T^{(0)}$-invertible and (\ref{k-d-0}))
imply that $\alpha = \tilde{\alpha} $.
The first claim follows.\

ii) The claim follows by
considering an holomorphic isomorphism $T$ between $\nabla$ and $\tilde{\nabla}$, 
identifying coefficients in 
(\ref{2.18}) with computations similar to those already done before, and using that $T$
is a polynomial (see Exercise 3.10 of \cite{sabbah}).
\end{proof}

\begin{cor}\label{imp-c1} In the setting of Proposition \ref{classif-Birk}, assume that $\tilde{c}_{1} =0.$
Then $\nabla$ is isomorphic to  $\tilde{\nabla}$ if and only if
$c =\tilde{c}$, $\alpha = \tilde{\alpha}$, $(c_{0})^{2}=(\tilde{c}_{0})^{2}$ and  
 there is $n\in \mathbb{Z}_{\geq 2}$ such that
\begin{equation}\label{cond-c0-c1}
c_{0}c_{1} \in \{ \frac{ (n-1)(2n-1)}{2}, \frac{ (n-1) (2n-3)}{2}\} .
\end{equation}
\end{cor}

\begin{proof} When $\tilde{c}_{1} =0$, relation (\ref{gen-c01})  is an equation in $c_{0}c_{1}$,
with solutions given by the right hand side of (\ref{cond-c0-c1}).  If  $\tilde{c}_{1} =0$ then 
(\ref{cond-c0-c1}) implies  (\ref{gen-cc}). 
\end{proof}

\begin{rem}\label{explanations}{\rm i) 
Lemmas  \ref{char-non-elem}
and  \ref{conj-ct},  combined with
Proposition \ref{classif-Birk},  provide a criterion to decide when two 
non-elementary $(TE)$-structures are isomorphic, using their restriction at the origin of $\mathcal N_{2}$\

ii)  The only non-elementary  formal normal forms  from Theorem \ref{te}
are  those 
from Theorem \ref{te} i) with 
$c_{0}\neq 0$.
They represent the formal normal forms of non-elementary $(TE)$-structures. 
Their restriction at the origin  are in Birkhoff normal form, with 
\begin{equation}\label{initial}
B_{0}^{o}=c C_{1}+ c_{0} C_{2},\  B_{\infty}=\alpha C_{1} -\frac{1}{4} D +  c_{0}E, 
\end{equation}
(where $c_{0} \neq 0$).\

iii) There are non-elementary $(TE)$-structures which are not (holomorphically) isomorphic to their
formal normal form(s): from  Corollary \ref{imp-c1},   they coincide (up to isomorphism) with the Malgrange universal deformations
of  connections  $\nabla^{B_{0}^{o}, B_{\infty}}$ in Birkhoff normal form,
with matrices $B_{0}^{o}$, $B_{\infty}$ as in  (\ref{almost-fin}), such that $c_{0}c_{1}$  does not satisfy
(\ref{cond-c0-c1}).}
\end{rem}

In order to obtain  a list of holomorphic normal forms for  non-elementary
$(TE)$-structures 
we will  express the Malgrange universal  deformations of the meromorphic connections 
$\nabla^{B_{0}^{o}, B_{\infty}}$ 
in Birkhoff normal form (\ref{conn-form-elem}), with matrices $B_{0}^{o}$, $B_{\infty}$ given by
(\ref{almost-fin}), with $c_{0}B^{\infty}_{12}\neq 0$,   in local  coordinates $(t_{1}, t_{2})$ of $\mathcal N_{2}.$ 
This will be done in the next sections.

\subsubsection{Non-elementary $(TE)$-structures and Malgrange universal connections}\label{coord-sect}

Let $B_{0}^{o}$, $B_{\infty}\in M_{2\times 2} (\mathbb{C})$ be two  matrices, where 
$B_{0}^{o}$ is  regular, with  one Jordan block, i.e. 
$B_{0}^{o} = cC_{1} + c_{0} C_{2}$,
and $c_{0}\neq 0$.
We are interested in the case $B^{\infty}_{12}\neq 0$ but for the moment we don't make this assumption.
We denote by 
$\nabla^{\mathrm{univ}}= \nabla^{\mathrm{univ}, B^{o}_{0}, B_{\infty}}$ the Malgrange universal deformation 
of  the meromorphic connection $\nabla^{B_{0}^{o}, B_{\infty}}$ 
with connection form 
$$
\Omega^{B_{0}^{o}, B_{\infty}} = \frac{1}{z^{2}} ( B_{0}^{o} + B_{\infty} z) dz
$$
We consider $(M^{\mathrm{univ}},\circ_{\mathrm{univ}}, e_{\mathrm{univ}}, E_{\mathrm{univ}})$ the 
parameter space of $\nabla^{\mathrm{univ}}$. The germs $((M^{\mathrm{univ}},0), \circ_{\mathrm{univ}}, e_{\mathrm{univ}})$ and $\mathcal N_{2}$ are isomorphic.

For any $\Gamma \in M_{2\times 2} (\mathbb{C})$, we identify 
$T_{\Gamma} M_{2\times 2} (\mathbb{C})$ with $M_{2\times2}(\mathbb{C})$ in the natural way.  Therefore, vector fields
on $M_{2\times 2} (\mathbb{C})$ or on the submanifold $M^{\mathrm{univ}}$ will be viewed as $M_{2\times 2}(\mathbb{C})$-valued functions (defined on $M_{2\times 2}(\mathbb{C})$ or $M^{\mathrm{univ}}$ respectively).

Let  $X_{0}$, $X_{1}$ be  vector fields on $M_{2\times 2}(\mathbb{C})$ 
defined by 
$(X_{0})_{\Gamma} = C_{1}$ and $(X_{1})_{\Gamma} = B_{0}^{o} - \Gamma + [ B_{\infty}, \Gamma ]$,
for any $\Gamma \in M_{2\times 2}(\mathbb{C})$. In the standard coordinates  $(\Gamma_{ij})$ of  $M_{2\times 2} (\mathbb{C})$
(where $\Gamma_{ij}: M_{2\times 2}(\mathbb{C}) \rightarrow \mathbb{C}$ is
the function  which assigns to $\Gamma \in M_{2\times 2}(\mathbb{C})$ its $(i,j)$-entry), 
\begin{equation}\label{x01-coord}
X_{0} = \sum_{i,j=1}^{2} \delta_{ij} \frac{\partial}{\partial \Gamma_{ij}},\ 
X_{1}  = \sum_{i,j=1}^{2} ( (B_{0}^{o})_{ij} - \Gamma_{ij} + (B_{\infty})_{ik} \Gamma_{kj} - \Gamma_{ik} (B_{\infty})_{kj})
\frac{\partial}{\partial \Gamma_{ij}}.
\end{equation}
(To simplify notation, we omitted the summation sign over $k\in \{ 1,2\}$). 
Let
\begin{equation}\label{def-tk}
\tilde{k}: M_{2\times 2} (\mathbb{C}) \rightarrow \mathbb{C},\ 
\tilde{k}(\Gamma ):= -\frac{1}{2} \mathrm{trace} (X_{1})_{\Gamma}= \frac{1}{2} \sum_{i=1}^{2} \Gamma_{ii} - c.
\end{equation}

Viewing a vector field $X$ on $M_{2\times 2} (\mathbb{C})$   as  an $M_{2\times 2}(\mathbb{C})$-valued function,
we  can consider its derivative   along any  other vector field $Y$ on $M_{2\times2}(\mathbb{C})$.
The result is a function $Y(X) : M_{2\times 2}(\mathbb{C}) \rightarrow M_{2\times 2}(\mathbb{C})$,
whose $(i,j)$-entry is the function $Y(X_{ij}).$  Various such  derivatives are computed in the next lemma
(below $C_{1}$ denotes the constant function on $M_{2\times 2} (\mathbb{C})$ equal to $C_{1}$).

\begin{lem}\label{comput1} The following relations hold:
\begin{align}
\nonumber& X_{0} ( X_{1})= - C_{1},\ X_{0}(\tilde{k})=1,\ X_{0} ( \tilde{k}X_{0} + X_{1}) =0; \\
\label{compute-deriv} & X_{1} (X_{1}) =  - X_{1} + [ B_{\infty}, X_{1}],\ X_{1} (\tilde{k}) =-\tilde{k}.
\end{align}
\end{lem}

\begin{proof}
As $X_{0}(\Gamma_{ij}) = \delta_{ij}$ for  any $i, j\in \{ 1, 2\}$ we obtain   
\begin{align*}
& X_{0} ( [ B_{\infty}, \Gamma ]_{ik}) = \sum_{j=1}^{2} X_{0}( (B_{\infty})_{ij} \Gamma_{jk} - \Gamma_{ij} (B_{\infty})_{jk} )\\
 & = \sum_{j=1}^{2} ( (B_{\infty})_{ij}\delta_{jk} - (B_{\infty})_{jk} \delta_{ij} ) = (B_{\infty})_{ik} - (B_{\infty})_{ik} =0
\end{align*}
and 
$$
X_{0}( (B_{0}^{o})_{ik} - \Gamma_{ik} + [ B_{\infty}, \Gamma]_{ik}) = -\delta_{ik}.
$$
We proved  that $X_{0} (X_{1}) = - C_{1}$.  Since $X_{0}(\Gamma_{ii})=1$ we obtain 
$X_{0} (\tilde{k}) = 1.$ Obviously, $X_{0} (X_{0})=0$ 
(since $(X_{0})_{ij} = \delta_{ij}$ are constants) 
and so
$$
X_{0} ( \tilde{k} X_{0} + X_{1}) = X_{0} (\tilde{k}) X_{0} + X_{0} (X_{1}) = C_{1} - C_{1} =0.
$$ 
The first line of (\ref{compute-deriv}) follows. The second line can be proved similarly. 
\end{proof}

Using the expression (\ref{x01-coord})
of $X_{0}$ and $X_{1}$ we compute
the Lie derivative of $X_{1}$ in the direction of $X_{0}$:  $L_{X_{0}} X_{1}  = - X_{0}$.
Using  $X_{0}(\tilde{k})=1$ we obtain
that the vector fields  $X_{0}$ and $\tilde{k} X_{0} + X_{1}$ commute. 
Their restriction to $M^{\mathrm{univ}}$  
are the fundamental vector fields of a coordinate system $(t_{1}, t_{2})$ on $M^{\mathrm{univ}}$,
which we choose to be centred at the origin of $M^{\mathrm{univ}}$.
As $X_{0} (\tilde{k}) =1$ and  $(\tilde{k}X_{0} + X_{1})(\tilde{k})=0$ 
(from Lemma \ref{comput1})
and $\tilde{k}(0) = -c$ (from the definition  (\ref{def-tk}) of $\tilde{k}$)
we obtain that $\tilde{k}(t_{1}, t_{2} ) =t_{1}- c$.

\begin{rem}\label{properties-mat}{\rm 
For any $\Gamma \in M^{\mathrm{univ}}$, the matrix $(\tilde{k} X_{0} + X_{1})(\Gamma )$ has the following properties:
it is trace-free (from the definition of $\tilde{k}$); it is regular (since $(X_{1})(\Gamma )$ is regular, being regular
at $\Gamma =0$); it has only one Jordan block (since $B_{0}^{o}$ 
has this property; see appendix). 
We obtain that $(\tilde{k} X_{0} + X_{1})(\Gamma )$ is conjugated to a matrix with all entries zero except
the $(2,1)$-entry which is non-zero. 
In particular, $(\tilde{k} X_{0} + X_{1})(\Gamma )^{2} =0$. 
Also,
$(\tilde{k} X_{0} + X_{1})_{21} = (X_{1})_{21}$ at $\Gamma =0$ is equal to $c_{0}\neq 0$. The function
$y:= (X_{1})_{21}: M^{\mathrm{univ}}\rightarrow \mathbb{C}$ is non-vanishing in a neighborhood of the origin in $M^{\mathrm{univ}}$ 
and the function $\tilde{k}X_{0} + X_{1}: M^{\mathrm{univ}}\rightarrow M_{2\times 2}(\mathbb{C})$ 
can be written as 
\begin{equation}\label{ex-matr}
\tilde{k} X_{0} + X_{1} = y \left(\begin{tabular}{cc}
$x$ & $- x^{2}$\\
$1$ &  $-x$
\end{tabular}\right),
\end{equation}
where 
\begin{equation}\label{simpl-not}
x= \frac{1}{(X_{1})_{21}} (\tilde{k} + (X_{1})_{11}) = \frac{1}{2 (X_{1})_{21}}( (X_{1})_{11} - (X_{1})_{22})  :M^{\mathrm{univ}} \rightarrow \mathbb{C} .
\end{equation}}
\end{rem}

\begin{prop}\label{expr-mal-coord}
In the coordinate system $(t_{1}, t_{2})$,  
the  Malgrange universal connection
$\nabla^{\mathrm{univ}}$ is given by
the matrices $A_{1}$, $A_{2}$ and 
$B = B^{(0)} + B^{(1)} z$, where 
 \begin{equation}\label{te-2}
A_{1} = C_{1},\ A_{2} = y \left(\begin{tabular}{cc}
$x$ & $- x^{2}$\\
$1$ &  $-x$
\end{tabular}\right) 
\end{equation}
and  
\begin{equation}\label{te-1}
B^{(0)} =   \left(\begin{tabular}{cc}
$xy - t_{1}+ c$ & $- x^{2}y$\\
$y$ &  $-xy - t_{1}+ c$
\end{tabular}\right) , \ B^{(1)} = B_{\infty}.
\end{equation}
\end{prop}

\begin{proof}
Recall  the definition (\ref{omega-cann}) of the Malgrange universal connection.
The equality $A_{1} = C_{1}$ follows from 
the fact that 
$A_{1} = A_{\frac{\partial}{\partial t_{1}}}$ is the matrix valued function on $M^{\mathrm{univ}}$  identified with
the vector field $\frac{\partial}{\partial t_{1}}= X_{0}$, which is $C_{1}$.
The expression of $A_{2}$ follows similarly, from $A_{2} = C_{\frac{\partial}{\partial t_{2}}}$, together with
$\frac{\partial}{\partial t_{2}} = \tilde{k}X_{0} + X_{1}$ and 
(\ref{ex-matr}). 
The  expression of $B^{(0)}$ is obtained as follows:
from (\ref{B-0}), $B^{(0)}(\Gamma )  = (X_{1}) (\Gamma)$ for any $\Gamma \in M^{\mathrm{univ}}$. Therefore, 
\begin{align*}
\nonumber B^{(0)} &=  (\tilde{k}X_{0}+X_{1})-\tilde{k}X_{0} =  y \left(\begin{tabular}{cc}
$x$ & $- x^{2}$\\
$1$ &  $-x$
\end{tabular}\right) - (t_{1}- c )\mathrm{Id}\\
\nonumber & =  \left(\begin{tabular}{cc}
$xy - t_{1}+ c$ & $- x^{2}y$\\
$y$ &  $-xy - t_{1}+ c$
\end{tabular}\right)
\end{align*}
where we used (\ref{ex-matr}) and $\tilde{k} (t_{1}, t_{2}) = t_{1}- c.$
\end{proof}

To  simplify notation, in the proof of the next lemma  instead of the vector field  $X_{1}$ we simply write $X$.
The $(i, j)$-entry of $X_{1}$ (viewed as a $M_{2\times 2}(\mathbb{C})$-valued function) 
will be denoted by $X_{ij}.$

\begin{lem}\label{comput2} The functions $x$ and $y$ are independent on $t_{1}$ and their derivatives with respect to $t_{2}$ are given by
\begin{align}
\nonumber & \dot{x} = - B^{\infty}_{21} x^{2} + (B^{\infty}_{11} - B^{\infty}_{22}) x + B^{\infty}_{12}\\
\label{dx} & \dot{y} = y( 2 B^{\infty}_{21} x + B^{\infty}_{22} - B^{\infty}_{11} -1),
\end{align}
where $B^{\infty}_{ij}$ denotes the $(i, j)$-entry of the matrix $B_{\infty}$. 
They satisfy the initial conditions $x(0)=0$ and $y(0) =c_{0}$.
\end{lem}

\begin{proof} From the first line of (\ref{compute-deriv}) and the definition of $x$ and $y$, $X_{0}(x) = X_{0}(y) =0$, i.e. $x$ and $y$ are independent on $t_{1}.$ 
From the second line of (\ref{compute-deriv}), 
\begin{equation}\label{x-11}
X(X_{11}) =- X_{11} + [B_{\infty}, X]_{11}=  - X_{11} + B^{\infty}_{12} X_{21} - B^{\infty}_{21}X_{12}
\end{equation}
and similarly 
\begin{equation}\label{x-21}
X(X_{21}) = - X_{21} + B^{\infty}_{21} (X_{11}- X_{22})  + (B^{\infty}_{22} - B^{\infty}_{11} ) X_{21} .
\end{equation}
Using (\ref{x-11}), (\ref{x-21}),  $X (\tilde{k}) =- \tilde{k}$ and $\mathrm{det}(X +\tilde{k} \mathrm{Id})=0$, we obtain
$$
X \left(\frac{\tilde{k} + X_{11}}{X_{21}}\right) =  B^{\infty}_{12}
+   (B^{\infty}_{11} - B^{\infty}_{22})\left(  \frac{X_{11} - X_{12}}{2 X_{21}} \right)  -  B^{\infty}_{21}
 \left(\frac{X_{11} - X_{22}}{2X_{21}}\right)^{2}, 
$$
which implies  the first relation (\ref{dx})
(we use the definition (\ref{simpl-not}) of $x$ and $\dot{x} = X(x)$, since 
$\frac{\partial}{\partial t_{2}} = \tilde{k}X_{0} + X$ and $X_{0}(x) =0$). 
The second relation (\ref{dx}) can be obtained similarly. 
\end{proof}

\begin{rem}{\rm   When $B^{\infty}_{12}=0$, the system (\ref{dx})  is solved by  $x=0$ and $y(t_{1}, t_{2}) = 
c_{0} e^{ kt_{2}}$ where $c_{0}\in \mathbb{C}$ and $k:=   B^{\infty}_{22} - B^{\infty}_{11} -1.$ 
Assume that $k\neq 0.$  Replacing the expressions of $x$ and $y$ in 
(\ref{te-2}),  (\ref{te-1})  
we obtain that 
$\nabla^{\mathrm{univ}}$ is the pull-back 
by $\mu (t_{1}, t_{2}) = (t_{1}, \frac{c_{0}}{k}(e^{kt_{2}} -1))$
of  the $(TE)$-structure $\tilde{\nabla}$ 
given by 
$$
\tilde{A}_{1} = C_{1},\ \tilde{A}_{2} = C_{2},\ \tilde{B} = (-t_{1}+ c) C_{1} + 
(k t_{2} + c_{0}) C_{2} + z B_{\infty}. 
$$
When $k=0$ the same statement holds with $\mu (t_{1}, t_{2}) = ( t_{1}, c_{0} t_{2}).$
We obtain that $\nabla^{\mathrm{univ}}$ is isomorphic to $\tilde{\nabla}.$
Remark that $\tilde{\nabla}$ is of the third type in Theorem \ref{te}.
} 
\end{rem}

We now turn to the Malgrange universal deformations $\nabla^{\mathrm{univ}}$ we are interested in, 
namely those which are non-elementary. Therefore, we assume that $B_{12}^{\infty}\neq 0.$ 
From Lemma \ref{conj-ct}, we may (and will) assume, without loss of generality, that 
$B^{\infty}_{11} - B^{\infty}_{22} =-\frac{1}{2}$
and $B_{12}^{\infty} = c_{0}$.
We distinguish two subcases, namely 
$B^{\infty}_{21} =0$ 
and  $B^{\infty}_{21} \neq 0.$ In the first subcase, $\nabla^{\mathrm{univ}}$ is isomorphic to a $(TE)$-structure of the 
first type in Theorem \ref{te}:

\begin{cor}\label{cor-2} If $B^{\infty}_{12} = c_{0}$,   $B^{\infty}_{21} =0$
and $B_{11}^{\infty} - B_{22}^{\infty} = -\frac{1}{2}$, 
then
$\nabla^{\mathrm{univ}}$ is isomorphic to the $(TE)$-structure $\tilde{\nabla}$ given by
\begin{align}
 \nonumber & \tilde{A}_{1} = C_{1},\  \tilde{A}_{2} =  C_{2}+ z  E\\
 \label{te-1-b} & \tilde{B} = ( - t_{1} + c +\alpha z ) C_{1} 
+(-\frac{t_{2}}{2} + c_{0}) C_{2}-\frac{z}{4} D + z ( -\frac{t_{2}}{2} + c_{0}) E,
\end{align}
where  $\alpha := \frac{1}{2} ( B_{11}^{\infty} + B_{22}^{\infty})$. 
\end{cor}

\begin{proof} The functions $x(t_{1}, t_{2}) = 2 c_{0} (1- e^{-\frac{t_{2}}{2}})$ and
$y(t_{1}, t_{2}) = c_{0} e^{-\frac{t_{2}}{2}}$ solve the system (\ref{dx}). From $y = \dot{x}$ we  
obtain that  $\nabla^{\mathrm{univ}}$
is the pull-back, by the function $\mu (t_{1}, t_{2}) = (t_{1},  x(t_{1}, t_{2}))$,  of the $(TE)$-structure ${\nabla}^{[1]}$ 
with matrices $A^{[1]}_{1}$, $A^{[1]}_{2}$, $B^{[1]}= \sum_{k\geq 0} B^{[1] , (k)} z^{k}$ 
given by 
 \begin{align}
\nonumber & {A}^{[1]}_{1} = C_{1},\  {A}^{[1]}_{2} =  C_{2}+ t_{2}D - t_{2}^{2} E\\
\nonumber & {B}^{[1],(0)} = ( - t_{1} + c) C_{1} +(-\frac{t_{2}}{2} + c_{0}) C_{2} + t_{2} (-\frac{t_{2}}{2} + c_{0}) D + t_{2}^{2} (\frac{t_{2}}{2} - c_{0})E\\
\nonumber&  {B}^{[1], (1)}= \alpha   C_{1} -\frac{1}{4} D + c_{0}E\\
\label{conn-univ-pb} & B^{[1],(k)} =0,\ k\geq 2.
\end{align} 
The  gauge isomorphism $T:= C_{1} +  t_{2}E$ maps ${\nabla}^{[1]}$ to 
$\tilde{\nabla}.$ 
\end{proof}

It remains to consider the case  when both $B^{\infty}_{12}$ and $B^{\infty}_{21}$ are non-zero.

\begin{cor} \label{non-ele-case}
Assume that $B_{12}^{\infty} = c_{0}$, $B_{21}^{\infty} \neq 0$ and $B_{11}^{\infty} - B_{22}^{\infty}
=-\frac{1}{2}.$   Define $a, b\in \mathbb{C}$ by 
 $a+b= - \frac{1}{2B^{\infty}_{21}}$ and $ab= -\frac{c_{0}}{B^{\infty}_{21}}.$

i) When $B_{12}^{\infty} B^{\infty}_{21} \neq -\frac{1}{16}$,
the system (\ref{dx}) is solved by  
\begin{align}
\nonumber& x(t_{1}, t_{2}) =  ab ( 1- e^{(b-a) B^{\infty}_{21} t_{2}})( b - a e^{(b-a)B^{\infty}_{21}t_{2}})^{-1}\\
\label{a-b}& y(t_{1}, t_{2}) =\frac{ c_{0}}{(b-a)^{2}} ( b-{a} e^{(b-a)B^{\infty}_{21}t_{2}})^{2} e^{(B^{\infty}_{21} (a-b)-1)t_{2}}.
\end{align}
ii) When $B_{12}^{\infty} B^{\infty}_{21} =  -\frac{1}{16 }$,  it is solved by 
\begin{equation}\label{a-b-sec}
x(t_{1}, t_{2})  = \frac{ 4 c_{0} t_{2}}{t_{2}+4},\ 
y(t_{1}, t_{2})  = \frac{c_{0}}{16}  e^{-t_{2}} (t_{2} +4)^{2}.
\end{equation}
\end{cor}

\begin{proof}
We write  equation
(\ref{dx}) as   $\dot{x} = - B^{\infty}_{21} ( x-a) (x-b)$.
Since $x(0) =0$ this determines $x$ as stated in the lemma.  Then $y$ is determined from the second equation (\ref{dx}).
Remark that $B_{12}^{\infty}B^{\infty}_{21}  =  -\frac{1}{16 }$ if and only if $a=b.$  
\end{proof}

To simplify terminology we introduce the following definition.

\begin{defn} A non-elementary Malgrange normal form  of the first (respectively second) type is a Malgrange
universal deformation $\nabla^{\mathrm{univ}}$  as in Proposition \ref{expr-mal-coord},
with functions $x$ and $y$ satisfying  (\ref{dx}),  
$B_{12}^{\infty} = c_{0}\neq 0$, $B_{11}^{\infty} - B_{22}^{\infty} = -\frac{1}{2}$
and  $B_{21}^{\infty} =0$ (respectively, $B_{21}^{\infty}\neq 0$).
\end{defn}

Non-elementary $(TE)$-structures coincide (up to  isomorphisms) to  non-elementary Malgrange normal
forms. From Corollary \ref{cor-2},  those  of the first type are isomorphic  to the $(TE)$-structures 
(\ref{f-1}) from Theorem \ref{te},  with $c_{0}\neq 0.$ In the next section 
we   study  the non-elementary Malgrange normal forms of the second type.

\subsubsection{Non-elementary Malgrange normal forms of second type}

We are looking for (holomorphic)  isomorphisms $T$ which map an arbitrary  
non-elementary Malgrange normal form
of the second type $\nabla := \nabla^{\mathrm{univ}, B_{0}^{0}, B_{\infty}}$ to  a  $(TE)$-structure  which is    'as close as possible' to the formal normal forms from Theorem \ref{te}.  
Let $A_{1}$, $A_{2}$, $B= B^{(0)} + z B^{(1)}$ be the matrices of $\nabla$, described  
in Proposition \ref{expr-mal-coord},  in terms of functions
$x$, $y$ determined in Corollary \ref{non-ele-case}.
Recall that $B_{0}^{o} = c C_{1} + c_{0} C_{2}$ (with $c_{0}\neq 0$), 
$B_{12}^{\infty} = c_{0}$ and $B_{11}^{\infty} - B_{22}^{\infty} = -\frac{1}{2}.$ 
From Remark \ref{properties-mat},  $A_{2} = A_{2}^{(0)}$ is conjugated 
to a matrix of the form $F C_{2}$, for a function $F= F(t_{2})$.
Therefore, there is a gauge  isomorphism $T = T^{(0)}$, which depends only on $t_{2}$,  such that the underling
$(T)$-structure of ${\nabla}^{[1]}:= T\cdot \nabla$ is  
\begin{equation}\label{new-t}
{A}^{[1]}_{1} = C_{1},\    {A}^{[1]}_{2} = F C_{2} + z T^{-1}\partial_{2} T.
\end{equation}
As $T$ is independent on $z$, the matrix $B^{[1]}$ of $\nabla^{[1]}$ is given by   
\begin{equation}\label{tilde-b}
{B}^{[1]} = T^{-1} B^{(0)} T + z T^{-1} B^{( 1)} T.
\end{equation}

\begin{lem}\label{non-compl} The gauge isomorphisms $T$, which depend only on $t_{2}$, 
and map $\nabla$ to a $(TE)$-structure $\nabla^{[1]}:= T\cdot \nabla$, whose underlying $(T)$-structure 
satisfies
\begin{equation}\label{t-t-a}
{A}^{[1]}_{1} = C_{1},\ {A}^{[1]}_{2} = F C_{2} + z G  E
\end{equation}
for functions $F= F(t_{2})$ and $G= G(t_{2})$, are of the form
\begin{equation}\label{expr-for-t}
T= \left(\begin{tabular}{cc}
$ k_{0} k$ & $ \frac{k_{1} x}{k-x}$\\
$k_{0}$ & $ \frac{k_{1}}{k-x}$
\end{tabular}\right) ,\ T= \left(\begin{tabular}{cc}
$ k_{1}$ & $ k_{0} x$\\
$0$ & $ k_{0}$
\end{tabular}\right) 
\end{equation}
where $k_{0}, k_{1}, k\in \mathbb{C}^{*}$. If  $T$ is given by the first formula (\ref{expr-for-t}), 
then
$F= \frac{k_{0}}{k_{1}} (k-x)^{2}y,\ G= \frac{ k_{1} \dot{x}}{k_{0} (k-x)^{2}}$.
If  $T$ is given by the second formula (\ref{expr-for-t}), 
then $F= \frac{k_{1}}{k_{0}}y$, $G= \frac{k_{0}}{k_{1}} \dot{x}$. In both cases, $F(0)\neq 0.$ 
\end{lem}

\begin{proof}
By a straightforward computation, the matrices $T = T(t_{2})$ which satisfy 
$T^{-1} A_{2} T = F C_{2}$ and $T^{-1}\partial_{2} T = G E$  are of the form 
\begin{equation}\label{expresie-t}
T = \left(\begin{tabular}{cc}
$q x + \tilde{q} \frac{F}{y}$ & $ x\tilde{q}$\\
$q$ & $ \tilde{q}$
\end{tabular}\right) ,
\end{equation}
where $q, \tilde{q}\in \mathbb{C} \{ t_{2}\}$ and  
\begin{align}
\nonumber& \frac{d}{dt_{2}}(\frac{\tilde{q}F}{y}) + q \dot{x} =0\\
\nonumber& \dot{\tilde{q}} \frac{F}{y} = q\dot{x}\\
\label{q-tildeq} & q^{2} \dot{x} +  q \frac{d}{dt_{2}} (\frac{\tilde{q} F}{y}) - \dot{q} \frac{ \tilde{q}F}{y} =0.
\end{align}
Moreover, if (\ref{q-tildeq}) are satisfied, then $G= \frac{\dot{x} y}{F}.$

If $q(0)\neq 0$, we divide the third relation (\ref{q-tildeq}) by $q^{2}$   and we obtain
$\frac{d}{dt_{2}}(\frac{\tilde{q}F}{qy}) = -\dot{x}$, 
which implies that $\tilde{q} F = qy (k-x)$ for $k\in \mathbb{C}.$ 
Using $\frac{\tilde{q} F}{y} = q(k-x)$, 
the  first relation (\ref{q-tildeq}) implies 
that $q = k_{0}$ is constant. The  second relation (\ref{q-tildeq}) determines $\tilde{q}$ as
$\tilde{q} = \frac{k_{1}}{k-x}$, for $k_{1}\in \mathbb{C}^{*}.$   The expressions for $T$, $F$ and $G$ follow.
As $T$ is invertible and $x(0) =0$, $k, k_{0}\in \mathbb{C}^{*}.$

If $q(0)=0$ then $q=0$ (otherwise $q(z ) =z^{r} \eta (z)$ for $r\in \mathbb{Z}_{\geq 1}$ and $\eta \in \mathbb{C} \{ t_{2}\}$ non-trivial.
But writing  $q$ in this way we obtain a contradiction in the third relation (\ref{q-tildeq})). 
The case $q=0$ can be treated similarly and leads to the second expression in (\ref{expr-for-t}) for $T$ and to
$F$, $G$ as required. 

Since $x(0) =0$, $k_{0} , k_{1}, y(0)\in \mathbb{C}^{*}$, we obtain that $F(0)\neq 0$ (in both cases). 
\end{proof}

Let $T$ be a  gauge  isomorphism  as in Lemma \ref{non-compl}. The underling $(T)$-structure of ${\nabla}^{[1]}= T\cdot \nabla$ is of the form
$$
{A}^{[1]}_{1} = C_{1},\  {A}^{[1]}_{2} = F ( C_{2} + \frac{G}{F} E) = \dot{\mu}_{2} (C_{2} + z f E)
$$
where $\mu_{2} \in \mathrm{Aut} (\mathbb{C}, 0)$ satisfies $\dot{\mu}_{2} = F$ and  
in the second expression for $A_{2}^{[1]}$ the function
$\frac{G}{F}$ is written in terms of $\mu_{2}$, i.e. $\frac{G}{F} = f(\mu_{2})$. 
(Remark that $\dot{\mu}_{2}(0)\neq 0$ since $F(0)\neq 0$). 
We obtain that 
$\nabla^{[1]}$ is the pull-back by $\mu (t_{1}, t_{2}) = (t_{1}, \mu_{2} (t_{2}))$ of 
the $(T)$-structure  
\begin{equation}\label{noname}
\tilde{A}_{1} = C_{1},\ \tilde{A}_{2} = C_{2} + z f E.
\end{equation}
Therefore,  the underling $(T)$-structure of ${\nabla}$ is 
isomorphic to the $(T)$-structure (\ref{noname}).  
In the following we will make  suitable choices 
in Lemma \ref{non-compl} which lead to 'simplest' expressions for the function $f$.

\begin{prop}\label{t-non-ele}  i) If $B^{\infty}_{12}B^{\infty}_{21} = - \frac{1}{16}$, then the underlying $(T)$-stucture of $\nabla$ is isomorphic
to the $(T)$-structure given by
\begin{equation}\label{t1}
\tilde{A}_{1} = C_{1},\ \tilde{A}_{2} = C_{2} + z \frac{(c_{0})^{2}}{1-t_{2}} E.
\end{equation}
ii) If $B^{\infty}_{12}B^{\infty}_{21}\neq -\frac{1}{16}$
and $B_{21}^{\infty} (b-a) \neq 1$  then 
 the underlying $(T)$-structure of $\nabla$ is isomorphic
to the $(T)$-structure given by 
\begin{equation}\label{t2}
\tilde{A}_{1} = C_{1},\ \tilde{A}_{2} = C_{2} + z (\frac{\lambda}{c_{0}} t_{2} + 1)^{-2-\frac{1}{\lambda }} E, 
\end{equation}
where 
$\lambda   := B_{21}^{\infty} (b-a) -1$ and $a, b\in \mathbb{C}$ are defined in Lemma \ref{non-ele-case}.\

iii) If  $B^{\infty}_{12}B^{\infty}_{21} \neq   -\frac{1}{16}$ 
and $B_{21}^{\infty}(b-a) =1$ 
then the 
underlying $(T)$-structure of $\nabla$ is isomorphic  to the $(T)$-structure given by 
\begin{equation}\label{t3}
\tilde{A}_{1} = C_{1},\ \tilde{A}_{2} = C_{2} + z (c_{0})^{2} e^{-t_{2}} E.
\end{equation}
\end{prop}

\begin{proof} i) Let $T$ be the gauge isomorphism given by the first expression
(\ref{expr-for-t}), with $k:= 4 c_{0}$ (and $k_{0}, k_{1}\in \mathbb{C}^{*}$ arbitrary).   
Define $\gamma := \frac{ k_{0}}{k_{1}}.$ 
Using $F= \gamma (k-x)^{2}y$ (see Lemma \ref{non-compl}) 
and the expressions of  $x$,  $y$ from 
(\ref{a-b-sec}), we obtain that
$$
F =\gamma (4c_{0}-x)^{2} y=  16\gamma (c_{0})^{3}e^{-t_{2}} = \dot{\mu}_{2}(t_{2})
$$
where  
$\mu_{2}\in \mathrm{Aut} (\mathbb{C}, 0)$ is given by
$\mu_{2} (t_{2})= 16 \gamma (c_{0})^{3}(1- e^{-t_{2}}).$ 
Then 
$e^{t_{2}} = \frac{ 16 (c_{0})^{3} \gamma }{ 16 (c_{0})^{3} \gamma - \mu_{2}}$ and 
$$
\frac{G}{F}=  \frac{\dot{x}}{ \gamma^{2}(4c_{0}-x)^{4}y}=
 (\frac{1}{16 \gamma (c_{0})^{2}})^{2}  e^{t_{2}}=  
\frac{1}{16 \gamma c_{0}\left (16\gamma (c_{0})^{3}  - \mu_{2}\right)}.
$$
We obtain that the underlying $(T)$-structure of $\nabla$ is isomorphic
to the $(T)$-structure $\tilde{\nabla}$ given by 
$$
\tilde{A}_{1} = C_{1},\ \tilde{A}_{2} = C_{2} +  \frac{z}{16 \gamma c_{0} \left( 16 \gamma (c_{0})^{3} 
- t_{2}\right)}E.
$$
For $\gamma  = \frac{1}{16 (c_{0})^{3}}$ we obtain the $(T)$-structure (\ref{t1}).\

ii) The claim follows by a similar argument, by taking  the gauge isomorphism $T$  given by the first formula
 (\ref{expr-for-t}) with
$k:=a$, $\frac{k_{1}}{k_{0}} := a^{2}$ and the automorphism $\mu_{2}\in \mathrm{Aut} (\mathbb{C} ,0)$
given by
$$
\mu_{2} (t_{2} ) :=  \frac{ c_{0}}{B_{21}^{\infty} (b-a) -1} (e^{(B_{21}^{\infty} (b-a) -1) t_{2}} -1).
$$

iii) The claim follows by a similar argument,   
by taking $T$ given by the first formula in (\ref{expr-for-t}), with   $k: =a$ and $\frac{k_{1}}{k_{0}}: = a^{2} c_{0}$
and $\mu_{2} \in \mathrm{Aut} ( \mathbb{C}, 0)$ the identity automorphism.
\end{proof}

We arrive at  our main result  from this section.

\begin{thm}\label{classif-non-ele} 
Let $\nabla = \nabla^{\mathrm{univ}, B_{0}^{o}, B_{\infty}}$ be a non-elementary Malgrange normal form of the second type, 
with $B_{0}^{o} = c C_{1} + c_{0} C_{2}$ 
and $B_{\infty} = (B^{\infty}_{ij})$ (thus $B_{12}^{\infty} = c_{0}\neq 0$ and $B_{11}^{\infty} - B_{22}^{\infty}
=-\frac{1}{2}$). Let $\alpha := \frac{1}{2}( B_{11}^{\infty} +  B_{22}^{\infty})$.\

i) If  $B^{\infty}_{12} B^{\infty}_{21} =-\frac{1}{16}$,  then $\nabla$ is isomorphic to the  $(TE)$-structure $\tilde{\nabla}$ given by
\begin{align}
\nonumber &\tilde{A}_{1} = C_{1},\ \tilde{A}_{2} = C_{2} +  
 \frac{z(c_{0})^{2}}{ 1- t_{2}} E,\\ 
\label{te-mal-1}& \tilde{B} = ( -t_{1} + c +\alpha z) C_{1} + (1- t_{2}) C_{2} + z (c_{0})^{2} E.
\end{align}

ii)  If $B^{\infty}_{12}B^{\infty}_{21}\neq  -\frac{1}{16}$
and $B_{21}^{\infty}(b-a) \neq 1$ then  $\nabla$ is 
isomorphic to the $(TE)$-structure $\tilde{\nabla}$ given by
 
\begin{align}
\nonumber& \tilde{A}_{1} = C_{1},\ \tilde{A}_{2} = C_{2} + z (\frac{\lambda }{c_{0}} t_{2} + 1)^{-2-\frac{1}{\lambda }} E\\
\label{te-mal-2}& \tilde{B} = (-t_{1} + c +\alpha z) C_{1} + (\lambda  t_{2} + c_{0}) C_{2} -\frac{z}{2} (\lambda +1) D + z c_{0}
( \frac{\lambda }{c_{0}} t_{2} 
+1)^{-1-\frac{1}{\lambda }}E,
\end{align}
where $\lambda := B_{21}^{\infty} (b-a) -1$ and $a, b\in \mathbb{C}$ are defined in Lemma \ref{non-ele-case}.

iii)  If  $B^{\infty}_{12}B^{\infty}_{21}\neq   - \frac{1}{16}$
and $B_{21}^{\infty} (b-a) =1$, then $\nabla$ is isomorphic to  the $(TE)$-structure  $\tilde{\nabla}$ given by
\begin{align}
\nonumber& \tilde{A}_{1} = C_{1},\ \tilde{A}_{2} = C_{2} + z (c_{0})^{2}e^{-t_{2}} E\\
\label{te-mal-3}& \tilde{B} = (-t_{1} + c +\alpha z) C_{1} + C_{2} -\frac{z}{2} D + z(c_{0})^{2} e^{-t_{2}} E.
\end{align} 
\end{thm} 

\begin{proof}
The idea of the proof  is common to the three cases. 
Recall that the matrix $B$ of $\nabla$ is given by
$B = B^{(0)} + z B^{(1)}$ with $B^{(0)}$, $B^{(1)}$ given by 
(\ref{te-1}) and  functions $x$, $y$ as in Lemma
\ref{non-ele-case}. Let $T$ be the gauge isomorphism used in the proof of Proposition
\ref{t-non-ele}. Recall that $T$ is given by the first formula in
(\ref{expr-for-t}) (with various  choices of constants $k,k_{0}, k_{1}$, 
according to the three cases of Proposition \ref{t-non-ele}).
The $(TE)$-structure $\nabla^{[1]}= T\cdot \nabla$ has matrix $B^{[1]}$ given by
$B^{[1]} = T^{-1} BT$ and a straightforward computation shows that
\begin{align}
\nonumber B^{[1]}& = ( - t_{1} + c +\alpha z) C_{1} +  \frac{k_{0}}{k_{1}} \left(  y ( k - x)^{2} + z ( \frac{k}{2} + k^{2} B_{21}^{\infty} - c_{0})\right) C_{2}\\
\nonumber &+ \frac{z}{ 2 (k-x)} \left( (-\frac{k}{2} + 2 c_{0}) - x( 2k B^{\infty}_{21} +\frac{1}{2} )\right) D\\
\label{n-f-nonele} & + \frac{zk_{1}}{k_{0} ( k-x)^{2}}\left( c_{0} -  x(\frac{1}{2} +  B_{21}^{\infty} x)\right) E.
\end{align}   
The choice of $k$  in the proof of Proposition \ref{t-non-ele}  (in all three cases), implies that
$  \frac{k}{2} + k^{2} B_{21}^{\infty} - c_{0}=0$. 
Therefore, the coefficient of $C_{2}$ in $B^{[1]}$ reduces to $\frac{k_{0}}{k_{1}} y (k - x)^{2}$. 
Define $\tilde{B}$ by 
$\tilde{B} :=( \mu^{-1})^{*}B^{[1]}$, where 
$\mu (t_{1}, t_{2}) = (t_{1}, \mu_{2}(t_{2}))$ and 
$\mu_{2}$ is  the automorphism of $(\mathbb{C}, 0)$  constructed 
in the proof of Proposition \ref{t-non-ele}  (for each case).
The matrix $\tilde{B}$ is obtained by writing $B^{[1]}$ in terms of $\mu_{2}$ (and $t_{1}$). 
Together with the $(T)$-structures from Proposition \ref{t-non-ele}, 
the matrices $\tilde{B}$ form  $(TE)$-structures
$\tilde{\nabla}$  isomorphic to $\nabla$ (in all three cases). 
Their associated functions $\tilde{b}_{2}$  are obtained by writing the coefficient  
$b^{[1]}_{2} =\frac{k_{0}}{k_{1}}  y (k-x)^{2}$ of $C_{2}$ in 
the expression of $B^{[1]}$ in terms of $\mu_{2}$. Making the computations explicit we 
obtain that $\tilde{\nabla}$  have the expressions stated in Theorem 
\ref{classif-non-ele}.

To ilustrate  our argument  we consider  the case  
$B_{12}^{\infty} B_{21}^{\infty} = -\frac{1}{16}$. Then, from the proof of Proposition \ref{non-ele-case}, 
$k = 4 c_{0}$, $\frac{k_{0}}{k_{1}}= \frac{1}{16 (c_{0})^{3}}$, $\mu_{2} (t_{2}) = 1 - e^{-t_{2}}$
and the underlying $(T)$-structure  of $\tilde{\nabla}$ is given by the first line of  (\ref{te-mal-1}). 
Using the expressions for $x$, $y$ given by  (\ref{a-b-sec}), we obtain that 
$$
b^{[1]}_{2}= \frac{1}{ 16 (c_{0})^{3}} y ( 4 c_{0} -x)^{2} = e^{-t_{2}} = 1 - \mu_{2} (t_{2}). 
$$
Thus, $\tilde{b}_{2} (t_{2}) = 1 - t_{2}$, which leads to the matrix $\tilde{B}$ given  in
(\ref{te-mal-1}).
\end{proof}

\begin{defn}\label{concl-non-ele}  A holomorphic normal form for $(TE)$-structures over $\mathcal N_{2}$ is a $(TE)$-structure which belongs 
either to the list of $(TE)$-structures from Theorem \ref{te} or to the list of $(TE)$-structures from
Theorem \ref{classif-non-ele}.
\end{defn}

The next corollary summarises our discussion on the holomorphic classification.

\begin{cor}\label{concl-non-ele-1} Any $(TE)$-structure over $\mathcal N_{2}$  is   isomorphic to a holomorphic 
normal form.
\end{cor}

It remains to establish when two holomorphic normal forms $\nabla$ and $\tilde{\nabla}$ are  isomorphic. 
If $\nabla$ and $\tilde{\nabla}$ are  isomorphic (and distinct), then they are both 
elementary or both non-elementary.  In the first case, $\nabla$ and $\tilde{\nabla}$  are as in Theorem \ref{te} i) with $c_{0}=0$, 
ii) or iii). They are  isomorphic if and only if they are formally isomorphic and this happens if and only if the conditions
from Theorem \ref{te-class} ii) are satisfied. 
In the second case, $\nabla$ and $\tilde{\nabla}$ are as in Theorem \ref{te} i) with $c_{0} \neq 0$ or as in Theorem 
\ref{classif-non-ele}.  
We shall associate to $\nabla$ (and $\tilde{\nabla})$ a constant $c_{1}$ (respectively, $\tilde{c}_{1}$) which will be used to 
establish when $\nabla$ and $\tilde{\nabla}$ are  isomorphic.  
If $\nabla$ is of the form (\ref{te-mal-1}),  we define $c_{1}:= -\frac{1}{16 c_{0}}$; if $\nabla$ is of the form
(\ref{te-mal-2}), we define $c_{1} = \frac{1}{16 c_{0}} ( 4 \lambda^{2} + 8\lambda +3)$; if
$\nabla$ is of the form (\ref{te-mal-3}), we define $c_{1}:= \frac{3}{16 c_{0}}$. Finally, if $\nabla$ is 
as in Theorem \ref{te}  i) with $c_{0}\neq 0$,  we define $c_{1} :=0.$  In a similar way, we assign to $\tilde{\nabla}$ a constant
$\tilde{c}_{1}.$ 

\begin{cor} Let $\nabla$, $\tilde{\nabla}$ be two holomorphic normal forms, 
as in Theorem \ref{te} i) or  Theorem \ref{classif-non-ele} 
(and $c_{0}\tilde{c}_{0}\neq 0$ 
when $\nabla$ and $\tilde{\nabla}$ belong to Theorem \ref{te} i)).  
Then $\nabla$ and $\tilde{\nabla}$ are  isomorphic if and only if
the constants $(c, \alpha , c_{0})$ and $(\tilde{c}, \tilde{\alpha}, \tilde{c}_{0})$
involved in their expressions satisfy  
$c =\tilde{c}$, $\alpha =\tilde{\alpha}$, $c_{0} =\epsilon \tilde{c}_{0}$ where $\epsilon \in \{ \pm 1\}$ 
and relations 
(\ref{gen-c01}) and (\ref{gen-cc}), with $c_{1}$ and $\tilde{c}_{1}$ defined above,   are satisfied as well. 
\end{cor}

\begin{proof} We claim that the constant $c_{1}$ associated to $\nabla$ as  above coincides with the $(2,1)$-entry 
$B^{\infty}_{21}$  of the matrix $B_{\infty}$ from the
Malgrange universal deformation $\nabla^{\mathrm{univ}, B_{0}^{o}, B_{\infty}}$ isomorphic to $\nabla $
(and similarly for $\tilde{c}_{1}$ and $\tilde{\nabla}$).
 This is obvious for the $(TE)$-structures from Theorem \ref{te}  i) with $c_{0}\neq 0$ and for
the $(TE)$-structures from Theorem \ref{classif-non-ele} i).
For the $(TE)$-structures from Theorem \ref{classif-non-ele}
ii) the claim follows from
$B_{12}^{\infty} = c_{0}$ and 
$\lambda = B_{21}^{\infty} (b-a)-1$, which imply
$$
4\lambda^{2} + 8\lambda + 3 = 4 (\lambda + 1)^{2} - 1 = 16 B_{12}^{\infty} B_{21}^{\infty}
$$
where we used $a+ b= - \frac{1}{2 B_{21}^{\infty}}$ and
 $ab = -\frac{c_{0}}{B_{21}^{\infty}}$  (see Corollary  \ref{non-ele-case}).
For the $(TE)$-structures from Theorem \ref{classif-non-ele}
iii) the claim follows from the same argument with $\lambda =0$. 
We conclude
the proof using Proposition \ref{classif-Birk}. 
\end{proof}

\section{Which Euler fields on ${\mathcal N}_2$ come from
$(TE)$-structures?}\label{application}

This section has two parts.
In the first part, we determine normal forms for  Euler fields 
on ${\mathcal N}_{2}$. 
It turns out that 
$\mathcal N_{2}$ has  surprisingly many 
Euler fields. 
In the second part, we characterize the Euler fields on 
$\mathcal N_{2}$ which are induced by a $(TE)$-structure. 
At first sight, one might expect that all are induced
by $(TE)$-structures (as in the case of the
$F$-manifolds $I_2(m)$ \cite{DH-dubrovin}).
We will prove that this  is not the case.

\subsection{Normal forms for  Euler fields on $\mathcal N_{2}$}\label{euler-sect}

\begin {thm}\label{classif-euler} i)  Up to an automorphism,  any Euler field on $\mathcal N_{2}$ is of the form 
\begin{align}
\label{expr-euler-1} & E = (t_{1} + c) \partial_{1} +\partial_{2}\\
\label{expr-euler-2}& E = (t_{1}+c) \partial_{1}\\
\label{expr-euler-3}& E = (t_{1}+c) \partial_{1} +{c}_{0} t_{2}\partial_{2},\\
\label{expr-euler-4} & E = (t_{1}+c) \partial_{1} + t_{2}^{r}(1 + c_{1} t_{2}^{r-1})\partial_{2},
\end{align} 
where $c, c_{1}\in \mathbb{C}$,  $c_{0}\in \mathbb{C}^{*}$ and $r\in \mathbb{Z}_{\geq 2}.$\   

ii) Any  two (distinct)  Euler fields from the above list belong to distinct orbits of
the natural action of $\mathrm{Aut} (\mathcal N_{2})$  on the space of Euler fields. 
\end{thm}

We divide the proof  into several steps. 

\begin{lem}\label{euler-1} i) A vector  field on $\mathcal N_{2}$ is an Euler field if and only if 
\begin{equation}\label{form-euler}
E = (t_{1} + c) \partial_{1} + g(t_{2}) \partial_{2},
\end{equation}
for $c\in \mathbb{C}$ and $g\in \mathbb{C} \{ t_{2} \} .$\ 

ii) If $g\neq 0$,  then   $r:= \mathrm{ord}_{0} (g)\geq 0$ is $\mathrm{Aut} (\mathcal N_{2})$-invariant.
If $g\neq 0$ is constant, then up to an automorphism, $E$ is of the form (\ref{expr-euler-1}).\

iii)  If $g=0$ then $E$ is of the form (\ref{expr-euler-2}) and  the $\mathrm{Aut} (\mathcal N_{2})$-orbit of $E$ reduces to $E$.
\end{lem}

\begin{proof} 
i) Let $E:= f  \partial_{1} + g \partial_{2}$ be a vector field on $\mathcal  N_{2}$, where 
$f, g\in\mathbb{C} \{ t_{1}, t_{2}\} .$ A straightforward computation
shows that $L_{E}(\circ ) = \circ$  if and only if $f= t_{1} + c$ (with $c\in \mathbb{C}$)
and $g\in \mathbb{C} \{ t_{2}\}$, i.e.  $E$ is of the form (\ref{form-euler}).\

ii)  Let  $h (t_{1}, t_{2}) =(t_{1}, \lambda (t_{2}))$ be an automorphism of 
$\mathcal N_{2}$,  where $\lambda\in \mathrm{Aut} (\mathbb{C}, 0)$ and $E$ an Euler field given by
(\ref{form-euler}).  Then
\begin{equation}\label{change-iso-euler}
(h_{*}E)_{(t_{1}, t_{2} )} = (t_{1}+c) \partial_{1} + (\dot{\lambda} g) \circ \lambda^{-1} \partial_{2}. 
\end{equation}  
Assume   that $g\neq 0$ and  let $r:= \mathrm{ord}_{0}(g)\in \mathbb{Z}_{\geq 0}$. 
Relation (\ref{change-iso-euler}) and $\lambda \in \mathrm{Aut} (\mathbb{C}, 0)$  implies that  $r$ is an invariant of the $\mathrm{Aut}(\mathcal N_{2})$-action on Euler fields. If  $g = g_{0}$ is a (non-zero) constant,  let 
$h(t_{1}, t_{2}) : =  (t_{1}, g_{0}^{-1} t_{2}).$ Then $h_{*}E =    (t_{1} + c) \partial_{1} +\partial_{2}$
is of the form  (\ref{expr-euler-1}).\

iii) Claim iii) is obvious from (\ref{change-iso-euler}). 

\end{proof}

The next lemma concludes the proof of Theorem \ref{classif-euler} i).

\begin{lem}\label{euler-2} Let $E$ be an Euler field given by (\ref{form-euler}), with $g$ non-constant and 
$r= \mathrm{ord}_{0}(g)\geq 0.$ 

i)  If $r=0$ then, up to an automorphism of $\mathcal N_{2}$, $E$ is of the form (\ref{expr-euler-1}).\

ii) If   $r=1$ then, up to an automorphism of $\mathcal N_{2}$,  $E$ is of the form (\ref{expr-euler-3}).\
 
iii) If $r\geq 2$ then,  up to an automorphism of $\mathcal N_{2}$, $E$ is of the form (\ref{expr-euler-4}).

\end{lem}

\begin{proof} i) From (\ref{change-iso-euler}),    we  need to find $\lambda \in \mathrm{Aut} (\mathbb{C}, 0)$ such that
$(\dot{\lambda} g) (\lambda^{-1}(t)) = 1$, or 
$\dot{\lambda} (t) g(t) =1.$ Writing $\lambda (t) = t \tilde{\lambda }(t)$ with $\tilde{\lambda }\in \mathbb{C} \{ t\}$ a
unit, the problem reduces to showing that the differential equation
\begin{equation}\label{lambda-tilde}
t \dot{\tilde{\lambda}}(t) +\tilde{\lambda}(t) =\frac{1}{g(t)}
\end{equation}
admits a holomorphic solution with $\tilde{\lambda}(0) \neq 0.$  
Equation (\ref{lambda-tilde}) admits a (unique) formal solution
$\tilde{\lambda }$,  which is holomorphic  from Lemma \ref{ecuatia-baza}.  
Moreover, $\tilde{\lambda }(0) = \frac{1}{ g(0)}\in \mathbb{C}^{*}.$ This proves claim i).\

ii) From (\ref{change-iso-euler}),  we  need to show that  there is $\lambda \in \mathrm{Aut} (\mathbb{C}, 0)$ such that 
$(\dot{\lambda} g)(\lambda^{-1}(t)) = c_{0}t$ or 
\begin{equation}\label{ec-ajutatoare}
\dot{\lambda} (t)  g(t ) = {c}_{0}\lambda (t),
\end{equation}
for a suitable  ${c}_{0}\in \mathbb{C}^{*}$. Writing  as before
$\lambda (t) = t\tilde{\lambda }(t)$
and $g(t) =\frac{t}{f(t)}$,  
with $f, \tilde{\lambda}\in \mathbb{C} \{ t\}$ units, 
we obtain that equation (\ref{ec-ajutatoare})  is equivalent to 
\begin{equation}\label{last-t-l}
t\dot{\tilde{\lambda}}(t) + (1-{c}_{0}f(t))  \tilde{\lambda}(t) =0.
\end{equation}
It is easy to  check that (\ref{last-t-l})  has a formal solution  $\tilde{\lambda}=\sum_{n\geq 0} \tilde{\lambda}_{n}t^{n}$ 
(unique, when $\tilde{\lambda}_{0}$ is given)
if and only if  $c_{0} = \frac{1}{f(0)}.$ Define  $c_{0} := \frac{1}{f(0)}$ and let $\tilde{\lambda}$ be a formal solution
of (\ref{last-t-l}) with $\tilde{\lambda}_{0}\neq 0.$ 
 From Lemma \ref{ecuatia-baza}, $\tilde{\lambda}$ is holomorphic. 
Let $h(t_{1}, t_{2}) := (t_{1},t_{2} \tilde{\lambda}  (t_{2}))$.  Then  $h_{*}E$  
is of the form (\ref{expr-euler-3}).\

iii)  From (\ref{change-iso-euler}), we need to find $\lambda \in \mathrm{Aut} (\mathbb{C}, 0)$ such that 
$(\dot{\lambda} g)( \lambda^{-1} (t)) = t^{r} ( 1 + c_{1} t^{r-1})$ or
\begin{equation}\label{l-change}
(\dot{\lambda } g )(t) = \lambda (t)^{r} ( 1 + c_{1} \lambda (t)^{r-1}),
\end{equation}
for a suitably chosen $c_{1}\in \mathbb{C}.$ Writing $g(t) = t^{r} f(t)$,
$\lambda (t) = t \tilde{\lambda }(t)$ with  
$f, \tilde{\lambda }\in \mathbb{C} \{ t\}$ units, 
and $\tau (t) = (1-r) \tilde{\lambda}(t)^{r-1}$,   
equation (\ref{l-change}) becomes
\begin{equation}\label{tau-last-add}
t \dot{\tau} (t) + (r-1) \tau (t) = -\frac{\tau (t)^{2}}{f}  ( 1 +\frac{ c_{1}}{1-r} t^{r-1} \tau (t)). 
\end{equation}
This is an equation,   in the unknown function $\tau$, of type (\ref{study-eqn}).
Lemma \ref{de-adaugat} concludes   claim iii).  
\end{proof}

It remains to prove
Theorem \ref{classif-euler} ii). From (\ref{change-iso-euler}), 
the constant $c$ from the Euler fields of Theorem \ref{classif-euler} i)  is $\mathrm{Aut} (\mathcal N_{2})$-invariant. 
From Lemma \ref{euler-1}  we deduce that any two  distinct Euler fields $E$ and $\tilde{E}$
from Theorem \ref{classif-euler} i), which
belong to the same orbit  of $\mathrm{Aut} (\mathcal N_{2})$, are necessarily either both of the  form 
(\ref{expr-euler-3})  (with the same constant $c$ and distinct constants
$c_{0}, \tilde{c}_{0}$) or both of  the  form 
(\ref{expr-euler-4})  (with the same constant $c$ and distinct constants $c_{1}$, $\tilde{c}_{1}$). But these cases cannot hold: assume, e.g. that
$E$ and $\tilde{E}$ are of the form (\ref{expr-euler-4}). 
Then there is a solution $\tau$, with $\tau_{0} \neq 0$,  of the equation (\ref{tau-last-add}) 
with $f:= 1+ \tilde{c}_{1} t^{r-1}$. From the uniqueness of the constant $c$ in Lemma \ref{de-adaugat}  
we obtain that $c_{1} = \tilde{c}_{1}.$ The other case can be treated similarly.

\subsection{Euler fields and $(TE)$-structures}\label{applic}

For a $(TE)$-structure in pre-normal form, determined by $(f, b_{2}, c, \alpha )$, 
the induced Euler field on $\mathcal N_{2}$  is given by
\begin{equation}\label{expression-euler-B}
E = ( t_{1} - c ) \partial_{1} - b_{2}^{(0)} \partial_{2}
\end{equation}
(see Theorem \ref{t2.4}). 
Recall also that any isomorphism 
$f : (M_{1}, \circ_{1}, e_{1}, E_{1}) \rightarrow (M_{2}, \circ_{2}, e_{2}, E_{2})$
between $F$-manifolds with Euler fields 
 defines  (by the pull-back
$(\mathrm{Id}\times f)^{*}$) an isomorphism between the spaces of $(TE)$-structures over
$(M_{2}, \circ_{2}, e_{2})$ and  $(M_{1}, \circ_{1}, e_{1})$, which induce $E_{2}$ and $E_{1}$ respectively.
Using these facts, we obtain:

\begin{prop}  An Euler field on $\mathcal N_{2}$  is induced by a $(TE)$-structure if and only if its normal form
 is of the type (\ref{expr-euler-1}), (\ref{expr-euler-2}), (\ref{expr-euler-3}) or
of the type (\ref{expr-euler-4}) with $r=2$ and $c_1=0$. 
\end{prop}

\begin{proof}   The Euler fields  induced by the holomorphic normal forms are 
given by  $E= ( t_{1} - c) \partial_{1} + g(t_{2}) \partial_{2}$,  where 
$c\in \mathbb{C}$  and $g=0$, $g=-1$,
$g = - t_{2}^{2}$ 
or $g$ is non-constant with $\mathrm{ord}_{0}(g)\in \{ 0, 1\}$
(from (\ref{expression-euler-B}) and the explicit expression of the holomorphic normal forms).
Such  Euler fields 
have the normal forms 
(\ref{expr-euler-1}), (\ref{expr-euler-2}), (\ref{expr-euler-3}) or (\ref{expr-euler-4}) with $r =2$ and $c_1=0$. 
\end{proof}

\begin{rem}{\rm The Euler fields 
in normal form  (\ref{expr-euler-1}), (\ref{expr-euler-2}) and  (\ref{expr-euler-3})
are the Euler fields of a Frobenius manifold with underlying $F$-manifold $\mathcal N_{2}$ 
(we may choose a  constant 
metric $g = (g_{ij})$ with $g_{11} = g_{22}=0$ in the standard coordinate system $(t_{1}, t_{2})$ for $\mathcal N_{2}$, as a Frobenius metric). The Euler fields in normal form
(\ref{expr-euler-4}) with $r=2$ and $c_1=0$ are not Euler
fields of a Frobenius manifold. The reason is that in the
case of a Frobenius manifold, there is a $(TE)$-structure
which induces the $F$-manifold with Euler field and which
extends to a trivial bundle on ${\mathbb P}^1\times M$ with 
a logarithmic pole along $\{0\}\times M$, i.e. which 
can be brought into a Birkhoff normal form.
But the only $(TE)$-structures over ${\mathcal N}_2$
with an Euler field in normal form (\ref{expr-euler-4})
with $r=2$ and $c_1=0$ are the $(TE)$-structures in the 
second normal form in Theorem \ref{te} iii). 
At the point $t_1=t_2=0$ and $c=0$, the matrix $B$
takes the form $B=\alpha z C_1-z\frac{1}{2}D-z^2E$.
Example 5.5 in \cite{DH-dubrovin} shows that this connection
cannot be brought into Birkhoff normal form. }
\end{rem}

\section{Appendix}

\subsection{Differential equations}\label{diff-eqn-sect}

Along this section $t\in (\mathbb{C},0)$ is the standard coordinate.   
We shall use repeatedly  the following well-known  lemma 
(see  e.g. \cite{wasow} and the proof of Lemma 12 of \cite{DH-AMPA}).

\begin{lem}\label{ecuatia-baza} Any
formal solution $u\in \mathbb{C} [[t]]$ of a differential equation of the form
\begin{equation}\label{ec-w}
t\dot{ u}(t) + A(t) u(t) = b(t),
\end{equation}
where $A: ( \mathbb{C}, 0) \rightarrow M_{n}(\mathbb{C})$ and $b: (\mathbb{C} , 0) \rightarrow \mathbb{C}^{n}$
are holomorphic, is holomorphic. 
\end{lem}

The next class of inequalities will be used in 
Lemma \ref{de-adaugat} below.

\begin{lem}\label{her}
Let $C:= 4\sum_{n =1}^{\infty} (\frac{1}{n})^{2} = \frac{2}{3} \pi^{2}$.    For
any $b, l\in \mathbb{Z}_{\geq 2}$ with  $b\geq l$, 
\begin{equation}\label{ineq}
\sum_{a_{i}:(*)_{l,b}} (a_1\cdots  a_l)^{-2}\leq C^{l-1} b^{-2},
\end{equation}
where the condition $(*)_{l,b}$ 
on $(a_{1}, \cdots , a_{l})$ 
means  $a_{i}\in \mathbb{Z}_{\geq 1}$
(for any $1\leq i\leq l$)  and   $\sum_{i=1}^{l}a_i =b$.
\end{lem}

\begin{proof} We prove (\ref{ineq})  by induction on $l$.  
Consider first  $l=2$. 
\begin{align*}
&\sum_{a_1,a_2:(*)_{2,b}}(a_1a_2)^{-2} 
= \sum_{a=1}^{b-1} \left( a^{-1} (b-a)^{-1}\right)^{2}=
\sum_{a=1}^{b-1}\left((a^{-1}+(b-a)^{-1})b^{-1}\right)^2\\
&= b^{-2} \sum_{a=1}^{b-1}\left(a^{-1}+(b-a)^{-1}\right)^2 
\leq 2 b^{-2} \sum_{a=1}^{b-1} \left( a^{-2} + (b-a)^{-2} \right) \leq  b^{-2} C.
\end{align*}
Suppose that (\ref{ineq}) holds for any $l\leq n-1$. Using (\ref{ineq}) for $l=2$ and $l=n-1$,  
\begin{align*}
&\sum_{a_1,...,a_n:(*)_{n,b}}(a_1\cdots  a_n)^{-2} =   \sum_{b_1,b_2:(*)_{2,b}} \sum_{a_1,\cdots ,a_{n-1}:(*)_{n-1,b_{1}}} 
(a_1 \cdots  a_{n-1}\cdot b_2)^{-2}\\
&\leq  \sum_{b_1,b_2:(*)_{2,b}} C^{n-2}(b_1b_2)^{-2}\leq  C^{n-1} b^{-2},
\end{align*}
i.e. (\ref{ineq})  holds for $l=n$ as well.\end{proof}

\begin{lem}\label{de-adaugat}
Let $f\in \mathbb{C} \{t\}$ be a unit and $r\in \mathbb{Z}_{\geq 1}.$ 
There is a unique $c\in \mathbb{C}$ such that  the  differential equation
\begin{equation}\label{study-eqn}
t \dot{\tau} (t) + r\tau (t) = \tau (t)^{2} f(t) ( 1 + c t^{r} \tau (t))
\end{equation} 
admits  a formal solution $\tau =\sum_{n\geq 0} \tau_{n}t^{n}\in \mathbb{C} [[t]]$ with $\tau_{0}\neq 0.$ Any such solution
$\tau$ is  holomorphic. It is uniquely determined by $\tau_{r}\in \mathbb{C}$, which can be  chosen arbitrarily.   
\end{lem}

\begin{proof}
We write $f = \sum_{n\geq 0} f_{n} t^{n}$. Identifying the coefficients in (\ref{study-eqn}) 
(and using that $\tau_{0} \neq 0$) we obtain that
$\tau_{n}$, for $n\in \{ 1, \cdots ,  r\}$,  are determined inductively by
\begin{equation}
\tau_{0} = \frac{r}{f_{0}},\ \tau_{n} = \frac{1}{n-r} \sum_{j+ k+p=n;\   k,p\leq n-1} f_{j}\tau_{k}\tau_{p},\ \forall
n< r.
\end{equation}
Identifying the coefficients of $t^{r}$ 
in (\ref{study-eqn}) 
we obtain that $c$ is determined by 
\begin{equation}\label{def-c}
c= -\frac{1}{  \tau_{0}^{3} f_{0}} \sum_{j+k+p= r;\  k,p\leq r-1} f_{j} \tau_{k}\tau_{p}.
\end{equation}
It follows that there is a unique $c\in \mathbb{C}$, namely the one defined by (\ref{def-c}),  for which
(\ref{study-eqn}) admits  a formal solution $\tau$  with $\tau_{0}\neq 0$: the coefficient $\tau_{r}$
of $\tau$  can be chosen arbitrarily and
the remaining coefficients $\tau_{n}$,  for $n\geq r+1$,  
are  determined inductively by
\begin{equation}\label{tau-n-rem}
\tau_{n} = \frac{1}{n-r} \left( \sum_{j+ k+p = n;\  k,p\leq n-1} f_{j}\tau_{k} \tau_{p} + c \sum_{j+k+p+s= n-r} 
\tau_{j}\tau_{k}\tau_{p} f_{s}\right) .
\end{equation}
It remains to prove  that $\tau$ is holomorphic. 
Since $f$ is holomorphic, there is $M>0$ and $\tilde{r}>0$ such that
\begin{equation}\label{ev-f}
| f_{n}| \leq \frac{ M \tilde{r}^{n}}{ (n+1)^{2}},\ \forall n\geq 0.
\end{equation}
The above relation for $n=0$ implies that $M \geq | f_{0}|.$
We further assume that $M\geq 1.$ 
 We claim that for a suitable choice of  $M$ and $\tilde{r}$ satisfying relations (\ref{ev-f}), 
the coefficients $\tau_{n}$ of $\tau$ satisfy 
\begin{equation}\label{ev-tau-prima}
| \tau_{n} |  \leq \frac{M^{n+1} \tilde{r}^{n}}{(n+1)^{2}},\ \forall n\geq 0. 
\end{equation}  
Remark that for $n=0$ relation (\ref{ev-tau-prima}) is equivalent to $M\geq  |\tau_{0} | = \frac{r}{ | f_{0}|}.$

To prove the claim,  let  $n\geq r+1$ be fixed. We   assume that (\ref{ev-tau-prima}) holds  for all 
$\tau_{0}, \cdots , \tau_{n-1}$ and  we study when it holds 
for $\tau_{n}.$
For this, we evaluate, using (\ref{tau-n-rem}),  
\begin{align}
\nonumber& |\tau_{n}| \leq 
\frac{1}{n-r} \left( \sum_{j+ k+p=n;\ k, p\leq n-1}| f_{j}| |\tau_{k} | |\tau_{p} |  + |c| \sum_{j+k+p+s= n-r}  |\tau_{j}| |\tau_{k} | | \tau_{p} | | f_{s} |\right)\\
\nonumber& \leq \frac{ \tilde{r}^{n}}{n-r} \sum_{j+ k+p=n; \ k, p\leq n-1} M^{n-j +3} ( j+1)^{-2} (k+1)^{-2} (p+1)^{-2} \\
\nonumber&+ \frac{| c| \tilde{r}^{n-r} }{n-r} \sum_{j+k+p+s= n-r}  M^{n-r-s +4}  (j+1)^{-2} ( k+1)^{-2} (p+1)^{-2} (s+1)^{-2}.
\end{align}
Since $M\geq 1$, $M^{n-j+3}\leq M^{n+3}$ and $M^{n-r-s +4}\leq M^{n-r +4}.$ From Lemma 
\ref{her}, we  obtain that 
\begin{equation}
|\tau_{n} | \leq \frac{1}{n-r} \left( \tilde{r}^{n} M^{n+3}C^{2}(n+3)^{-2} + |c| M^{n-r+4} \tilde{r}^{n-r}C^{3} (n-r+4)^{-2}\right) .
\end{equation}
We deduce that a sufficient condition for (\ref{ev-tau-prima}) to hold also for $\tau_{n}$ 
is that 
\begin{equation}\label{condition-0-added}
M^{2} + \frac{C|c| M^{3-r}}{ \tilde{r}^{r}} \left( \frac{n+3}{n-r +4}\right)^{2} \leq \frac{1}{C^{2}} ( n-r) \left( \frac{n+3}{n+1}\right)^{2}.
\end{equation} 
Consider now $\epsilon_{0} > 0$ small,
$M_{0}\geq \mathrm{max} \{ \frac{r}{ |f_{0}| }, M\}$  and $n_{0} > r$ such that 
\begin{equation}\label{condition-1}
M_{0}^{2} \leq \frac{1}{C^{2}}  ( n-r)  \left( \frac{n+3}{n+1}\right)^{2}-\epsilon_{0} ,\forall n\geq n_{0}.
\end{equation}
(This is possible since the right hand side of (\ref{condition-1}) tends to $+\infty$ for $n\rightarrow +\infty$). 
With this choice of $(M_{0}, n_{0}, \epsilon_{0} )$, we choose $\tilde{r}_{0}\geq\tilde{r} $ such that
\begin{equation}\label{condition-2}
\tilde{r}_{0}^{r} \geq \frac{ C| c| M_{0}^{3-r}}{\epsilon_{0}} \left( \frac{n+3}{n-r+4}\right)^{2},\ \forall n\geq n_{0}.
\end{equation}
(This is possible since the right hand side of (\ref{condition-2}) is bounded when $n\rightarrow +\infty$). 
Relations (\ref{condition-1}) and (\ref{condition-2}) imply 
\begin{equation}\label{condition-0}
M_{0}^{2} + \frac{C|c| M_{0}^{3-r}}{ \tilde{r}_{0}^{r}} \left( \frac{n+3}{n-r +4}\right)^{2} \leq \frac{1}{C^{2}} ( n-r) \left( \frac{n+3}{n+1}\right)^{2},\ \forall n\geq n_{0},
\end{equation} 
i.e. relation (\ref{condition-0-added}) (with $M$ and $\tilde{r}$ replaced by $M_{0}$ and $\tilde{r}_{0}$ respectively) holds.   

The above argument shows that the inequalities 
\begin{equation}\label{ev-tau-m0}
| \tau_{n} |  \leq \frac{M_{0}^{n+1} \tilde{r}_{0}^{n}}{(n+1)^{2}},\ \forall n\geq 0
\end{equation}  
hold  if they hold for any $n\leq n_{0} -1$ and
\begin{equation}\label{ev-f-1}
| f_{n}| \leq \frac{ M_{0} \tilde{r}_{0}^{n}}{ (n+1)^{2}},\ \forall n\geq 0.
\end{equation}
But (\ref{ev-f-1}) is obviously true from 
(\ref{ev-f}), since $M_{0}\geq M$ and $\tilde{r}_{0}\geq \tilde{r}.$ 
Relations (\ref{ev-tau-m0}) for $n\leq n_{0}-1$ are satisfied as well,  by imposing to $\tilde{r}_{0}$ 
(which can be chosen as large as needed)
the additional conditions
\begin{equation}
\tilde{r}_{0}^{n}\geq \frac{(n+ 1)^{2} | \tau_{n} | }{ M_{0}^{n+1}},\ \forall  0\leq n\leq n_{0} -1.
\end{equation}
From (\ref{ev-tau-m0}), $\tau \in \mathbb{C} \{ t\} .$ 
\end{proof}

The following lemma  will be used in our  formal  classification of $(TE)$-structures.  Its proof is straightforward and will be omitted.

\begin{lem}\label{third-der} Consider the system of equations 
\begin{equation}\label{trei}
m x + \dot{b} x - b \dot{x} = g,\ x^{(3)} =0
\end{equation}
in the unknown function $x\in \mathbb{C} \{ t\}$, 
where $m\in \mathbb{C}^{*}$ and $g = g_{2}t^{2} + g_{1}t + g_{0}$, for $g_{i}\in
\mathbb{C}.$\

i) Assume that $b(t) = \lambda t$, for $\lambda \in \mathbb{C}^{*}$. If $m\notin\{ \pm\lambda \}$ then there is a 
unique solution
$x$ of (\ref{trei}).  
If $m=\lambda$, then  there is a solution of (\ref{trei}) if and only if $g_{2}=0.$ If $m = -\lambda$,
then  there is a solution
of (\ref{trei}) if and only if $g_{0} =0.$

ii) Assume that $b(t) = \lambda t +1$, for $\lambda \in \mathbb{C}.$   If $m\notin\{ \pm\lambda \}$, then there is  a 
unique solution
$x$ of (\ref{trei}).  
If $m=\lambda$, then there is a solution of (\ref{trei}) if and only if $g_{2}=0.$  If $m = -\lambda$, then
there is a solution
of (\ref{trei}) if and only if $m^{2}g_{0} + mg_{1} + g_{2}=0.$\

iii) Assume that $b(t) = t^{2}.$ Then there is a unique solution of (\ref{trei}).   
\end{lem}

\subsection{The Fuchs criterion}

For the definition and properties of meromorphic connections with regular singularities,
see e.g. \cite{sabbah}, Chapter II. Consider a meromorphic connection $\nabla$ on the germ 
$\mathcal M = \textit{\textbf k}^{d}$ of the meromorphic rank $d$ trivial vector bundle over $(\mathbb{C}, 0)$ , with pole at the origin only. Its sections are germs at the origin of  vector valued  functions $(f_{1}, \cdots , f_{d})$, where each
$f_{i}$ is meromorphic, with pole at the origin only.
The Fuchs criterion is an effective way to check if $\nabla$ has a regular singularity at the origin. Namely, one considers  a cyclic vector, i.e. a section $v_{0}$ such that
$\{ v_{0},  v_{1}:= \nabla_{\partial_{ z}} (v_{0}), \cdots ,  v_{d-1}:= \nabla^{d-1}_{\partial_{ z}} (v_{0}) \}$ is a 
basis  of $\mathcal M$ (such a vector always exists).  In this basis, $\nabla$ has the expression
\begin{align}
\nabla_{\partial_{ z}} (v_{i}) & = v_{i+1},\ 0\leq i\leq d-2\\
\nabla_{\partial_{ z}} (v_{d-1}) &= a_{0} v_{0} + \cdots + a_{d-1} v_{d-1},
\end{align}
for some $a_{i}\in \textit{\textbf{k}}$. We denote by $v(f)$ the valuation of a function $f\in \textit{\textbf{k}}$, i.e. the unique integer such that $f(z) = z^{v(f)} h(z)$, where 
$h\in \mathbb{C} \{ z\}$ and $h(0) \neq 0.$  The Fuchs criterion is stated as follows (see \cite{mal-book}):

\begin{thm} \label{fuchs} The connection $\nabla$ has a regular singularity at the origin if and only if 
$v(a_{i}) \geq i-d$, for any $0\leq i\leq d-1.$
\end{thm}

\subsection{ Irreducible bundles and Birkhoff normal form}
 
Consider a meromorphic connection $\nabla$ on the germ 
$E = ({\mathcal O}_{(\mathbb{C},0)})^{d}$ of the  holomorphic rank $d$  trivial vector bundle over $(\mathbb{C}, 0)$,  with pole of order $r\geq 0$
at the origin.  In the standard basis  of $E$, the connection form of $\nabla$ is given by  $A(z  ) dz $, where 
$A\in M_{d\times d} (\textit{\textbf k}) $ is such that  $z^{r+1} A(z)$ is holomorphic.  We say that $(E,\nabla )$ can be put in  Birkhoff normal form if there is a holomorphic isomorphism  $T\in M_{d\times d} ( {\mathcal O}_{(\mathbb{C}, 0)})$ 
such that the image $T\cdot \nabla$ of $\nabla$ by $T$ has connection form 
$$
\Omega :=  z^{-(r+1)}  \left( B_{0} z^{0} + \cdots + z^{r} B_{r}\right) dz 
$$
where $B_{i}$ are constant matrices. 

The irreducibility criterion (see \cite{AB} and \cite{sabbah}, Chapter IV)  provides a sufficient condition
for $(E, \nabla )$ as above to be put in the Birkhoff normal form. 
Let $(\mathcal M  = \textit{\textbf{ k}}^{d} , \nabla )$ be the germ of the meromorphic bundle with connection, with singularity at the origin only,
for which $(E, \nabla )$ is a lattice. From the Riemann-Hilbert correspondence (see e.g. \cite{sabbah}, page 99), 
there is a unique (up to isomorphism) meromorphic bundle with connection $(\tilde{\mathcal M}, \tilde{\nabla })$ on $\mathbb{P}^{1}$, with poles at $0$ and $\infty$,  
whose germ at $0$  is isomorphic to  $(\mathcal M, \nabla )$
and such that $\infty$ is a regular singularity for $\tilde{\nabla}.$ 
 We say that $(\tilde{\mathcal M}, \tilde{\nabla })$ is irreducible if there is no proper meromorphic subbundle $\mathcal N\rightarrow \mathbb{P}^{1}$
of $\tilde{\mathcal M}$, which is preserved by $\tilde{\nabla}$, i.e. $\tilde{\nabla} (\mathcal N ) \subset \Omega^{1}_{\mathbb{P}^{1}}\otimes \mathcal N .$ 
The irreducibility criterion states that if $(\tilde{\mathcal M}, \tilde{\nabla })$ is irreducible, then $(E, \nabla )$ can be put in  Birkhoff normal form: one extends the lattice $E$ around the origin to a globally defined lattice $\tilde{E}$ of $(\tilde{\mathcal M}, \tilde{\nabla})$, logarithmic at $\infty$,  and  applies  Corollary 2.6 of  \cite{sabbah} (page 154)  to obtain
a new lattice $\tilde{E}^{\prime}$ of  $\tilde{\mathcal M}$, which also extends $E$, is logarithmic at $\infty$ and is trivial as a holomorphic vector bundle.
A base change between the standard basis of $E$ and a basis of $\tilde{E}^{\prime}$  in a neighborhood of 
$0\in \mathbb{C}$   gives 
the  holomorphic isomorphism $T$ above.  In particular,  for rank $2$ bundles we can state: 

\begin{lem}\label{lem-ired}  Assume that $E$ is of rank two and let $\{ v_{1}, v_{2}\}$ be its standard
basis. 
If there is no non-zero $w = gv_{1} + fv_{2}$, with $f, g\in\textit{\textbf k}$, such that  
$\nabla_{\partial_{z}}(w)= hw$ for a function 
$h\in \textit{\textbf k}$,  then $\nabla$ can be put in Birkhoff normal form.
\end{lem}

\begin{proof} In the above notation, we claim that  $(\tilde{\mathcal M}, \tilde{\nabla })$  is irreducible: if it were reducible, then 
a basis of $\mathcal N$ around the origin would  provide a section $w$ as in the statement of the lemma. We obtain
a contradiction.
\end{proof}

\subsection{ Malgrange universal connections}
Let $(H\rightarrow \mathbb{C}\times (M,0), \nabla )$ be a $(TE)$-structure with unfolding condition
over a germ $((M, 0),\circ , e, E)$ of an $F$-manifold with Euler field. 
Assume that the restriction $\nabla^{0}$ of $\nabla$ to the slice at the origin
$\mathbb{C} \times \{ 0\}$ 
can be put in  Birkhoff normal form.
Let $\vec{v}_{0}$  be a basis  
of $H\vert_{(\mathbb{C}, 0) \times \{ 0\}}$ in which the connection form of $\nabla^{\mathrm{0}}$
is given by
\begin{equation}\label{omega-0}
\Omega^{0}=(\frac{B_{0}^{o}}{z } + B_{\infty})
\frac{d z}{z },
\end{equation}
where $B^{o}_{0}, B_{\infty}\in M_{n\times n}(\mathbb{C})$ (and $n= \mathrm{rank}(H) =\mathrm{dim}(M)$). 
If $B_{0}^{o}$ is a regular matrix 
(i.e. distinct Jordan blocks in its Jordan normal form have distinct eigenvalues, or the vector space of matrices which commute with
$B_{0}^{o}$ has dimension $n$, with basis  $\{ \mathrm{Id}, B_{0}^{o},\cdots , (B_{0}^{o})^{n-1}\}$), then  
$\nabla^{0}$  
has a universal deformation
$\nabla^{\mathrm{univ}}:= \nabla^{\mathrm{univ}, B_{0}^{o}, B_{\infty}}$. In particular, 
$\nabla^{\mathrm{univ}}$  is isomorphic to the given $(TE)$-structure
$\nabla$ and so are the parameter spaces 
of  $\nabla^{\mathrm{univ}}$ and $\nabla$ 
(as $F$-manifolds with Euler fields). 
The universal deformation  $\nabla^{\mathrm{univ}}$  was constructed by Malgrange in \cite{mal1,mal2}
(see also \cite{sabbah}, Chapter
VI, Section 3.a; see e.g. \cite{sabbah}, page 199, for the definition of the universal deformation).
We now recall its definition.
Let $\mathcal
D\subset T M_{n\times n}(\mathbb{C})$ be defined by
\begin{equation}\label{rond}
{\mathcal D}_{\Gamma}:= \mathrm{Span}_{\mathbb{C}} \{ \mathrm{Id},
(B_{0})_{\Gamma}, \cdots ,(B_{0})_{\Gamma }^{n-1}\}\subset
T_{\Gamma}M_{n\times n}({\mathbb{C}}) = M_{n\times n}({\mathbb{C}}),
\end{equation}
where
\begin{equation}\label{B-0}
(B_{0})_{\Gamma }:= B^{o}_{0}-\Gamma  +[B_{\infty}, \Gamma ].
\end{equation}
Because $B^{o}_{0}$ is regular, so is $(B_{0})_{\Gamma}$, for any
$\Gamma\in W$, where $W$ is a small open neighborhood of  $0$ in
$M_{n}({\mathbb{C}})$. For any $\Gamma \in W$,  $\mathcal
D_{\Gamma}$ is the ($n$-dimensional) vector space of polynomials
in $(B_{0})_{\Gamma}$  and the distribution  $\mathcal
D\rightarrow W$ is integrable.  The parameter space $M^{\mathrm{univ}}$ of $\nabla^{\mathrm{univ}}$ 
 is the maximal integral submanifold of $\mathcal
D\vert_{W}$, passing through $0\in M_{n\times n}(\mathbb{C})$
(the trivial matrix). Let $\circ_{\mathrm{univ}}$ be the multiplication on  $TM^{\mathrm{univ}}$, 
which, on any tangent space $T_{\Gamma} M^{\mathrm{univ}}= \mathcal D_{\Gamma}$,  
is given by multiplication of matrices. It has  unit field $e_{\mathrm{univ}}:= \mathrm{Id}$ (i.e. $(e_{\mathrm{univ}})_{\Gamma}
:= \mathrm{Id}\in \mathcal D_{\Gamma}$, for any $\Gamma \in M^{\mathrm{univ}}$).
Let $E_{\mathrm{univ}}$ be the vector field on $M^{\mathrm{univ}}$  
defined by  $E_{\mathrm{univ}}: = - B_{0}$
(i.e. $E_{\Gamma}:= - (B_{0})_{\Gamma}$,  for any $\Gamma \in M^{\mathrm{univ}}$).  
Then $(M^{\mathrm{univ}}, \circ_{\mathrm{univ}}, e_{\mathrm{univ}}, E_{\mathrm{univ}})$ is a regular $F$-manifold
(see Definition 2 of \cite{DH}). The germ 
$((M^{\mathrm{univ}},0) ,\circ_{\mathrm{univ}}, e_{\mathrm{univ}}, E_{\mathrm{univ}})$  is
universal in the following sense: it is  the unique (up to isomorphism) germ of $F$-manifold  with Euler field
$((M,0), \circ , e, E)$ for which the endomorphism
$\mathcal U (X):= E\circ X$ of $T_{0} M$ has the same conjugacy class as $B_{0}^{o}$  (see \cite{DH}).
Moreover,  if $B_{0}^{o}$ has a unique eigenvalue (or a unique Jordan block)  
then the matrix  $(B_{0})_{\Gamma}$,  for any $\Gamma \in M^{\mathrm{univ}}$, has this property as well 
(see Proposition 15 of  \cite{DH}).\

The universal deformation  $\nabla^{\mathrm{univ}}$ of $\nabla^{0}$ is
defined on the trivial bundle $E = (\mathbb{C} \times M^{\mathrm{univ}}) \times  \mathbb{C}^{n} \rightarrow \mathbb{C}\times M^{\mathrm{univ}}$. Its  connection form in the standard trivialization of $E$ is given
by
\begin{equation}\label{omega-cann}
\Omega^{\mathrm{univ}} = \left( \frac{B_{0}}{z} +
B_{\infty}\right) \frac{dz}{z } +\frac{\mathcal C}{z}.
\end{equation}
Here $B_{0} : M^{\mathrm{univ}}\rightarrow M_{n\times n}(\mathbb{C})$,
$(B_{0})(\Gamma) : = (B_{0})_{\Gamma}$ is given by (\ref{B-0}) and
${\mathcal C}_{X}:= X$ is the action of the matrix $X$ on
$\mathbb{C}^{n}$, for any $X\in T_{\Gamma}M^{\mathrm{can}}\subset
M_{n\times n}(\mathbb{C})$.\\

{\bf Acknowledgements.}  L.D.\ was supported by a grant of the Ministery of Research and Innovation, project no PN-III-ID-P4-PCE-2016-0019 within PNCDI. Part of this work was done during her visit at the University of Mannheim in October 2017.
She thanks University of Mannheim for hospitality and excelent working conditions.

{\it Liana David}:  Institute of Mathematics 'Simion Stoilow' of the
Romanian Academy, Research Unit 4, Calea Grivitei nr. 21,
Bucharest, Romania; liana.david@imar.ro\\

{\it Claus Hertling}: Lehrstuhl f\"{u}r Algebraische Geometrie, 
Universit\"{a}t  Mannheim,
B6, 26, 68131, Mannheim, Germany;
hertling@math.uni-mannheim.de

\end{document}